\documentclass[11pt,a4paper]{article}
\usepackage{graphicx}
\usepackage{amsmath,amsthm,amsfonts} 
\usepackage{hyperref}
\usepackage{smallcaptions}

\newcommand\tetr{\mathcal{T}}
\newcommand\bonds{\mathrm{Bo}}
\newcommand\circu{\mathrm{circ}}
\newcommand\Span{\mathrm{span}}
\newcommand\interior{\mathrm{int}}
\newcommand\clust{\mathrm{Cl}}
\newcommand\cyl{T}

\newcommand\conv{\mathop{\mathrm{conv}}}
\newcommand\psiC{\psi_\C}

\newcommand{\discxi}{{{\zeta}}}
\newcommand{\discbeta}{{{\xi}}}
\newcommand{\contbeta}{{{\beta}}}
\newcommand{\conteta}{{{\eta}}}

\newcommand{\burgersbi}[1]{{b_{#1}}}

\newcommand\Ce{\mathrm{Core}_\eps}

\renewcommand\ln{\log}
\newcommand\pl{\mathrm{pl}}
\newcommand\el{\mathrm{el}}

\newcommand\contbetapl{\contbeta^\pl}

\newcommand\discbetapl{\discbeta^\pl}

\newcommand\discbetael{\discbeta^\el}
\newcommand\R{\mathbb{R}}
\newcommand\C{\mathbb{C}}
\newcommand\Z{\mathbb{Z}}
\newcommand\N{\mathbb{N}}

\newcommand\calT{\mathcal{T}}
\newcommand\calH{\mathcal{H}}

\newcommand\calB{\mathcal{B}}

\newcommand\calM{\mathcal{M}}
\newcommand\calN{\mathcal{N}}

\newcommand\calL{\mathcal{L}}
\newcommand\calA{\mathcal{A}}
\newcommand\calC{\mathcal{C}}

\newcommand\weakstarto{{\displaystyle\mathop{\rightharpoonup}^*}}
\newcommand\weakto{\mathop{\rightharpoonup}}
\newcommand\e{\varepsilon}

\newcommand\supp{\mathop{\mathrm{supp}}}

\newcommand\Div{\mathop{\mathrm{div}}}
\newcommand\Curl{\mathop{\mathrm{curl}}}
\newcommand\curl{\mathop{\mathrm{curl}}}
\newcommand\dist{{\mathrm{dist}}}
\newcommand\rel{{\mathrm{rel}}}

\newcommand\Id{{\mathrm{Id}}}
\newcommand\SO{{\mathrm{SO}}}
\newcommand\sym{{\mathrm{sym}}}
\newcommand\skw{{\mathrm{skew}}}

\newcommand\loc{{\mathrm{loc}}}
\newcommand\eps{{\varepsilon}}
\newcommand\Ade{{\calA_\eps^{\mathrm d}}}
\newcommand\Adek{{\calA_{\eps_k}^{\mathrm d}}}

\newcommand\diam{\mathop{\mathrm{diam}}}
\newcommand{\LM}[1]{\hbox{\vrule width.2pt \vbox to#1pt{\vfill \hrule width#1pt height.2pt}}}
\newcommand{\vertici}{\mathop{\mathrm{vert}}}
\newcommand{\edges}{\mathop{\mathrm{edges}}}

\newcommand\EC{E_\calC}
\newcommand{\LL}{{\mathchoice
{\,\LM7\,}{\,\LM7\,}{\,\LM5\,}{\,\LM{3.35}\,}}}
\newtheorem{theorem}{Theorem}[section]
\newtheorem{lemma}[theorem]{Lemma}
\newtheorem{definition}[theorem]{Definition}
\newtheorem{remark}[theorem]{Remark}

\newtheorem{proposition}[theorem]{Proposition}
\newtheorem{example}[theorem]{Example}
\numberwithin{equation}{section}

\newcommand\ccorerad{C_\mathrm{core}}  \renewcommand\ccorerad{k_*}  
\begin{document}
\begin{center}
    {\LARGE  A discrete crystal model in three dimensions:\\
    the line-tension limit for dislocations} \\[5mm]
{\today}\\[5mm]
Sergio Conti$^{1}$, Adriana Garroni$^{2}$, and Michael Ortiz$^{1,3}$
\\[2mm]
{\em $^1$ Institut f\"ur Angewandte Mathematik,
Universit\"at Bonn\\ 53115 Bonn, Germany }\\
{\em $^{2}$ Dipartimento di Matematica, Sapienza, Universit\`a di Roma\\
00185 Roma, Italy}\\
 {\em $^{3}$
        Division of Engineering and Applied Science,
        California Institute of Technology \\
        Pasadena, CA 91125, USA
}\\[4mm]
\begin{minipage}[c]{0.8\textwidth}
We propose a discrete lattice model of the energy of dislocations in three-dimensional crystals which properly accounts for lattice symmetry and geometry, arbitrary harmonic interatomic interactions, elastic deformations and discrete crystallographic slip on the full complement of slip systems of the crystal class. Under the assumption of diluteness, we show that the discrete energy converges, in the sense of $\Gamma$-convergence, to a line-tension energy defined on Volterra line dislocations, regarded as integral vector-valued currents supported on rectifiable curves. Remarkably, the line-tension limit is of the same form as that derived from semi-discrete models of linear elastic dislocations based on a core cutoff regularization. In particular, the line-tension energy follows from a cell relaxation and differs from the classical {\sl ansatz}, which is quadratic in the Burgers vector. 
\end{minipage}
\end{center}

\tableofcontents

\section{Introduction}

The passage from discrete models
governed by interatomic potentials to
 continuum elasticity models is a long-standing problem in mathematical physics and the calculus of variations, and already proving that for realistic potentials the ground state is periodic is a major challenge (see \cite{LewinBlanc2015crystallization} and references therein).
To make the problem tractable, it is usual to
  assume crystallization and the existence of a global reference configuration. The variational approach via $\Gamma$-convergence permits to derive continuum limiting energies under the assumption that interactions are rapidly decaying
 \cite{AlicandroCicalese2004}.

In systems with defects, it is not clear that one can use a reference configuration to label the individual atoms. If defects are sufficiently dilute, one can identify locally an ordered configuration, but not a global one. A model accounting for local order was proposed and studied in \cite{LuckhausMugnai2010}.
The typical localized defect which plays a crucial role in the analysis of
lattice-invariant transformations and
the plastic behavior of metals is represented by dislocations.
In a very simplified scalar setting one can still use periodic interactions and a reference configuration, as was done for parallel screw dislocations in \cite{AlicandroDelucaGarroniPonsiglione2014}, see also
\cite{SalmanTruskinovsky2011}.
Treating general dislocations in three dimensions at the discrete level is a very challenging goal.
A major contribution in this direction was given by Ariza and Ortiz \cite{ArizaOrtiz06}, who proposed a reference-configuration based discrete model which encodes lattice-invariant transformations via an additional field representing plastic slip, and accounts for realistic geometrical properties of crystal lattices. A simplification of this model was used in \cite{Ponsiglione2007} for a scalar two-dimensional study of dislocations, in \cite{GiulianiTheil2021} for investigating statistical mechanical properties of defects, and in  \cite{AlicandroDelucaLazzaroniPalombaroPonsiglione2023} for edge dislocations in two dimensions.

We formulate here a variant of the model by Ariza and Ortiz \cite{ArizaOrtiz06}, which has been chosen so that it permits to rigorously address
the discrete-to-continuum limit in the presence of dilute dislocations in three dimensions, without geometric restrictions on the crystal lattice and the active slip systems.

For low-energy states  the physical interatomic potential is approximately quadratic and the limiting energy, in the sense of $\Gamma$--convergence, has an integral representation in $W^{1,2}(\Omega)$. However, the presence of Volterra line dislocations inserts line singularities into the continuum equilibrium problem for the elastic strain and solutions, when they exist, do not belong to the energy space $W^{1,2}(\Omega)$.
This singular structure raises the fundamental question of how to characterize the energetics of Volterra line dislocations, including an appropriate functional setting in which to represent such energetics, and the identification of the precise form of the energy as a continuum limit of discrete models.
The classical approach to this question is represented by a class of models, so-to-say semidiscrete, in which one assumes a quadratic energy density away from the singularity, and an \emph{ad hoc} regularization of the singularity at the dislocation core, where the continuum approximation fails. All these models lead to energies which diverge logarithmically as the radius of the core region (which models the lattice spacing) tends to zero.

A typical semidiscrete model of dislocations is the
one proposed by Peierls and Nabarro, then extended in \cite{KoslowskiCuitinoOrtiz2002, KoslowskiOrtiz2004}, which permits a full mathematical analysis when dislocations are confined to a plane \cite{GarroniMueller2005, GarroniMueller2006,ContiGarroniMueller2011,ContiGarroniMueller2023JEMS}.
For more general, but still well-separated, dislocation geometries one normally uses a core-cutoff regularization.
The analysis in \cite{ContiGarroniOrtiz2015}, reveals that, in the dilute regime the presence of the singularity of the dislocations determines the leading-order term in the energy. Specifically, if $\varepsilon$ denotes the characteristic length of the regularization, which loosely stands in for the lattice size of the crystal, the energy scales logarithmically as $\eps^2\log(1/\eps)$ to leading order for small $\varepsilon$. Furthermore, as proved in \cite{ContiGarroniMarziani}, the details of the regularization do not affect the limiting energy, which is of the {\sl line tension} type,
\begin{equation}\label{MaApOG}
    E(\mu)
    =
    \int_{\gamma\cap\Omega} \psi^\rel_\C(b,t)d\calH^1,
\end{equation}
where $\mu=b\otimes t\calH^1\LL\gamma$ is the dislocation measure supported on a one-rectifiable curve $\gamma\subseteq\Omega$, $t\in L^\infty(\calH^1\LL\gamma; S^{2})$ is the tangent to $\gamma$, $b\in L^1(\calH^1\LL\gamma;\calB)$ is the Burgers vector, which belongs to a lattice $\calB$ of possible Burgers vectors determined by the crystal, and $\psiC^\rel$ is the $\calH^1$-elliptic envelope \cite{ContiGarroniMassaccesi2015} of the line-tension energy of an infinite straight dislocation $\psiC$ (see Section~\ref{seccellproblem} below).
One aim of this work is to compare these semi-discrete models with a fully discrete model. Actually,
the main result of this paper, proved in Theorem~\ref{theomainresult} in Section~\ref{secgammaconvresult}, is that, despite the extensive generalization and added physical fidelity of the present crystalline-slip lattice model relative to the semi-discrete model of \cite{ContiGarroniOrtiz2015}, remarkably, the {\sl dilute limit} of the energy  remains of the same form, namely (\ref{MaApOG}).
In particular, the effective line tension $\psiC^\rel$ follows from the same $\calH^1$-elliptic envelope construction and depends solely on the effective elastic moduli $\C$ of the crystal. We note, however, that $\C$ is an input in the semi-discrete model of \cite{ContiGarroniOrtiz2015}, whereas here $\C$ is obtained from the discrete model, and arises from the crystal geometry and the specifics of the interatomic interactions.
Indeed, the energetics of dislocations in crystals depends sensitively on the geometry of the lattice and the physics of the atomic interactions, including multibody terms and finite range of interaction (see, e.~g., \cite{OrtizPhillips:1999, RamaArizaOrtiz2007, Finnis:2010, Moriarty:2023}).

A key contribution of the present work is the mathematical formulation of the discrete dislocation model, which requires careful but efficient book-keeping schemes to be put in place enabling the representation of general crystal classes and atomic interactions in a manner that is amenable to analysis. In particular, a rigorous description of discrete kinematics, in which crystallographic slip systems are explicitly taken into account, is required. The present formulation is a streamlined subset of the Ariza-Ortiz theory \cite{ArizaOrtiz06}, which regards crystal lattices as discrete differential complexes where discrete lattice dislocations arise as the result of discrete plastic slip, understood as lattice-preserving deformations \cite{ericksen:1979,PitteriZanzottoBook2003,SalmanTruskinovsky2011,BaggioArbibBiscariContiTruskinovskyZanzottoSalman2019-Landau} restricted to well-defined crystallographic planes and slip directions. This kinematics introduces an additional variable, the {\sl plastic slip variable}, that keeps track of the slip activity on the various slip systems and operates mainly on the atomic bonds that cross the slip planes. 

The energy of the crystal is then built on the kinematics of plastic slip and the displacements of the atoms. We specifically envision harmonic finite-range interactions such that the global energy of the crystal is the sum over the lattice of local contributions. The local energies are defined on local clusters of interacting atoms and satisfy the usual properties of material-frame indifference, or invariance under superimposed translations and linearized rotations. In addition, on infinite lattices the energy is translation-invariant and invariant under the symmetry group of the crystal. 
Focusing on cluster interactions also makes it possible for coercivity of the energy to be obtained and localization to be used without requiring the individual interactions to be nonnegative.
The cluster energy thus arises as the natural object on which to enforce invariance under (linearized) rotations and define the action of plastic slip.
This selects a subset within the class of admissible discrete energies in the class proposed by Ariza and Ortiz \cite{ArizaOrtiz06}, guaranteeing non-degeneracy of the model.

As in the other three-dimensional results on semidiscrete models of dislocations, 
we are concerned here with well-separated distributions of dislocations. In this limit, dislocations are assumed to be polygonal lines on a scale $h$ that is intermediate between the lattice scale $\eps$ and the size $L$ of the domain $\Omega$. 
Indeed, 
in three dimensions the energy of small-scale dislocation microstructures may be small, with the result that the total length of the dislocation lines is not controlled by the energy. 
In two dimensions, the well-separation condition has been removed in many recent results by a careful study of local clustering over many scales, see, e.~g., \cite{Ponsiglione2007,DelucaGarroniPonsiglione2012,Ginster1,Ginster2,AlicandroDelucaLazzaroniPalombaroPonsiglione2023}, building upon ideas first developed for the Ginzbug-Landau functional \cite{SandierSerfaty2007}. We remark that a restriction of our model to  two-dimensional lattices results in a general framework for discrete dislocation models in two dimensions to which similar strategies could be applied in order to remove the well-separation condition. In this paper, however, we focus on the three-dimensional case.

In closing this introduction, we point out a number of limitations of the present theory that suggest possible directions for further extensions. Firstly, as a matter of convenience we restrict attention to single-species centrosymmetric crystals in which all atoms are of the same kind and 'see' identical neighborhoods, or clusters. This simplification results in considerable notational convenience but rules out alloys or materials such as graphene, characterized by a number of different local clusters.
Secondly, while harmonic models of discrete dislocations such as considered here suffice to characterize crystal elasticity \cite{sengupta:1988}, including crystal acoustics \cite{musgrave:1970} and the topology of stable low-energy dislocation structures \cite{kuhlmann:1989}, consideration of anharmonic effects---and the attendant core relaxations---is important for achieving quantitative and material-specific accuracy \cite{GallegoOrtiz1993, WoodwardRao2002, Arca:2019}. Thirdly, our analysis is predicated on the assumption that the crystal is outfitted with a complete set of crystallographic slip systems, see Definition~\ref{defbetapl2}. However, some crystal classes are deficient as regards the number of operative slip systems (see, e.~g., \cite{Britton2015} for a recent review) and, therefore, are excluded by the present analysis. 
Finally, the issue of well-separation of line dislocations
remains open.

The paper is structured as follows. In
Section~\ref{secdiscretemodel} we introduce our discrete model including the kinematics of plastic slip and the corresponding energy. Moreover, we describe in detail the assumptions needed for the discrete formulation and the passage to the continuum limit, including in particular the  assumption of well-separation.
In Section~\ref{seclocalcontinuum} we relate the discrete energy to a continuum energy in the regions without dislocations.
In Section~\ref{sectionextensionrigidity} we prove compactness for a
semidiscrete continuum model with well-separated dislocations. A key ingredient in the proof is the construction of an extension of the elastic strain inside the core, permitting to apply a rigidity result. Some of the details of the extension are presented in Appendix~\ref{secappext}. Finally,
Section~\ref{secgammaconvresult} is devoted to the statement and proof of the main $\Gamma$-convergence result, Theorem~\ref{theomainresult}.
As a preliminary step before proving Theorem~\ref{theomainresult}, we show in Theorem~\ref{theomainresultelastic} that our discrete energy without defects $\Gamma$-converges to a linearized elastic model, which in particular identifies the elastic constants $\C$ appropriate for the description of the line-tension energy $\psi_\C$.

\section{The discrete model}

We consider systems of interacting atoms arranged as a crystal lattice away from line defects, or {\sl dislocations}, and seek to characterize the low-energy configurations of the system under suitable diluteness conditions. The main elements of the model concern the geometrical description of the local neighborhoods of the atoms, or {\sl clusters}, and their possible deformations, including crystallographic slip and incompatibility in the core regions; the local energetics of the clusters and the attendant global energetics of the crystal, including traction-free boundary conditions; scaling of the crystal lattice and attendant continuum limit; and appropriate notions of diluteness of the dislocation lines. 

\label{secdiscretemodel}
\begin{figure}
 \begin{center}
  \includegraphics[width=11cm]{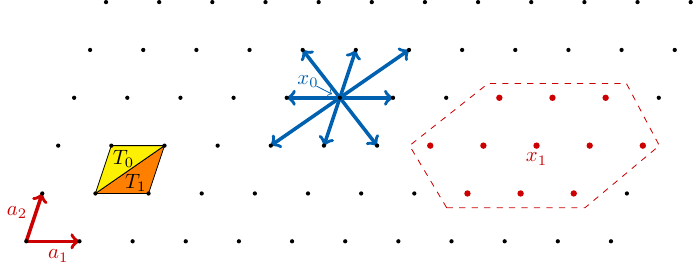}
 \end{center}
\caption{Sketch of the geometry in the simple example discussed in Section \ref{seckinematics}.
The two vectors $a_1$ and $a_2$ are displayed in red in the lower-left corner, a copy of each of the two simplexes is drawn in yellow and orange, and the bonds starting from a point $x_0$  (i.e., the points $x_0+h$, $h\in\calN$) are drawn in blue. The points marked red are the copy  $x_1+\calC$ of a  cluster  $\calC$, it consists of eleven points, the nine in $\{x_1+\sum_i \lambda_i a_i: \lambda\in\{-1,0,1\}^2\}$ as well as  $\{x_1\pm 2 a_1\}$.
}
\label{fig-lattice}
\end{figure}

\subsection{Discrete kinematics in $\R^n$}
\label{seckinematics}
For the entirety of the paper, we shall assume that a lattice $\calL$, a set of bonds $\calN\subseteq\calL$, and a cluster $\calC\subseteq\calL$ are fixed, defined as follows, cf.~Figure~\ref{fig-lattice}. 
 
\begin{definition}[Discrete crystal]\label{deflattice}
By a discrete crystal we shall understand the combination of the following objects:
\begin{itemize}
\item[i)]{Bravais lattice.} A Bravais lattice $\calL\subseteq\R^n$ is a set of the form $\calL = \{\sum_{i=1}^n z_i a_i : z\in\Z^n\}$, where $\{a_1, \dots, a_n\}$ is a basis of $\R^n$.
\item[ii)]{Bond set.} A set of bonds $\calN$ is a finite subset of $\calL$ such that $\calN=-\calN$, $0 \not \in \calN$. 
\item[iii)]{Cluster.} A cluster $\calC$ is a finite subset of $\calL$ such that $0\in\calC$.
\end{itemize}
\end{definition}

The sets $\calN$ and $\calC$ are the building blocks used to define the kinematics and then the energetics of defects, and need to be sufficiently rich (for example, Definition~\ref{def-tetrahedra}). 

The bond set $\calN$ is determined by the atomic interactions, in the sense that if atoms $x,y\in\calL$ interact then necessarily $x-y\in\calN$. In particular, discrete displacement gradients and deformations are defined on such pairs, and not on all possible pairs, so that the kinematics already encodes a finite interaction range. We remark that $\calN$ is, in general, larger than the set of first neighbours, and can be chosen to have the same symmetry of $\calL$. 

The cluster $\calC$ is the basic unit on which the energy is defined. The energies under consideration are quadratic, but not necessarily the sum of pairwise interactions, so that, in particular, three-point and higher-order interactions are also possible.
One crucial aspect in the choice of the cluster is that the energy of a single cluster  needs to be nonnegative and coercive, in a sense made precise below.
These concepts are further formalized in Section \ref{seccluster}, see, in particular, Definition~\ref{defcluster}.

\begin{example}[Two-dimensional discrete crystal]
{\rm 
A pair of vectors $a_1$, $a_2$ spanning $\R^2$ defines a two-dimensional discrete crystal as follows: i) A Bravais lattice $\calL:=\{a_1z_1+a_2z_2: z\in\Z^2\}$; ii) a bond set $\calN:=\{\pm a_1, \pm a_2, \pm (a_1+a_2), \pm (a_1-a_2)\}$; iii) a cluster
$\calC:=\{0, \pm a_1, \pm a_2,  \pm (a_1+a_2), \pm (a_1-a_2),\pm 2a_1\}$, cf.~Figure~\ref{fig-lattice}.}\hfill$\square$
\end{example}

We next introduce displacement fields and strains on discrete crystals without regard for domains or boundaries. In Section \ref{seclocalresc}, we rescale and localize fields to subsets of $\R^n$ defining the domain of a crystalline solid.

\begin{definition}[Displacements and strains]\label{eqdefduxi}
A displacement field on scale $1$ is a map $u:\calL\to \R^n$. Its discrete gradient is the map $du:\calL\times \calN\to\R^n$ defined by
\begin{equation*}
    du(x,h):=u(x+h)-u(x), \quad x\in\calL, h\in\calN.
\end{equation*}
We verify that $du(x,h)=-du(x+h,-h)$.

A strain field, or discrete deformation, on $A\subseteq\calL$ is a function $\discbeta:A \times \calN\to \R^n$. We say that $\discbeta$ is admissible if $(x,h) \in A$ implies $(x+h,-h)\in A$ and
\begin{equation}\label{eqbetacompatible1}
    \discbeta(x,h)=-\discbeta(x+ h, -h), \quad \text{for all} \; (x,h)\in A.
\end{equation}
\end{definition}

We recall that, in the elastic case, discrete deformations $\xi$ are set directly by the displacement field $u$, in the sense that $\discbeta(x,h)=du(x,h)$, and are always admissible.

In order to relate discrete to continuous deformations, it will be useful to fix a decomposition of a unit cell of $\calL$ into simplices. By unit cell we understand a set of the form $\{\sum_{i=1}^n \lambda_i a'_i : \lambda \in [0,1]^n\}$ for some choice of the vectors $a_i'$ such that $\calL = \{\sum_{i=1}^n z_ia'_i: z\in\Z^n\}$. We require simplices to have vertices in $\calL$ and edges in $\calN$. Evidently, the existence of this set of simplices is only possible if the sets $\calN$ and $\calC$ are not too small.

\begin{definition}[Simplicial cover]\label{def-tetrahedra}
A set $T\subseteq\R^n$ is a lattice simplex if there is a set  $\{x_1, \dots, x_{n+1}\} \subseteq\calL$, denoted by   $\vertici(T)$, such that $x_i-x_j\in\calN$ for all $i\ne j$, $T=\conv\{x_1,\dots, x_{n+1}\}$ and $T$ has nonempty interior.

A $(\calC,\calN)$-admissible cover of  the lattice $\calL$ is a finite set $\{T_0, \dots, T_{N_T}\}$ of lattice simplices such that they have disjoint interior, $T_*:=\cup_{i=0}^{N_T} T_i$ is a unit cell of $\calL$, $0\in T_0$, $\vertici(T_i)\subseteq\calC$ for all $i$, and the set $\{x+T_i: x\in \calL, i\in\{0,\dots, N_T\}\}$ is a conforming mesh. This means that if $T',T''$ are two elements of the mesh then $T'\cap T''=\conv(\vertici(T')\cap \vertici(T''))$.

We also define the constants
\begin{equation}\label{eq-def-dT}
	d_{T_*}:= 2\max \{|x|: x\in T_*\}\geq \diam(T_*)
\end{equation}
and 
\begin{equation}\label{eq-def-dC}
	d_{\calC}:= 2\max \{|x|: x\in \calC\}\geq \diam(\calC).
\end{equation} 
\end{definition}
 
We shall therefore assume that $\calC$ and $\calN$ are sufficiently rich to ensure the possibility of constructing the mesh in a way which is compatible with them. This is certainly {satisfied} if they contain all interactions up to a given range, which is larger than the diameter of the unit cell. We stress that the choice of the cover is, in general, not unique, it may have lower symmetry than $\calL$, $\calC$, and $\calN$. Both the discrete and the limiting energy do not depend on the choice of the cover, but many steps in the proof, and the constants in many estimates, will depend on it. One can easily see that, for a given triplet $(\calL,\calC,\calN)$ there are only finitely many possible choices of the set of simplices, therefore the constants can be made to depend only on $(\calL,\calC,\calN)$ with only notational changes.

\begin{remark}
The definition above implies that the closed set $T_i$ does not contain any lattice point except for the vertices, $\# (T_i\cap\calL)=n+1$ and $\calL^n(T_i)=\calL^n(T_*)/n!$ for all $i$ and $N_T=n!-1$. As these results are not used in the following, we do not provide details of the proofs.
\end{remark}

For example, $T_0, \dots, T_{N_T}$ could be obtained as the Freudenthal partition of a unit cell $\{ \sum_i  \lambda_i a'_i, \lambda\in[0,1]^n\}$.
Specifically, for any permutation $\sigma$ of $\{1, \dots, n\}$ one takes the closed simplex
\begin{equation}
 T_\sigma:=\{ \sum_{i=1}^n \lambda_i a'_i: \lambda\in[0,1]^n, \lambda_{\sigma(1)}
\le \lambda_{\sigma(2)}\le \dots\le \lambda_{\sigma(n)}\},
\end{equation}
these $n!$ sets have disjoint interior. This automatically produces a conforming mesh. In the two-dimensional situation sketched above one obtains
$T_0=\{\lambda_1a_1+\lambda_2a_2: 0\le \lambda_1\le \lambda_2\le 1\}$ and
$T_1=\{\lambda_1a_1+\lambda_2a_2: 0\le \lambda_2\le \lambda_1\le 1\}$, see Figure~\ref{fig-lattice}.

\subsection{Interaction energy in a cluster}
\label{seccluster}
The interaction energy is defined clusterwise. We start by a single cluster. The elastic energy of a cluster is a map $\EC$ defined on the set of deformations on the cluster $\calC$.  Precisely,
given $\calL$, $\calC$ and $\calN$ as in Definition~\ref{deflattice} such that a $(\calC,\calN)$-admissible cover $\{T_0, \dots, T_{N_T}\}$ in the sense of Definition~\ref{def-tetrahedra} exists, we denote by
\begin{equation*}
 \calC_\calN:=\{(x,h)\in\calC\times\calN: x+h \in\calC\}
\end{equation*}
the set of bonds within the cluster, and by
\begin{equation}
    D_\calC:=\{\discbeta: \calC_\calN\to\R^n, \text{ $\discbeta$ is admissible}\}
\end{equation}
the set of possible cluster deformations.
We assume that $\EC$ is a quadratic function on $D_\calC$, invariant under linearized rotations, and coercive, in a sense that we now make precise.

We start by introducing a simplified reference energy.
Let $\calC^0:=\bigcup_i\vertici T_i\subseteq\calC$, and $\calC_\calN^0:=\{(x,h)\in \calC^0\times \calN:x+h\in\calC^0\}\subseteq \calC_\calN$.
We define the reference energy  $E^0_\calC:D_\calC\to\R$ as
\begin{equation}\label{eqdefE0C}
 E^0_\calC[\discbeta]:=
 \min\{ \sum_{(x,h)\in  \calC_\calN^0}|\discbeta(x,h)-Sh|^2: S\in \R^{n\times n}_\skw\}.
\end{equation}
Here and in the entire paper, we use the Euclidean norm on all finite-dimensional vector spaces, such as for example $D_\calC$, after identifying them with $\R^k$ as appropriate.
{Given a matrix $S\in\R^{n\times n}$, representing a continuum deformation gradient, one can obtain the corresponding discrete strain field by
\begin{equation}
 \xi^S(x,h):=Sh.
\end{equation}}
\begin{definition}\label{defcluster}
A cluster elastic energy is a quadratic function $\EC: D_\calC\to\R$ which is invariant under linearized rotations and coercive, in the sense that
\begin{equation}\label{eqclusterlinearrotations}
\EC[\discbeta+\discbeta^S]=\EC[\discbeta], \hskip5mm\text{ for every $\discbeta\in D_\calC$ and  $S\in\R^{n\times n}_\skw$},
\end{equation}
and there is $\alpha >0$ such that
\begin{equation}\label{eqasscoerc}
  \alpha E^0_\calC\le  \EC,
\end{equation}
with $E^0_\calC$ as in \eqref{eqdefE0C}.
From $\EC$ we define $\C\in\R^{n\times n\times n\times n}_\sym$  as
\begin{equation}\label{eqdefC}
  \frac12 \C A\cdot A = \frac{1}{\calL^n(T_*)}\EC[\discbeta^A]
\end{equation}
for all $A\in\R^{n\times n}$. 
\end{definition}
Since $E_\calC$ is a quadratic form, there is $c>0$ such that
\begin{equation}\label{eqclusterfromabove}
    \frac1c E_C[\discbeta]\le |\xi|^2= \sum_{(y,h)\in \calC_\calN}
    |\discbeta(y,h)|^2 \text{ for all }\xi\in D_\calC.
\end{equation}
Here and subsequently, $c$ denotes a generic constant that may depend on $\calL$, $\calC$, $\calN$, $E_\calC$ and the tetrahedra, but not on $\eps$, $u$ or $\xi$.
We stress that these quantities are only defined if a
 $(\calC,\calN)$-admissible cover $\{T_0, \dots, T_{N_T}\}$ in the sense of Definition~\ref{def-tetrahedra} exists.
\begin{remark}
The definition of $\C$ is well posed since $A\mapsto\discbeta^A(x,h):=A h$ is linear and $\EC$ is quadratic, so the right-hand side of \eqref{eqdefC} defines a quadratic form on $\R^{n\times n}$. From the invariance under linearized rotations in \eqref{eqclusterlinearrotations} we obtain $\C A=0$ whenever $A+A^T=0$. From the coercivity in \eqref{eqasscoerc} and the definition~\eqref{eqdefE0C} applied to $\xi^A$ we obtain 
\begin{equation}\label{eqpropC}
    \C A\cdot A \ge c_0|A+A^T|^2 \hskip2mm\text{ and } 
    \hskip2mm \C(A-A^T)=0 \hskip2mm\text{ for all } A\in\R^{n\times n}\,.
\end{equation}
\end{remark}

\subsection{Localization and rescaling}
\label{seclocalresc}
We shall be interested in a crystal occupying a finite domain $\Omega\subseteq\R^n$, with atomic distance $\eps$ much smaller than the size domain 
\begin{definition}[Scaled displacements and deformations]\label{defdepsu}
Let $\Omega\subseteq\R^n$ be open, $\eps>0$. A displacement field on scale $\eps$ on $\Omega$ is a map $u:\Omega\cap\eps\calL\to\R^n$, its discrete gradient $d_\eps u: \bonds_\eps^\Omega\to\R^n$ is defined by
\begin{equation}\label{eqdefdepsu}
 d_\eps u(x,h):=\frac{u(x+\eps h)-u(x)}{\eps},\end{equation}
where
\begin{equation}
 \bonds_\eps^\Omega := \{(x,h)\in (\eps\calL\cap\Omega)\times\calN:  x+\eps h\in\Omega\}
\end{equation}
denotes the set of bonds in $\Omega$. A deformation is a map $\discbeta : \bonds_\eps^\Omega\to\R^n$, it is admissible if 
\begin{equation}\label{eqbetacompatible}
    \discbeta(x,h)=-\discbeta(x+\eps h, -h) 
    \text{ for all $(x,h)\in  \bonds_\eps^\Omega $.}
\end{equation}
\end{definition}

Furthermore, we denote by 
\begin{equation}\tetr_\eps^\Omega:=\{x+\eps T_i \subseteq\Omega: x\in\eps \calL,\, 0\le i\le N_T\}
\end{equation}
 the set of all scaled lattice simplices contained in $\Omega$. For $T\in \mathcal{T}_\eps^\Omega$ we denote by 
\begin{equation}
\vertici(T):=T\cap \eps \calL
\end{equation}
its $n+1$ vertices, and by 
\begin{equation}
\edges(T):=\{(x,y)\in \vertici(T)\times \vertici(T): x\ne y\}
\end{equation} its $n(n+1)$ oriented edges.
The clusters contained in $\Omega$ are represented by the points
\begin{equation}\label{eqdefclustepsomega}
 \clust_\eps^\Omega := \{x\in\eps\calL: x+\eps\calC\subseteq\Omega\}.
\end{equation}

We stress that the bond $h\in\calN$ in the expression $d_\eps u(x,h)$ is not scaled, as it is mainly used to denote a direction and to label a type of connection between lattice atoms. If $u$ is continuously differentiable, then $d_\eps u(x,h)$ is the average of {$(Du)h$} over the segment $[x,x+\eps h]$.

\begin{figure}
 \begin{center}
  \includegraphics[width=6cm]{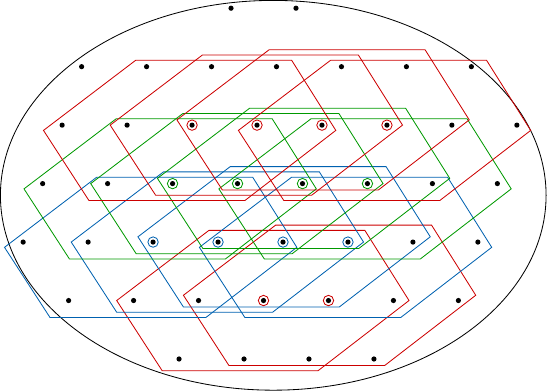}
 \end{center}
\caption{Sketch of a set $\Omega$, the points $\eps\calL\cap\Omega$,
and the set $\clust_\eps^\Omega$ (circled dots) of centers of clusters contained in $\Omega$, for the lattice in Figure~\ref{fig-lattice}. The colored lines mark the sets $x+\eps\calC$, for each $x\in \clust_\eps^\Omega$.}
\label{fig-lattice3}
\end{figure}

The energy in the open set $\Omega\subseteq\R^n$ is then defined, for an admissible deformation $\discbeta:\bonds_\eps^\Omega\to\R^n$, by
\begin{equation}\label{eqdeftotalenergy}
    E_\eps[\discbeta, \Omega]:=\sum_{x\in\clust_\eps^\Omega } \eps^n
    \EC[\discbeta(x+\eps \cdot, \cdot)]=\sum_{x\in\clust_\eps^\Omega } \eps^n
    \EC[\tau^x_\eps\discbeta],
\end{equation}
where $\tau^x_\eps$ denotes the translation and scaling operator, defined by
\begin{equation}\label{eqdeftaux}
(\tau^x_\eps\discbeta)(y,h):=\discbeta(x+\eps y,h).
\end{equation}

This concludes the definition of the model in the elastic case.

\subsection{Slip and plastic strains}
In order to introduce plastic slip, we fix a set of slip systems, according to the crystal class of the lattice \cite{HirthLothe1968,HullBacon:2001}, and define the plastic strain $\xi^\pl$ so that the slip across each bond is an integer multiple of a possible slip vector.
Plastic slip can be seen as a \emph{reconstructive} phase transformation, bringing the material outside of the Ericksen-Pitteri neighbourhood \cite{ericksen1989weak,PitteriZanzottoBook2003,BaggioArbibBiscariContiTruskinovskyZanzottoSalman2019-Landau}, and therefore is not reversible. From a purely variational viewpoint, the relaxation of the continuum energy in the presence of dislocations is degenerate \cite{Fonseca1988}. In the present discrete setting, the quantization of plastic slip leads naturally to a finite energy cost for dislocations, and therefore to a penalization of incompatibilities, even without adding regularizing terms. In particular, the discreteness of plastic slip prevents a degenerate relaxation of the elastic energy.

\begin{definition}[Slip systems]\label{defslipsys}
A set of slip systems is a  finite set of pairs 
\begin{equation}
    (\burgersbi{1},m_1),\dots, (\burgersbi{N_s}, m_{N_s})\in \calL\times \calL^*,
\end{equation} 
with $\burgersbi{\ell}\ne0$, $m_\ell\ne0$ and  $\burgersbi{\ell}\cdot m_\ell=0$ for all $\ell$. Here $\calL^*$ denotes the dual lattice,
\begin{equation}
    \calL^*:=\{y\in \R^n: y\cdot x\in \Z \text{ for all } x\in \calL\}.
\end{equation}
The $\burgersbi{\ell}$ are called Burgers vectors, the $m_\ell$ are called slip-plane normals. We let $\calB:=\{\sum_{\ell=1}^{N_s} z_\ell \burgersbi{\ell}: z\in\Z^{N_s}\}\subseteq \calL$.
\end{definition}

\begin{definition}[Discrete plastic strain]\label{defbetapl} 
We say that  $\discbetapl:\bonds_\eps^\Omega\to\calB$ is a discrete plastic strain if, for all $(x,h)\in\bonds_\eps^\Omega$, $\discbetapl(x,h) = -\discbetapl(x+\eps h,-h)$ and there is $\discxi(x,h)\in\R^{N_s}$ such that $\discxi_\ell(x,h)m_\ell\cdot h\in\Z$ for all $\ell$ and
\begin{equation}\label{eqbetapldisc}
    \discbetapl(x,h)
    =
    \sum_{\ell=1}^{N_s}
    \discxi_\ell(x,h) (m_\ell\cdot h)
    {\burgersbi{\ell}}.
\end{equation}
\end{definition}
From \eqref{eqbetapldisc}, we get $\discbetapl(x,h)\in\calB\subseteq \calL$ for all $x$ and $h$. From $\calB\subseteq \calL$, we see that $\calB$ is a lattice that can be generated by $n'$ linearly independent vectors, for some $n'\le n$.

\begin{definition}\label{defbetapl2}
We say that the set of slip systems is complete if there is a set of $n'$ independent vectors $\hat b_1$, $\hat b_2, \dots, \hat b_{n'}$, with $n'$ the dimension of $\calB$, such that $\calB=\Span_\Z\{\hat b_1, \hat b_2, \dots, \hat b_{n'}\}$ and for every $j\in\{1,2,\dots, n'\}$ there are $n-1$ linearly independent vectors $m_j^1,\dots, m_j^{n-1}$ such that $(\hat b_j, m_j^k)$  is a slip system for all $j$ and $k$.
\end{definition} 

The completeness condition states that for every vector $h\in\R^n$ which is not parallel to $\hat b_j$ there is at least one index $k$ such that $m_j^k$ is not orthogonal to $h$. This condition will be important for the discretization of continuum plastic strains in the upper bound construction in Section \ref{secdiscreteupperbound}. It is easy to see that the classical sets of slip systems for simple cubic, face-centered cubic (fcc) and body-centered cubic (bcc) lattices fulfill this condition.
 
Given a deformation $u:\Omega\cap \eps \calL\to\R^n$ and a discrete plastic strain $\discbetapl$, the elastic strain $\discbetael:\bonds_\eps^\Omega\to\R^n$ can be determined by
\begin{equation*}
    \discbetael (x,h)
    :=
    d_\eps u(x,h)-\discbetapl(x,h)
    =
    \frac{u(x+\eps h)-u(x)}{\eps}-\sum_{\ell=1}^{N_s} 
    \discxi_\ell(x,h)\burgersbi{\ell} m_\ell\cdot h.
\end{equation*}
The total elastic energy in~\eqref{eqdeftotalenergy} can then be understood as only depending on the elastic strain, $E_\eps[\discbetael, \Omega]= E_\eps[d_\eps u-\discbetapl,\Omega]$.

We remark that, if $\discxi(x,h)=\discxi^*\in\R^{N_s}$ is constant and $u(x):=Fx$ is affine, then
\begin{equation*}
    \discbetael(x,h)
    =
    (F-\sum_{\ell=1}^{N_s} \discxi^*_\ell \burgersbi{\ell}\otimes m_\ell)h.
\end{equation*}

In order to detect the presence of dislocations we shall examine the circulation of the deformation $\discbeta$ along discrete paths. In terms of the theory of de Rham chain complexes, one can identify the lattice as a 0-chain and the discrete path as a 1-chain, see \cite{ArizaOrtiz06} for discussions thereof. We adopt here instead a more explicit notation that is better adapted to treating boundary and localization effects, including general finite-range energies.

\begin{definition}\label{definitioncurlfree}
Let $\omega\subseteq\R^n$. An $\eps$-discrete path in $\omega$ is a tuple $P=(x_0, x_1, \dots, x_K)$ such that  $ x_k\in\omega\cap \eps \calL$ and $x_k-x_{k-1}\in\eps\calN$ for all $k$. The path $P$ is closed if $x_0=x_K$. A path may consist of a single point. An edge of $P$ is a pair $(x_{k-1}, x_{k})$, for $1\le k\le K$.

The circulation of 
$\discbeta:\bonds_\eps^\omega\to\R^n$ over the closed path $P=(x_0, \dots, x_K)$ is defined as
\begin{equation}
\circu(\discbeta, P):= \sum_{k=1}^K \eps \discbeta(x_{k-1}, \frac{x_k-x_{k-1}}\eps ) .
\end{equation}

We say that $\discbeta:\bonds_\eps^\omega\to\R^n$ is \emph{exact}
if for every closed $\eps$-discrete path  $P$ in $\omega$ 
the circulation of $\discbeta$ over $P$ vanishes.

We say that an \emph{$\eps$-discrete path $P$ is elementary} if for each edge $(x_{k-1}, x_{k})$ of $P$ there is $T\in \calT^{\R^n}_\eps$  such that $(x_{k-1},x_{k})\in\edges(T)$.

We write $\oplus$ for the concatenation operation, $(x_1,\dots, x_K)\oplus(x_K,\dots, x_H)=(x_1,\dots, x_K,\dots x_H)$ and write briefly $P\subseteq A$, for $A\subseteq\R^n$, if all elements of $P$ are in $A$.
\end{definition}

\begin{remark}\label{remarkcircinL}
From this definition and the kinematic condition in \eqref{eqbetapldisc}, it follows that $\circu(\discbetapl,P)\in\eps \calB$ for any discrete plastic strain $\discbetapl$ and any closed $\eps$-discrete path $P$.
\end{remark}

\begin{lemma}\label{lemmadiscretepath}
Let $\omega\subseteq\R^n$. A map $\discbeta: \bonds_\eps^\omega\to\R^n$ is exact if and only if there is $u:\eps\calL\cap\omega\to\R^n$ such that $\discbeta=d_\eps u$ on $\bonds_\eps^\omega$.
\end{lemma}
\begin{proof}
If $u$ exists, then for any closed path as in Definition \ref{definitioncurlfree} we have
\begin{equation}
\begin{split}
 \sum_{k=1}^K \discbeta( x_{k-1}, \frac{x_k-x_{k-1}}\eps) =&
 \sum_{k=1}^K   d_\eps u(x_{k-1}, \frac{x_k-x_{k-1}}\eps) \\=&
\sum_{k=1}^K \frac1\eps\left(  u(x_k)-u(x_{k-1})\right)=0.
\end{split}
\end{equation}
 
 Assume now that $\discbeta$ is exact. 
 We say that a set $A\subseteq\omega\cap \eps \calL\subseteq\R^n$ is $\eps$-discretely connected if for any $a,b\in A$,
there is a discrete path in $\omega$ with $x_0=a$ and $x_K=b$.
 Let $A_1,\dots, A_J$ be the discretely connected components of $\omega\cap \eps \calL$, defined via the equivalence relation $a\sim b$ if there is a discrete path in $\omega$ which contains both $a$ and $b$.
 We pick an element $a_j\in A_j$ for each connected component, and set $u(a_j)=0$.
 For $b\in A_j$, we consider any path $(a_j=x_0, x_1, \dots, x_K=b)$ and set
 \begin{equation}
  u(b):=\sum_{k=1}^K \eps\discbeta( x_{k-1}, \frac{x_k-x_{k-1}}{\eps}).
 \end{equation}
Since $\discbeta$ is exact, this definition is well posed.
\end{proof}

\begin{definition}\label{defLbeta}
For $\Omega\subseteq\R^n$ open and $\discbeta:\bonds^\Omega_\eps\to\R^n$ we define the interpolation
$L\discbeta\in L^\infty_\loc(\R^n;\R^{n\times n})$ as follows. Inside each simplex $T\in\tetr^\Omega_\eps$, $L\discbeta$ is constant and is defined as
the unique minimizer over all $F\in\R^{n\times n}$ of
\begin{equation}\label{eqdefLbeta}
 \calA_{T}(F):=\sum_{(x,x+\eps h)\in \edges(T)} |Fh-\discbeta(x,h)|^2.
\end{equation}
In the rest of $\R^n$ we set $L\discbeta=0$. We denote by $L\xi|_T$ the value of $L\xi$ inside $T$.
\end{definition}
As $L\xi$ is piecewise constant, $L\xi\in SBV_\loc(\R^n;\R^{n\times n})$, with $\nabla L\xi=0$ and $J_{L\xi}\subseteq \bigcup_{T\in \calT^{\Omega}_\eps}\partial T$.
\begin{remark}\label{remark-Lcsiexact}
	If $\discbeta$ is exact  on $T$, with $T\in \tetr_\eps^\Omega$,
	then the minimum in (\ref{eqdefLbeta}) is zero, the minimizer coincides with the gradient of the affine interpolation between the values of the map $u$  obtained in Lemma \ref{lemmadiscretepath}, and
\begin{equation}
 L\discbeta|_Th=\discbeta(x,h) \text{ if } (x, x+\eps h)\in \edges(T).
\end{equation}
In particular, if $(x,x+\eps h)$ is a common edge of two simplices $T$, $T'\in \tetr_\eps^\Omega$
and $\discbeta$ is exact on both, then $L\xi|_Th=L\xi|_{T'}h$.
If $\discbeta$ is exact on a set $\omega$, and two simplices
$T\ne T'\in \calT^\omega_\eps$ share a face, then
the jump $[L\discbeta]$ over the face $T\cap T'$ is orthogonal to $T\cap T'$.
	\end{remark}

The energy of a cluster gives a bound on the symmetric part of $L\discbeta$ and on the incompatibility of $L\discbeta$ on the boundary of the tetrahedra.

\begin{proposition}\label{prop-Lcsi}
There exists a constant $c>0$ such that, for every $\Omega'\subset\subset\Omega$ with $\dist(\Omega', \partial\Omega)\ge d_\calC\eps$ and 	every $\discbeta : \bonds^\Omega_\eps \to \R^n$, 
\begin{equation}\label{eqLbetasymEcint}
    \int_{\Omega'}|L\discbeta+(L\discbeta)^T|^2dy \le c E_\eps[\discbeta,\Omega].
\end{equation}
Moreover, if $T\ne T'\in \calT^\Omega_\eps$ share a face, with $T\subseteq z+\eps T_*$ and $T'\subseteq z'+\eps T_*$, then
\begin{equation}\label{eqcurlLbetatetr}
	|D L\discbeta-(DL\discbeta)^T|(T\cap T') 
    \le 
    c \eps^{n-1}(\EC[\tau^{z}_\eps\discbeta]
    +
    \EC[\tau^{z'}_\eps\discbeta])^{1/2}.
\end{equation}
\end{proposition}
In \eqref{eqcurlLbetatetr}, transposition acts on the last two indices of the third-order tensor $DL\discbeta$. If  $n=3$, then
$(D L\discbeta-(DL\discbeta)^T)_{ijk}=\sum_l \epsilon_{jkl} (\curl L\discbeta)_{il}$, where $\epsilon_{123} = 1$ and $\epsilon$ is antisymmetric in all pairs of indices, and similarly for $n=2$. In particular, for $n=2$ and $n=3$ we obtain
\begin{equation}
    |D L\discbeta-(DL\discbeta)^T|
    =
    \sqrt2 |\curl L\discbeta|.
\end{equation}
\begin{proof}
Consider any cluster $z+\eps\calC$, $z\in \clust_\eps^\Omega$. From \eqref{eqdefE0C} and \eqref{eqasscoerc}, there is $S\in \R^{n\times n}_\skw$ such that
\begin{equation}\label{eq224}
	\sum_{(x,h)\in\calC_\calN^0} 
    |Sh-\discbeta(z+\eps x, h)|^2
    =
    \EC^0[\tau^z_\eps\discbeta]
    \le 
    \frac1\alpha \EC[\tau^z_\eps\discbeta].
\end{equation}
Let $T$ be a simplex contained in $z+\eps T_*$. Since all edges of $T$ are contained in the sum in \eqref{eq224}, and $L\discbeta|_T$ is a minimizer of \eqref{eqdefLbeta}, we have
\begin{equation}\label{eq225}
	\sum_{(x,x+\eps h)\in\edges(T)} |L\discbeta|_Th-\discbeta(x,h)|^2  
    \le 
    \sum_{\edges(T)} |Sh-\discbeta(x,h)|^2  
    \le 
    \frac1\alpha \EC[\tau^z_\eps\discbeta].
\end{equation}
With a triangular inequality, $|L\discbeta|_T-S|^2\le c \EC[\tau^z_\eps\discbeta]$. Since $S+S^T=0$,
\begin{equation}\label{eqLbetasymEc}
	\left|L\discbeta|_T+(L\discbeta|_T)^T\right|^2
    \le 
    c \EC[\tau^z_\eps\discbeta],
\end{equation}
the same holds for all tetrahedra contained in $z+\eps T_*$. For every $y\in \Omega'$, there is $z\in \eps\calL$ such that $y\in z+\eps T_*$. By the assumption $\dist(\Omega', \partial\Omega)\ge d_\calC\eps$ we obtain $z+\eps\calC\subseteq\Omega$, so that $z\in \clust_\eps^\Omega$. We multiply \eqref{eqLbetasymEc} by $\calL^n(\eps T_*)$ and sum over all $z\in \clust_\eps^\Omega$ to obtain \eqref{eqLbetasymEcint}.

We now turn to \eqref{eqcurlLbetatetr}. As  $L\xi\in SBV_\loc(\R^n;\R^{n\times n})$, to estimate $DL\xi$ it suffices to consider the faces of the simplices. Fix two simplices $T$, $T'\in \tetr_\eps^\Omega$ with a common edge $(x,x+\eps h)$. Let $z,z'\in \eps\calL$ be nodes  (possibly equal) such that $T\subseteq z+\eps T_*$, $T'\subseteq z'+\eps T_*$. Then, the two values $L\xi|_T$ and $L\xi|_{T'}$  by \eqref{eq225} obey
\begin{equation}\label{eqjumplxih1}
    \bigl|L\discbeta|_Th-\discbeta(x,h)\bigr|^2 
    \le 
    c \EC[\tau^z_\eps\discbeta]
	\text{ and } \bigl|L\discbeta|_{T'}h-\discbeta(x,h)\bigr|^2 
    \le 
    c \EC[\tau^{z'}_\eps\discbeta].
\end{equation}
If $T\cap T'$ is an $n-1$-dimensional face, then the jump over $T\cap T'$ is  $[L\xi] = L\xi|_T-L\xi|_{T'}$ and \eqref{eqjumplxih1} gives
\begin{equation}\label{eqjumplxih}
	\bigl|[L\xi]h\bigr| \le c (\EC[\tau^{z}_\eps\discbeta]+\EC[\tau^{z'}_\eps\discbeta])^{1/2}.
\end{equation}
As the mesh is conforming, \eqref{eqjumplxih} holds for any edge of the face $T\cap T'$, and allows to control all components of the jump $[L\xi]$ except for the one along the normal $\nu$ to $T\cap T'$,
\begin{equation}\label{eqjumplxih2}
	\bigl|[L\xi]-[L\xi]\nu\otimes \nu\bigr| 
    \le 
    c (\EC[\tau^{z}_\eps\discbeta]+\EC[\tau^{z'}_\eps\discbeta])^{1/2}.
\end{equation}
Finally, using $(DL\xi)\LL (T\cap T')=[L\xi]\otimes \nu\calH^{n-1}\LL(T\cap T')$,
\begin{equation}\label{eqcurlLbetatetra}
	|DL\discbeta-(DL\discbeta)^T|( T\cap {T'}) \le2 \calH^{n-1}(T\cap T')
	\bigl|[L\xi]-[L\xi]\nu\otimes \nu\bigr| 
\end{equation}
and \eqref{eqcurlLbetatetr} follows.
\end{proof}

\subsection{Incompatibilities and core region}

The analysis of dislocations will be based on the assumption that the strain $\discbeta:\bonds_\eps^\Omega\to\R^n$ is compatible over most of the domain, with the exceptional set representing, in a continuum picture, the dislocation cores. We shall denote the region  which contains incompatibilities as $\Ce$, see Definition~\ref{definitioncore}. Since we work with a discrete model, the question of compatibility or incompatibility of $\xi$ can only be posed on a length scale larger than the maximal length of one of the segments on which it is defined. We first show that we can reduce any discrete path to an elementary  path, and in the process introduce the relevant length scale $\ccorerad\eps$.

\begin{figure}
 \begin{center}
  \includegraphics[width=9cm]{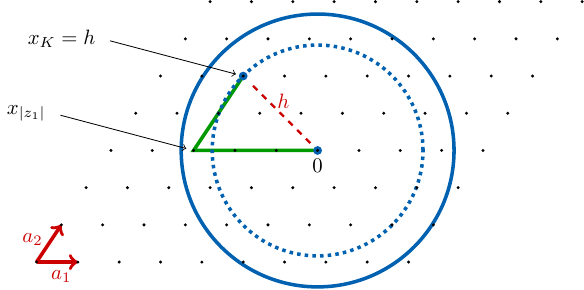}
 \end{center}
\caption{Sketch of the geometry in the construction of Lemma \ref{lemmaccorerad}. The dashed red line represents the segment joining $x_0=0$ with $x_K=h$, the green line represents the polygonal $P$, the blue circle has radius $\ccorerad$. The dotted circle shows that, at least with this construction, we cannot take $\ccorerad=\max\{|h|: h\in\calN\}$.}
\label{fig-lattice2}
\end{figure}%
\begin{lemma}\label{lemmaccorerad}
There is ${\tiny \ccorerad}\geq d_\calC$, depending only on $\calL$, $\calN$, $\calC$, and $\{T_0, \dots, T_{N_T}\}$, such that for all $h\in \calN$ and $\eps>0$ there is an elementary $\eps$-discrete path $P=(x_0, x_1, \dots, x_{K-1},x_K)$ which starts at $x_0=0$, ends at $x_K=\eps h$, and satisfies $|x_j|< \ccorerad\eps$ for all $j$.
\end{lemma}
\begin{proof} By scaling, it suffices to consider the case $\eps=1$. We set $A:=\sum_{i=1}^n a_i\otimes e_i$, where the vectors $a_i$ are such that $\vertici(T_0)=\{0,a_1,\dots ,a_n\}$. For $h\in\calN$ we set $z:=A^{-1}h\in\Z^n$. We let $P_1$ be an elementary path joining $x$ with $x+ z_1 A e_1$, which consists of $|z_1|$ steps all along $ A e_1$ (and $|z_1|+1$ vertices). Analogously, $P_2$ joins $x+ A z_1e_1$ with $x+ A (z_1e_1+z_2e_2)$, and so on. The requisite path is then $P:=P_1\oplus P_2\oplus\dots\oplus P_n$, see Figure \ref{fig-lattice2}. The distance from $x$ of every point in the path is then bounded by $|A||z|\le |A|\, |A^{-1}|\,|h|$, so that the statement holds with $\ccorerad:=\max\{1+|A|\,|A^{-1}|\, \max\{|h|: h\in\calN\},d_\calC\}$.
\end{proof}

\begin{definition}\label{definitioncore}
Let $\ccorerad$ be the constant from Lemma \ref{lemmaccorerad}. Given $\discbeta:\bonds_\eps^\Omega\to\R^n$, we define the \emph{core region} $\Ce(\discbeta,\Omega)$ as the set of $x\in \Omega_{\ccorerad\eps}$ such that there is an $\eps$-discrete closed path $P$ (in the sense of Definition \ref{definitioncurlfree}) contained in $\eps\calL\cap B_{\ccorerad\eps}(x)$ with $\circu(\discbeta, P)\ne 0$. If the domain is clear from the context we simply write $\Ce(\xi)$.
\end{definition}

In some simple cases, the core can alternatively be defined as the union of the lattice tetrahedra with nonzero circulation. For example, in two dimensions, if $\calL=\Z a_1+\Z a_2$ and $\calN=\{\pm a_1,\pm a_2, \pm (a_1+a_2)\}$, then a simpler definition of $\Ce$ can be taken as the union of lattice triangles of the type $\conv\{x,x+a_1,x+a_1+a_2\}$ or $\conv\{x,x+a_2,x+a_1+a_2\}$ such that the circulation around the boundary is nonzero, as was done in \cite{AlicandroDelucaGarroniPonsiglione2014}. This simple procedure, however, does not generalize to more complex lattices. Indeed, if $\calN$ is replaced by the more symmetric version $\calN = \{\pm a_1, \pm a_2, \pm (a_1+a_2), \pm(a_1-a_2)\}$, then the simple procedure above fails to account for $\discbeta(\cdot, a_1-a_2)$.

\begin{remark}\label{remarkexactsimplex}
Let $\discbeta:\bonds_\eps^\Omega\to\R^n$ and $T\in \tetr_\eps^\Omega$ with $B_{\ccorerad\eps}(T)\subseteq\Omega$. If $T$ is not contained in $\Ce(\discbeta)$, then $\discbeta$ is exact on $T$. To see this, let $x\in T$, which implies $B_{\ccorerad \eps}(x)\subseteq\Omega$, and let $P$ be a closed path in $\vertici(T)$. As $\ccorerad\eps>\diam(T)$, we have $P\subseteq B_{\ccorerad\eps}(x)$. Therefore, $\circu(\discbeta,P)\ne0$ implies $x\in \Ce(\discbeta)$. Lemma \ref{lemmaelemenmtarycontinuouscirc} \ref{lemmaoutofcoreLb} extends this observation to appropriate balls.
\end{remark}

 \begin{lemma}\label{lemmaelemenmtarycontinuouscirc}
Let $\Omega\subseteq\R^n$ be open,
 $\discbeta:\bonds_\eps^\Omega\to\R^n$, $\omega:=\Omega_{2\ccorerad \eps}\setminus {\Ce(\discbeta)}$,
 $\omega':=\{x: B_{\ccorerad \eps}(x)\subseteq\omega\}$.

 \begin{enumerate}
\item\label{lemmaelemenmtarycontinuouscirci}  Let $P=(x_1, \dots, x_K)$ be a closed elementary $\eps$-discrete path, $P\subseteq \omega$. Then $\circu(\discbeta, P)=\int_{\gamma(P)} L\discbeta \tau_\gamma d\calH^1$, where $\gamma(P)$ is the union of the segments $[x_k, x_{k+1}]$ and $\tau_\gamma:\gamma(P)\to S^{n-1}$ denotes the unit tangent vector, oriented on each segment from $x_k$ to $x_{k+1}$.

\item\label{lemmaelemenmtarycontinuouscirccurl}  
$DL\discbeta=(DL\discbeta)^T$ as distributions on $\interior(\omega)$.

\item\label{lemmaelementarycirc} 
For every closed curve $\gamma\subseteq\omega'$, $$
\int_{\gamma} L\discbeta \tau_\gamma d\calH^1\in \eps\calB.
$$

\item\label{lemmaelemenmtarycontinuouscircinonel}  
Let $P=(x_1, \dots, x_K)$ be a closed $\eps$-discrete path, $P\subseteq \omega'$. Then $\circu(\discbeta, P)=\int_{\gamma(P)} L\discbeta \tau_\gamma d\calH^1$.

\item \label{lemmaoutofcorecircu}\label{lemmaoutofcoreLb}
Let $x_0$, $r>0$ be such that $B_{r+3\ccorerad\eps }(x_0)\subseteq\Omega\setminus \Ce(\discbeta,\Omega)$. Then, $\discbeta$ is exact on $\eps\calL\cap B_r(x_0)$.
 \end{enumerate}

\end{lemma}
The trace of the tangential component $L\discbeta\tau$ is well defined by Remark \ref{remark-Lcsiexact}.
\begin{proof}
\ref{lemmaelemenmtarycontinuouscirci}:
We observe that, if $T\in \calT^{\R^n}_\eps$ and $T\cap \omega\ne\emptyset$, then $B_{\ccorerad\eps}(T)\subseteq\Omega$, so that by Remark~\ref{remarkexactsimplex} $\xi$ is exact on $T$. Consider any pair $(x_k, x_{k+1})$ in the elementary path $P$ and any simplex $T\in \calT^\Omega_\eps$ with $(x_k, x_{k+1})\in \edges(T)$. Since $\discbeta$ is exact on this simplex, by Remark~\ref{remark-Lcsiexact} we have
\begin{equation}
 \discbeta(x_k, h)=L\discbeta|_T h   \hskip1cm 
\text{ with } h=\frac{x_{k+1}-x_k}\eps.
\end{equation}
Therefore,
\begin{equation}
    \int_{[x_k, x_{k+1}]} L\discbeta \tau_\gamma d\calH^1 
    = 
    L\discbeta(x_{k+1}-x_k) = \eps \discbeta(x_k, \frac{x_{k+1}-x_k}\eps).
\end{equation}
The assertion follows summing over all segments.

\ref{lemmaelemenmtarycontinuouscirccurl}:
Consider two simplices, say, $T_1$ and $T_2$, which share an $n-1$-dimensional face, which in turn intersects $\omega$. As above, this implies that $\xi$ is exact on both. For every common edge $(x,x+\eps h)$ we have that $L\discbeta|_{T_1}h = L\discbeta|_{T_2}h = \discbeta(x,h)$ (see Remark~\ref{remark-Lcsiexact}). Therefore, the tangential component of $L\discbeta$ on the common face $T_1\cap T_2$ has the same value on the two sides and the proof is concluded as in the second part of the proof of Proposition~\ref{prop-Lcsi}.

\ref{lemmaelementarycirc}: 
We first observe that the integral is well defined. Indeed, for any $T\in \calT^\Omega_\eps$ and $\calH^1$-almost any $x\in  \gamma\cap\partial T$ we have that $\tau_\gamma(x)$ is tangent to $\partial T$ in $x$ (in the sense that it belongs to each of the planes spanned by the faces of $\partial T$ which contain $x$). As shown in proving \ref{lemmaelemenmtarycontinuouscirccurl}, this implies that the tangential component $L\xi\tau_\gamma$ is well-defined $\calH^1$-almost everywhere on $\gamma$.

If $\gamma$ is an elementary path as in \ref{lemmaelemenmtarycontinuouscirci}, the result is immediate from Remark~\ref{remarkcircinL}. Otherwise, $\gamma$ is deformed into an elementary path (for example, projecting on the edges of the tetrahedra it intersects) and  \ref{lemmaelemenmtarycontinuouscirccurl} is used.

\ref{lemmaelemenmtarycontinuouscircinonel}:
We construct an elementary path $P'$ with the same circulation. For every $i\in \{1,\dots, K-1\}$, let $h_i\in\calN$ be such that $x_{i+1}=x_i+\eps h_i$. By Lemma \ref{lemmaccorerad} there is an elementary 
$\eps$-discrete path $P_{i}$ joining $x_i$ with $x_{i+1}$ and contained in $\eps\calL \cap B_{\ccorerad\eps}(x_i)$. We observe that $\circu(\discbeta, P_i\oplus (x_{i+1},x_i ))=0$. Indeed, the path $P_i\oplus (x_{i+1},x_i )$ is contained in the ball
$B_{\ccorerad\eps}(x_i)\subseteq \Omega$. If $\circu(\discbeta, P_i\oplus (x_{i+1},x_i ))\ne 0$, then $x_i$ is a core point by definition.

We define $P'$ as the path obtained from $P$ by replacing every edge $(x_i,x_{i+1})$ by $P_i$. Then $\circu(\discbeta,P)=\circu(\discbeta,P')$, $P'$ is an elementary path, and $P'\subseteq \omega$. By \ref{lemmaelemenmtarycontinuouscirci}, we have $\circu(\discbeta,P')=\int_{\gamma(P')} L\discbeta \tau_{P'} d\calH^1$. By \ref{lemmaelemenmtarycontinuouscirccurl}, we have $D L\xi=(DL\xi)^T$ on $\omega$. By Stokes theorem, $\int_{\gamma(P')} L\discbeta \tau_{P'} d\calH^1= \int_{\gamma(P)} L\discbeta \tau_{P} d\calH^1$, and \ref{lemmaelemenmtarycontinuouscircinonel} follows.

\ref{lemmaoutofcorecircu}: We have to show that, for every $\eps$-discrete closed path $P\subseteq B_r(x_0)$, $\circu(\discbeta,P)=0$. The assertion follows from \ref{lemmaelemenmtarycontinuouscircinonel} and \ref{lemmaelemenmtarycontinuouscirccurl} using $B_r(x_0)\subseteq \omega'$ and the fact that $B_r(x_0)$ is simply connected.
\end{proof}

\subsection{Configurations with dilute dislocations in $\R^3$}

Working in $\R^3$, dislocations, here interpreted in the sense of the set $\Ce(\discbeta)$, are concentrated along one-dimensional sets. The diluteness condition is then phrased naturally in terms of the geometry of this set. To this end, following \cite{ContiGarroniOrtiz2015} it is convenient to introduce the class of dilute dislocations identified with a class of divergence-free measures concentrated on polyhedral curves. 

\def\segment{s}
\begin{definition}\label{defdiluteness-curve}
Let $\Omega\subseteq \R^3$ be open and bounded. Given two positive parameters $\alpha,k > 0$, we say that a  polyhedral curve $\gamma\subseteq\Omega$ is $(k,\alpha)$-{\em dilute} if it is the union of finitely-many closed segments $\segment_j\subseteq\overline\Omega$  such that
	\begin{enumerate}
\item every $\segment_j$ has length at least $k$;
\item if $\segment_j$ and $\segment_i$ are disjoint, then their distance is at least $\alpha k$;
\item if the segments $\segment_j$ and $\segment_i$ are not disjoint, then they share an endpoint and the angle between them is at least $\alpha$;
		\item $\gamma$ does not have end points inside $\Omega$.
	\end{enumerate}
The set of $(k,\alpha)$-dilute polyhedral curves is denoted by $P(\Omega, k,\alpha)$.
\end{definition}

Next, we define the set of admissible configurations with dilute dislocation distributions.
\begin{definition}\label{def-Ade}
Let $\Omega\subseteq \R^3$ be open and bounded. For every $\eps>0$, $k>0$, $\alpha>0$, $m\ge \ccorerad$, we denote by $\Ade(\Omega,k,\alpha,m)$ the set of pairs $(\gamma,\xi)$ satisfying:
\begin{enumerate}
\item $\discbeta:\bonds_\eps^\Omega\to\R^n$ is admissible and $\discbeta=d_\eps u-\discbetapl$, where $u$ is a lattice deformation and $\discbetapl$ a discrete plastic strain in the sense of Definition~\ref{defbetapl};
\item\label{def-Ade-core} $\gamma\in P(\Omega, k,\alpha)$ with $\Ce(\discbeta)\subseteq B_{m\eps}(\gamma)$.
\end{enumerate} 
\end{definition}

The asymptotic analysis will be performed assuming that  the diluteness parameters $k$ and $\alpha$ are much larger than the lattice spacing $\eps$, in the sense that
\begin{equation}\label{eqdefheps}
    \lim_{\eps\to0} \frac{\log (1/(\alpha_\eps k_\eps))}{\log (1/\eps)}
    =
    \lim_{\eps\to0} \alpha_\eps
    =
\lim_{\eps\to0} k_\eps
    =
    0 \,.
\end{equation}
The core radius $m$ in Definition~\ref{def-Ade}\ref{def-Ade-core} will be kept fixed. We remark that our analysis holds, with minor changes in the proof, if $m=m_\eps\to\infty$ sufficiently slowly. Given $k_\eps$ and $\alpha_\eps$ that obey \eqref{eqdefheps} and $m\ge k_*$, we shall briefly write
\begin{equation}\label{eq-Ade-short}
    \Ade(\Omega)
    :=
    \{\xi\, : (\gamma,\xi)\in \Ade(\Omega,k_\eps,\alpha_\eps,m) \hbox{ for some } \gamma\in P(\Omega,k_\eps,\alpha_\eps)\}.
\end{equation}

\begin{remark}\label{rem-dilute-gamma}
Using \eqref{eqdefheps}  we see that any sequence $\gamma_\eps$ in $P(\Omega, k_\eps,\alpha_\eps)$ satisfies
\begin{equation} \label{eq-H1-gamma}
	\lim_{\eps\to0}	\calH^1(\gamma_\eps)\eps^s=0 ,
	\end{equation}
for every $s>0$. It is indeed enough to observe that if $\gamma_\eps\in P(\Omega, k_\eps,\alpha_\eps)$ we can construct $N_\eps$ disjoint cylinders of height $k_\eps/2$ and radius $\alpha_\eps^2 k_\eps$ whose axis are contained in $\gamma_\eps$. The total volume of these cylinders cannot exceed the volume of $\Omega$, so that 
	$$
	\calH^1(\gamma_\eps)
    \leq 
    N_\eps \diam(\Omega)
    \leq 
    C\frac{1}{\alpha_\eps^4k_\eps^3} ,
	$$
	which gives \eqref{eq-H1-gamma}.
\end{remark}

The diluteness condition permits to estimate the curl of $L\xi$ in terms of the energy and the length of the curve $\gamma$.

\begin{lemma}\label{lemmacoremollif}
Let $\Omega\subseteq\R^3$ be open and bounded. Let $(\gamma,\discbeta)\in \Ade(\Omega,k,\alpha,m)$ for some $\alpha\in(0,1]$, $m\in[\ccorerad,\infty)$, $k\in(0,\infty)$. Assume $ m\eps\le k$ and let $\Omega'\subseteq\Omega$ be open, with $\dist(\Omega',\partial\Omega)\ge 2m\eps$. 	Then,
	\begin{equation}
		\begin{split}
    |\curl L\discbeta|(\Omega')
    \le 
    c m
    (\calH^1(\gamma\cap B_{2m\eps}(\Omega')))^{1/2}  
    E_\eps^{1/2}[\discbeta,\Omega] .
		\end{split}
	\end{equation}
\end{lemma}
\begin{proof}
Denote $\delta:=m\eps\in (0,k]$. We recall that $\curl L\discbeta$ is a measure concentrated on the faces of the tetrahedra in $\mathcal T_\eps^{\R^n}$ and absolutely continuous with respect to $\calH^2$. By $\Ce(\xi)\subseteq B_\delta(\gamma)$ and Lemma \ref{lemmaelemenmtarycontinuouscirc}\ref{lemmaelemenmtarycontinuouscirccurl}, $\curl L\discbeta=0$ on $\Omega'\setminus \overline{B_\delta(\gamma)}$, so that we have only to deal with the set 
\begin{equation}\label{eqomegabdeltaomega}
    \omega
    :=
    \Omega'\cap \overline{B_\delta(\gamma)}
    \subseteq
\Omega'\cap \overline{B_\delta}(\gamma\cap B_\delta(\Omega')). 
\end{equation}
Pick two tetrahedra $T\ne T'\in\mathcal T_\eps^{\R^n}$ that share a face and such that $T\cap T'\cap \omega\ne\emptyset$. By Proposition \ref{prop-Lcsi},
\begin{equation}
\begin{split}
    |\curl L\discbeta|(T\cap T')
    \le& c\eps^2 (\EC^{1/2}[\tau^z_\eps\discbeta]+
    \EC^{1/2}[\tau^{z'}_\eps\discbeta]),
\end{split}
\end{equation}
where $z$, $z'\in \eps\calL$ are such that $T\subseteq z+\eps T_*$, $T'\subseteq z'+\eps T_*$. As $m\ge k_*\ge d_\calC\ge d_\calT$, this implies $T\subseteq B_{\delta}(\omega)$, $z\in \clust^\Omega_\eps$, and the same for $T'$ and $z'$. By H\"older's inequality and 
$\eps^3\le c |T|$, we obtain
\begin{equation}
\begin{split}
\sum_{T\ne T': T\cap T'\cap\omega\ne\emptyset}
|\curl L\discbeta|(T\cap T')
    \le & 
    c\eps^{-1} \left(\sum_{T\subseteq B_{\delta}(\omega)} 
    |T|\right)^{1/2} \left(\eps^3\sum_{z\in \clust^\Omega_\eps} \EC[\tau^z_\eps\discbeta]\right)^{1/2} 
    \\ \le & 
    c\eps^{-1} |B_{\delta}(\omega)|^{1/2} E_\eps^{1/2}[\xi, \Omega].
\end{split}
\end{equation}
It remains to estimate the measure of $B_{\delta}(\omega)$. For every segment $s$ with $\calH^1(s)\ge\delta$, $|B_{2\delta}(s)|=\frac{4\pi}3(2\delta)^3+\pi  (2\delta)^2 \calH^1(s)\le c \delta^2\calH^1(s)$. The set $\gamma\cap B_\delta(\Omega')$ entering \eqref{eqomegabdeltaomega} is contained within a finite union of segments, all belonging to $\gamma \cap B_{2\delta}(\Omega')$ and having length at least $\delta$ (we use here the definition of diluteness and $\delta\le k$). Summing over all these segments,
\begin{equation}
    |B_\delta(\omega)|
    \le
|B_\delta(\Omega')
\cap
B_{2\delta}(\gamma\cap B_\delta(\Omega'))|
    \le 
    c \delta^2\calH^1(\gamma \cap B_{2\delta}(\Omega')).
\end{equation}
Combining the previous two equations concludes the proof.
\end{proof}

\section{Local continuum lower bound}
We proceed to study the relation of the discrete energy to a continuum elastic energy in a region away from the core, where the deformation $\xi$ is locally exact, and hence can, locally, be seen as originating from the discrete deformation gradient of a displacement field $u$. In Section~\ref{subsec-moll}, we start by introducing two different interpolation schemes, which allow to pass locally from discrete displacements and deformations to corresponding continuum fields. The key result is Theorem~\ref{theomollifydisccont}, which provides a lower bound on the discrete energy in terms of a corresponding continuum energy, and is presented and proved in Section~\ref{secelastmollif}. The result is based on mollification and convexity of the cluster energy, and uses only marginally the fact that it is quadratic. It will be a crucial ingredient in proving the lower bound of the $\Gamma$-convergence result.

\label{seclocalcontinuum}
\subsection{Interpolations and mollifications}\label{subsec-moll}
In order to relate discrete and continuum deformations we shall use different interpolation schemes.  They are all based on the tetrahedra introduced in Section \ref{seckinematics}. Different quantities have different natural interpolations. For quantities that are defined on vertices, such as the displacement $u$, it is natural to use piecewise constant or piecewise affine interpolations.

\begin{definition}\label{definterpolation}
Let $u:A\subseteq\eps\calL\to\R^n$. The piecewise affine interpolation $I_\eps u$ is obtained setting $I_\eps u=u$ on $A$ and $I_\eps u$ affine on each $T\in \calT^A_\eps$.

For every map $v:A\subseteq\eps \calL\to X$ we define the piecewise constant interpolation $J_\eps v: \cup_{x\in A} \interior(x+\eps T_*)\to X$ by setting, for every
$x\in A$, $J_\eps v:=v(x)$ on $\interior(x+\eps  T_*)$.
\end{definition}

We remark that the piecewise affine interpolation $I_\eps$ will only be used for deformations $u$, while the piecewise constant interpolation $J_\eps$ can be used for any lattice function, including in particular the deformation $u$,  the strain $\discbeta(\cdot, h)$, and the slip $\discxi_i(\cdot,h)$. It defines a map in $L^\infty_\loc$. As the mesh is conforming (Definition~\ref{def-tetrahedra}), the function $I_\eps u$ is continuous.

\newcommand\Plat{P_{\eps\calL}}
\newcommand\Pt{P_{\eps T_*}}
\begin{remark}\label{remdepsucont}
We can extend the discrete differential $d_\eps$ to functions defined on generic sets $E\subseteq \R^n$ (see \eqref{eqdefdepsu} in Definition~\ref{defdepsu}), setting
\begin{equation}
    d_\eps f(x,h):=\frac{f(x+\eps h)-f(x)}{\eps}
\end{equation}
for any $\eps>0$, $x\in \R^n$, $h\in \calN$ such that $x$ and $x+\eps h$ are in $E$. Then, we obtain 
\begin{equation}\label{depsjepscommute}
    d_\eps J_\eps f = J_\eps d_\eps f
\end{equation}
$\calL^n$-almost everywhere on the set where the two are defined. 

To see this, let $\Plat:\R^n\to\eps \calL$ and $\Pt:\R^n\to\eps T_*$ be such that $x=\Plat (x) + \Pt (x)$ for any $x\in\R^n$ (this decomposition is unique outside the union of the boundaries of the $\eps T_*$, which is a null set), so that $J_\eps f(x)=f(\Plat(x))$ almost everywhere. Then, for almost every $x$, 
\begin{equation}
    (J_\eps d_\eps f)(x,h)
    = 
    d_\eps f (\Plat (x), h)
    =
    \frac{ f(\Plat(x)+\eps h)- f(\Plat(x))}{\eps}
\end{equation}
and, since $h\in\calN\subset\calL$ implies $\Plat(x+\eps h)=\Plat(x)+\eps h$,
\begin{equation}
 ( d_\eps J_\eps f)(x,h)=\frac{ (J_\eps f)(x+\eps h)- (J_\eps f)(x)}{\eps}
 =\frac{ f(\Plat(x)+\eps h)- f(\Plat(x))}{\eps}.
\end{equation}
This concludes the proof of \eqref{depsjepscommute}.
\end{remark}

We first show how coercivity permits to control the symmetrized gradient of the interpolation.

\begin{remark}\label{lemmacompactclusterkorn}
We observe that the reference energy $\EC^0$ was defined so that it controls the symmetric part of the strain of the piecewise affine interpolation $I_\eps u$. Specifically, for $u:\calC\to\R^n$, 
{and shortening $I:=I_1$,}
\begin{equation}
    \int_{T_*} |DIu+(DIu)^T|^2 dx \le c \EC^0[du].
\end{equation}
To see this, it suffices to prove the assertion for a single $T_i$ (the argument is similar to the one in Proposition~\ref{prop-Lcsi}).
Let $\zeta_i=DIu|_{T_i}\in\R^{n\times n}$. For every $F\in\R^{n\times n}$ and every $\xi:\calC^0_\calN\to\R^n$, we have
\begin{equation*}
   |\discbeta-\discbeta^F|^2
   =
   \sum_{(x,h)\in\calC^0_\calN} |\discbeta(x,h)-Fh|^2 .
\end{equation*}
Let $x,y\in \vertici(T_i)\subseteq \calC^0$, $x\ne y$. Then
$du(x,y-x)=\zeta_i(y-x)$. Therefore,
 \begin{equation*}
   |du-\discbeta^F|^2\ge \sum_{(x,y)\in \edges(T_i)} |\zeta_i(y-x)-F(y-x)|^2
   \ge c |\zeta_i-F|^2 ,
 \end{equation*}
where 
the constant depends only on the shape of $T_i$. 
Recalling \eqref{eqdefE0C}, 
\begin{equation*}
\EC^0[du]\ge c \min_{S\in \R^{n\times n}_\skw} |\zeta_i-S|^2=
    c |\zeta_i+\zeta_i^T|^2.
 \end{equation*}
\end{remark}

\begin{remark}\label{lemmadiffiepsjepsu}
For functions whose gradients are controlled the two interpolations are close. Specifically, for $\omega\subset\subset\Omega$ with $B_{d_{T_*}\eps/2}(\omega) \subseteq \Omega$,
\begin{equation}
    \|I_\eps u - J_\eps u\|_{L^2(\omega)}^2 
    \le 
    c \eps^2\| DI_\eps u\|_{L^2(\Omega)}^2.
\end{equation}
To see this, consider a set $x+\eps T_*\subseteq\Omega$, $x\in\eps \calL$. By Poincar\'e's inequality,
\begin{equation*}
    \int_{x+\eps T_*} |I_\eps u-J_\eps u|^2dy 
    = 
    \int_{x+\eps T_*} |I_\eps u(y)- I_\eps u(x)|^2dy 
    \le 
    c \eps^2 \int_{x+\eps T_*} |DI_\eps|^2dy,
\end{equation*}
with a constant that depends only on $T_*$. Summing over all simplices proves the assertion.
\end{remark}

\begin{lemma}\label{lemmacompactclusterkornbetaballdelta}
Let $\delta\ge d_{\calC}\eps$, $x_0\in\eps\calL$, $u:\eps\calL\cap  B_{2\delta}(x_0)\to\R^n$. Then, there is a matrix $S\in\R^{n\times n}_\skw$ such that
\begin{equation}\label{lemmacompactdelta1}
    \int_{B_{\frac32\delta}(x_0)} |DI_\eps u-S|^2 dx 
    \le 
    c E_\eps[d_\eps u, B_{2\delta}(x_0)]
\end{equation}
and there is $b\in\R^n$ such that the function $\hat u(y):=u(y)-Sy-b$ obeys
\begin{equation}\label{lemmacompactdelta2}
    \int_{B_{\delta}(x_0)} |J_\eps \hat u|^2 dx 
    \le 
    c \delta^2 E_\eps[d_\eps u, B_{2\delta}(x_0)]
\end{equation}
and
\begin{equation}\label{lemmacompactdelta3}
    \int_{B_{\delta}(x_0)} |I_\eps\hat u - J_\eps \hat u|^2 dx 
    \le 
    c \eps^2 E_\eps[d_\eps u, B_{2\delta}(x_0)].
\end{equation}
\end{lemma}
\begin{proof}
For notational simplicity, we work in the case $x_0=0$. We first observe that $\frac12\delta\ge \eps \diam \calC$ implies
\begin{equation}
    B_{\frac32\delta}
    \subseteq 
    \bigcup_{x\in \clust^{B_{2\delta}}_\eps} x+\eps T_*.
\end{equation}
Using Remark \ref{lemmacompactclusterkorn} on each $x+\eps T_*$ with $x\in \clust^{B_{2\delta}}_\eps$ and summing over $x$,
\begin{equation}\label{eqDIepsDIepsT}
\begin{split}
    \int_{B_{\frac32\delta}} |DI_\eps u+(DI_\eps u)^T|^2 dx 
    & \le 
    c \sum_{x\in \clust_\eps^{B_{2\delta}}} \eps^n \EC^0[\tau^x_\eps du]
    \le 
    c M ,
\end{split}
\end{equation}
where in the last step we use \eqref{eqasscoerc} and define $M:=E_\eps[d_\eps u,B_{2\delta}]$. By the Korn-Poincar\'e inequality, there is an affine infinitesimal isometry $y\mapsto Sy+b$ such that the function $\hat u(y) : = u(y) - Sy - b$ obeys
\begin{equation}\label{eqcompactkornpo}
    \delta^{-2}\int_{B_{\frac32\delta}} |I_\eps \hat u|^2 dx 
    +
    \int_{B_{\frac32\delta}} |DI_\eps \hat u|^2 dx 
    \le 
    c M,
\end{equation}
which in particular proves \eqref{lemmacompactdelta1}. By Remark \ref{lemmadiffiepsjepsu},
\begin{equation}
    \int_{B_{\delta}} |J_\eps \hat u-I_\eps\hat u|^2 dx 
    \le
    c \eps^2 \int_{B_{\frac32\delta}} |DI_\eps \hat u|^2 dx
\end{equation}
and, with \eqref{eqcompactkornpo}, this proves \eqref{lemmacompactdelta3}. Finally,
\begin{equation}
\begin{split}
    \int_{B_{\delta}} |J_\eps \hat u|^2 dx 
    & \le
    2 \int_{B_{\delta}} |J_\eps \hat u-I_\eps \hat u|^2 dx 
    +
    2\int_{B_{\delta}} |I_\eps \hat u|^2 dx 
    \\ & \le 
    c \eps^2
    \int_{B_{\frac32\delta}} |DI_\eps \hat u|^2 dx 
    +
    2\int_{B_{\delta}} |I_\eps \hat u|^2 dx
    \le 
    c\delta^2 M ,
\end{split}
\end{equation}
which proves \eqref{lemmacompactdelta2}.
\end{proof}

It will be useful to consider a mollification of the interpolations. Fix a mollifier $\psi_\delta\in C_c^\infty(B_{\delta/2};[0,\infty))$ with 
\begin{equation}\label{eqpsideltamoll}
    \psi_\delta(y)
    =
    \psi_\delta(-y), \hskip5mm \int_{\R^n}\psi_\delta\,dx
    =
    1,\hskip5mm \psi_\delta(y)=\delta^{-n}\psi_1(y/\delta). 
\end{equation}
For a generic lattice function $v:\eps\calL\to \R^M$ we consider the mollification of the piecewise constant interpolation $v_\delta:=\psi_\delta \ast J_\eps v$, and observe that for $x\in\eps\calL$ we have
\begin{equation}
    v_\delta(x)=\sum_{y\in \eps \calL} v(x-y) \psi_\delta^\eps(y) ,
\end{equation}
where $\psi_\delta^\eps:\eps\calL\to\R$ is defined as
\begin{equation}\label{eqdefpsideltaeps}
    \psi_\delta^\eps(y):=\int_{\eps T_*}\psi_\delta(y-z)dz.
\end{equation}
This implies $\sum_{x\in \eps \calL} \psi^\eps_\delta(x)=1$, $\psi^\eps_\delta \ge 0$, and $\psi^\eps_\delta=0$ outside $B_\delta$ whenever $d_{T_*}\eps\le\delta$.

\begin{remark}\label{remarkLbetaLbetadelta}
From \eqref{lemmacompactdelta1} in the proof of Lemma \ref{lemmacompactclusterkornbetaballdelta} we obtain, for $\delta\ge d_\calC\eps$,
\begin{equation}\label{lemmacompactdeltamoll}
    \int_{B_{\delta}(x_0)} |DI_\eps u - \psi_\delta\ast DI_\eps u|^2 dx 
    \le 
    c E_\eps[d_\eps u, B_{2\delta}(x_0)].
\end{equation}
Equivalently, if  $\discbeta$ is exact on $B_{2\delta}(x_0)$, then
\begin{equation}\label{lemmacompactdeltamollbeta}
    \int_{B_{\delta}(x_0)} |L\discbeta -\psi_\delta\ast L\discbeta |^2 dx 
    \le 
    c E_\eps[\discbeta, B_{2\delta}(x_0)].
\end{equation}
\end{remark}

\subsection{Elastic energy of the mollification}
\label{secelastmollif}
Next, we study the asymptotics of the discrete energy for deformation fields which are locally exact. The estimate proved in Theorem~\ref{theomollifydisccont} will be important in order to estimate the elastic energy away from the core region.

\begin{theorem}\label{theomollifydisccont}
Let $\omega\subseteq\R^n$ be an open set, $\omega'\subset\subset\omega$, and assume $\dist(\omega',\partial\omega)\ge 5\delta$, with $\delta\ge 3k_*\eps$. For $\xi:\bonds^\omega_\eps\to\R^n$ let $F_\delta:\omega'\to\R^{n\times n}$ be defined by
\begin{equation}\label{eqdefFdelta}
    F_\delta:= \psi_\delta \ast L\discbeta.
\end{equation}
If $\Ce(\xi,\omega)=\emptyset$, then
\begin{equation}
    \int_{\omega'} \frac12 \C F_\delta \cdot F_\delta \, dx 
    \le 
    (1+c\frac{\eps}{\delta}) E_\eps[\discbeta, \omega].
\end{equation}
The constant $c$ depends on the lattice, the cluster, and the cluster energy, but not on $\eps$, $\delta$, $\omega$, $\omega'$ and $\discbeta$.
\end{theorem}
\begin{proof}
Let $\omega'':=\{x\in \omega: B_{\delta}(x)\subseteq\omega\}$. We claim that 
\begin{equation}\label{eqpbetadconvd}
    E_\eps[\discbeta_\delta,\omega''] \le E_\eps[\discbeta, \omega],
\end{equation}
where the mollified lattice strains $\discbeta_\delta : \bonds_\eps^{\omega''} \to \R^n$ are defined by
\begin{equation}
    \discbeta_\delta(x,h)
    :=
    \sum_{z\in \eps\calL} \psi_\delta^\eps(z)\discbeta(x-z,h),
\end{equation}
the sum over $z$ is implicitly restricted to the points where $\psi_\delta^\eps$ does not vanish. This ensures that $\xi(x-z,h)$ is well defined.

To prove~\eqref{eqpbetadconvd}, we recall that, by definition of $E_\eps$ and $\discbeta_\delta$, 
\begin{equation*}
    E_\eps[\discbeta_\delta,\omega'']
    =
    \sum_{x\in \clust_\eps^{\omega''}} \eps^n \EC[\tau^x_\eps\discbeta_\delta]
    =
    \sum_{x\in \clust_\eps^{\omega''}} \eps^n \EC[\sum_{z\in\eps\calL} \psi_\delta^\eps(z)\tau_\eps^{x-z}\discbeta].
\end{equation*}
Since $\EC$ is convex, $\psi_\delta^\eps\ge0$ and $\sum_{z\in \eps \calL} \psi_\delta^\eps(z) = 1$ with Jensen's inequality we obtain
\begin{equation*}
\begin{split}
    E_\eps[\discbeta_\delta,\omega'']
    &
    \le 
    \sum_{x\in \clust_\eps^{\omega''}} \eps^n \sum_{z\in\eps\calL} \psi_\delta^\eps(z) \EC[\tau_\eps^{x-z}\discbeta]
    \\ & = 
    \sum_{z\in\eps\calL\cap B_\delta} \psi_\delta^\eps(z)
    \sum_{x\in \clust_\eps^{\omega''-z}} \eps^n \EC[\tau_\eps^{x}\discbeta]
    \le 
    \sum_{x\in \clust_\eps^{\omega}} \eps^n \EC[\tau_\eps^{x}\discbeta]=E_\eps[\discbeta,\omega] ,
\end{split}
\end{equation*}
which proves \eqref{eqpbetadconvd}.

We fix $x_0\in\eps \calL$ such that $(x_0+\eps T_*)\cap\omega'\ne\emptyset$, which implies $B_{4\delta}(x_0)\subseteq\omega$ and $x_0\in  \clust_\eps^{\omega''}$. Since $\Ce(\xi,\omega)=\emptyset$, by Lemma~\ref{lemmaelemenmtarycontinuouscirc}\ref{lemmaoutofcorecircu} with $r=3\delta$ we obtain that $\discbeta$ is exact on $B_{3\delta}(x_0)$, and by Lemma~\ref{lemmadiscretepath} there is $u:\eps\calL\cap B_{3\delta}(x_0)\to\R^n$ 
such that $\discbeta=d_\eps u$ as functions on $\bonds^{B_{3\delta}(x_0)}_\eps$. By Lemma \ref{lemmacompactclusterkornbetaballdelta}, there is a matrix $S\in\R^{n\times n}_\skw$ and a vector $b\in\R^n$ such that, defining $\hat u(y):=u(y)-Sy-b$, 
\begin{equation}\label{eqjepsuhat}
    \|J_\eps \hat u\|_{L^2(B_{\frac32\delta}(x_0))}^2
    \le 
    c \delta^2 E_\eps[d_\eps u, B_{3\delta}(x_0)].
\end{equation}
We observe that, since $\psi_\delta$ is even, $(\psi_\delta\ast I_\eps \hat u)(y)= (\psi_\delta\ast I_\eps  u)(y) - Sy -b$. We set $\hat \discbeta(x,h) := \discbeta(x,h)  -Sh$ and observe that by \eqref{eqclusterlinearrotations} we have $E_\eps[\hat\discbeta,\cdot] = E_\eps[\discbeta,\cdot]$, $\hat\discbeta_\delta(x,h) := (\psi_\delta\ast J_\eps\hat \discbeta)(x,h) = \discbeta_\delta(x,h)-Sh$, $E_\eps[\hat\discbeta_\delta,\cdot] = E_\eps[\discbeta_\delta,\cdot]$.

We define $\hat u_\delta:B_{2\delta}(x_0)\to\R^n$ by $\hat u_\delta := \psi_\delta\ast J_\eps \hat u$. By the linearity of the mollification, $\hat\discbeta_\delta=d_\eps \hat u_\delta$ on
$\bonds^{B_{2\delta}(x_0)}_\eps$. Indeed, for $x\in B_{2\delta}(x_0)\cap\eps\calL$, $h\in\calN$,
\begin{equation}\label{eqbetadeltaduepsdelta}
\begin{split}
    \hat\discbeta_\delta(x,h)
    & =
    (\psi_\delta\ast J_\eps \hat\discbeta(\cdot, h))(x)
    =
    (\psi_\delta\ast J_\eps d_\eps \hat u(\cdot, h))(x)
    \\ & = 
    (\psi_\delta\ast d_\eps J_\eps \hat  u)(x, h)
    =
    d_\eps (\psi_\delta\ast J_\eps  \hat  u)(x, h)
    =
    d_\eps \hat u_\delta(x,h),
\end{split}
\end{equation}
where we have used \eqref{depsjepscommute}. This relates the mollified strains $\hat\discbeta_\delta$ to finite differences of $u_\delta$.

Next, we relate  these finite differences at the discrete level to the continuous gradient of $u_\delta$. For any $x\in x_0+\eps \calC$, $y\in x_0+\eps T_*$, and $h\in\calN$, since $|x-y|\le \eps \diam \calC=\eps d_\calC/2\le \delta/2$ and $\psi_\delta\in C^\infty_c(B_{\delta/2})$,
\begin{equation}\label{eqduDu}
\begin{split}
    &
    d_\eps \hat u_\delta(x,h)-D \hat  u_\delta (y)h
    \\ & =
    \int_{B_{\delta}} 
        \frac{ \psi_\delta(z+\eps  h)-\psi_\delta(z)}\eps 
        (J_\eps \hat u)(x-z)- D\psi_\delta(z)h (J_\eps \hat u) (y-z)
    dz 
    \\ & =
    \int_{B_{\delta}} 
        \left[
            \frac{ \psi_\delta(z+\eps h)-\psi_\delta(z)}\eps 
            -
            D\psi_\delta(z+y-x)h
        \right]
        (J_\eps \hat  u)(x-z)
    dz .
\end{split}
\end{equation}
The square parenthesis can be estimated by 
\begin{equation}\label{eqd2psidelta}
    \|D^2\psi_\delta\|_{L^\infty} (\eps |h|+|y-x|)
    \le 
    c \frac{\eps}{\delta^{n+2}}.
\end{equation}
Squaring \eqref{eqduDu}, using H\"older, and recalling that $\hat \discbeta_\delta = d_\eps \hat u_\delta$ and \eqref{eqjepsuhat}, we obtain the pointwise estimate
\begin{equation}\label{eqbetadeltaDudelta}
\begin{split}
    |\hat \discbeta_\delta(x,h)-D \hat u_\delta (y)h|^2
    & \le  
    c \frac{\eps^2}{\delta^4\delta^{n}}  
    \| J_\eps \hat u\|_{L^2(B_{\frac32\delta}(x_0))}^2
    \\ & \le  
    c \frac{\eps^2}{\delta^2\delta^{n}}  
    E_\eps[d_\eps u, B_{3´\delta}(x_0)]
    =:
    R_{x_0}.
\end{split}
\end{equation}

The last step is to compare $D\hat u_\delta$ with $\hat F_\delta:=F_\delta-S$. We recall that $F_\delta$, that was defined in \eqref{eqdefFdelta}, satisfies  $F_\delta=\psi_\delta\ast D I_\eps u$ in $B_\delta(x_0)$, which gives $\hat F_\delta=\psi_\delta \ast DI_\eps \hat u$. We observe that
\begin{equation}
    D\hat u_\delta-\hat F_\delta
    =
    D \hat u_\delta-\psi_\delta\ast DI_\eps \hat u
    =
    D (\psi_\delta\ast J_\eps \hat u-\psi_\delta\ast I_\eps \hat u)  
    =
    D \psi_\delta\ast (J_\eps \hat u-I_\eps \hat u),
\end{equation}
so that, in particular for any $y\in B_\delta(x_0)$,
\begin{equation}
    |D \hat u_\delta-\hat F_\delta|(y)
    \le
    \| D \psi_\delta\|_{L^\infty} 
    \|J_\eps \hat u - I_\eps \hat u\|_{L^1(B_\delta(y))}.
\end{equation}
Integrating over $x_0+\eps T_*$, and recalling that $\|D \psi_\delta\|_{L^\infty} \le c \delta^{-n-1}$,
\begin{equation}\label{eqdistanceudeltafdelta}
\begin{split}
    \int_{x_0+\eps T_*} |D \hat u_\delta-\hat F_\delta|^2 dy 
    & \le
    \eps^n \calL^n(T_*) \|D \psi_\delta\|_{L^\infty}^2 
    \calL^n(B_\delta) 
    \|J_\eps \hat u-I_\eps \hat u\|_{L^2(B_{\frac32\delta}(x_0))}^2
    \\ & \le 
    c \frac{\e^{2+n}}{\delta^{2+n}} 
    E_\eps[\discbeta, B_{3\delta}(x_0)],
\end{split}
\end{equation}
where in the last step we used Lemma \ref{lemmacompactclusterkornbetaballdelta}.

At this point, we estimate the continuous energy of $F_\delta$ in the set $x_0+\eps T_*$ in terms of the energy of cluster centered in $x_0$. For any $w\in\eps\calC$, $h\in\calN$, and $z\in \eps T_*$ we obtain from \eqref{eqbetadeltaDudelta} with $x=x_0+w$ and $y=x_0+z$ that
\begin{equation}\label{eqdiffhatxideduhatd}
\begin{split}
    |\hat \discbeta_\delta(x_0+ w,h)-D\hat u_\delta (x_0+  z)h|^2
    \le 
    R_{x_0} .
\end{split}
\end{equation}
Since $\EC$ is a quadratic form, for any $\discbeta',\discbeta''\in D_\calC$ and $\eta\in (0,1]$, we have
\begin{equation}\label{ecyoung}
\begin{split}
    \EC[\discbeta']
    & \le 
    (1+\eta ) \EC[\discbeta'']+(1+\frac1\eta)\EC[\discbeta'-\discbeta'']
    \\ & \le 
    (1+\eta )\EC[\discbeta'']+\frac c\eta\sum_{(x,h)\in\calC_\calN} |(\discbeta'-\discbeta'')(x,h)|^2.
\end{split}
\end{equation}
Therefore, letting $\zeta_z:=D\hat u_\delta(x_0+ z)\in\R^{n\times n}$ and recalling \eqref{eqdefC}, \eqref{ecyoung}, and \eqref{eqdiffhatxideduhatd}, we have 
\begin{equation}
    \calL^n(\eps T_*) \frac12 \zeta_z\cdot \C\zeta_z
    =  
    \eps^n \EC[\discbeta^{\zeta_z}] 
    \le
    (1+\eta)  \eps^n  
    \EC[\tau_\eps^{x_0} \hat \discbeta_\delta]  
    +
    \frac{c\eps^n}{\eta} R_{x_0} .
\end{equation}
Integrating over all $z\in \eps T_*$,
\begin{equation}
    \int_{x_0+\eps T_*}  
        \frac12 D\hat u_\delta \cdot \C D\hat u_\delta 
    dy
    \le 
    (1+\eta) \eps^n \EC[\tau^{x_0}_\eps \discbeta_\delta]  
    +
    \frac{c\eps^n }{\eta} R_{x_0}.
\end{equation}
Recalling \eqref{eqdistanceudeltafdelta},
and in the last step the definition of $R_{x_0}$ in~\eqref{eqbetadeltaDudelta},
\begin{equation}
\begin{split}
    \int_{x_0+\eps T_*}  \frac12F_\delta \cdot \C  F_\delta dx
    & =
    \int_{x_0+\eps T_*} \frac12\hat F_\delta \cdot \C \hat F_\delta dx
    \\ & \le
    (1+\eta)
    \int_{x_0+\eps T_*} 
        \frac12 D \hat u_\delta \cdot \C D \hat u_\delta 
    dx
    + 
    \frac{c}{\eta} 
    \int_{x_0+\eps T_*} 
        |D \hat u_\delta -\hat F_\delta|^2 
    dx
    \\ & \le 
    (1+\eta)^2 \eps^n 
    \EC[\tau^{x_0}_\eps \discbeta_\delta]  
    +
    \frac{c\eps^n}{\eta} R_{x_0}
    + 
    \frac{c}{\eta}\frac{\e^{2+n}}{\delta^{2+n}} 
    E_\eps[\discbeta, B_{3\delta}(x_0)]
    \\ & \le 
    (1+\eta)^2 \eps^n \EC[\tau^{x_0}_\eps \discbeta_\delta]  
    + 
    \frac{c}{\eta}\frac{\e^{2+n}}{\delta^{2+n}} 
    E_\eps[\discbeta, B_{3\delta}(x_0)].
\end{split}
\end{equation}
We stress that the last estimate does not contain $u$. Summing over all $x_0$ such that $x_0+\eps T_*\cap\omega'\ne\emptyset$, and using that $\displaystyle\sum_{x_0\in\eps \calL} \chi_{B_{3\delta}(x_0)}\le c \delta^n/\eps^n$,
\begin{equation}
    \int_{\omega'} 
        \frac12F_\delta \cdot \C F_\delta 
    dx
    \le 
    (1+\eta)^2 E_\eps[\discbeta_\delta, \omega''] 
    +
    c \frac{\eps^2 }{\eta\delta^2} 
    E_\eps[\discbeta, \omega].
\end{equation}
With \eqref{eqpbetadconvd} and choosing $\eta=\eps/\delta$, the proof is concluded.
\end{proof}

\section{The continuum incompatible fields}\label{sectionextensionrigidity}
The scope of this section is to study the continuum model obtained from the discrete one in the region away from the dislocation cores and, in particular, to obtain a compactness that can then be translated to the discrete setting. These are the so-called semidiscrete core-region models, in which an elastic energy is integrated on a set of the form $\Omega\setminus B_{\rho_\eps}(\gamma)$, with $\rho_\eps\sim \eps$. We remark that \cite{ContiGarroniOrtiz2015} and \cite{ContiGarroniMarziani} did not prove compactness of the strains for core-region models of this type. The results of this section can also be seen as completion of what was done {in} that work.

The key aim of this Section is to prove Theorem~\ref{theo-compact-cont}, in which compactness of both the dislocation distributions and the rescaled strains is obtained for sequences of bounded energy. One key ingredient is the rigidity for almost compatible fields obtained in \cite{ContiGarroniRigidity} and recalled in Proposition~\ref{proprigiditycras} below. However, this type of rigidity requires fields that are defined over the entire set $\Omega$ (or on  generic Lipschitz subsets $\Omega'\subset\subset\Omega$). Therefore, here we develop a complex construction to extend a continuum strain field from $\Omega\setminus B_{\rho_\eps}(\gamma)$ to  $\Omega'$, concentrating the curl on $\gamma$ and without modifying it outside $B_{\rho_\eps'}(\gamma)$. The key construction is presented in Proposition~\ref{propextendcylincircles}. It uses specific constructions for hollow balls and hollow cylinders, which are obtained by relatively standard reflection procedures and presented in the Appendix, and a characterization of distributions which are obtained as curl of $L^1$ fields and concentrated on lines, discussed in Lemma~\ref{lemmacurll1h1}. As the strain is defined only away from the core region, its relation to the dislocation measure needs to be nonlocal. Here, {this} relation is effected using the compatibility concept presented in Definition~\ref{defbetamucompatible}. Finally, in Section~\ref{seccellproblem} we recall some properties of the classical dislocation line-tension energy. As the kinematics of dislocations is strongly dimension dependent, throughout the entire section we focus on the case $n=3$.

\subsection{Comparison with continuous dislocations}

The condition for the admissibility of $\discbeta$ can be formulated in terms of an associated discrete dislocation density, in analogy with what is done in the continuum case. If $\discbeta\in \Ade(\Omega)$, then we can define the continuum associated field $L\discbeta$ by Definition \ref{defLbeta}. In order to understand the behavior of the field $L\discbeta$, it is useful to consider three different regions. In the interior of $\Omega$, and outside the core, $L\xi$ is the unique interpolation of the values of $\discbeta$ on the bonds, and it is curl-free. In the core region, $L\xi$ is a best-fit approximation, but it does not necessarily agree with the value of $\xi$ on the bonds, and it is not necessarily curl-free. In a small region (of size approximately $\eps$) around the boundary of $\Omega$ it is arbitrarily set to $0$, hence at the interface between this boundary region and the interior it is not expected to be curl-free.

In order to make this precise, consider an open set $\Omega'\subset\subset\Omega$ such that $\dist(\Omega',\partial\Omega)\ge 2\ccorerad\eps$. By admissibility of $\discbeta$, $\Ce(\discbeta)\subseteq B_{m\eps}(\gamma)$ for some dilute polygonal $\gamma\in P(\Omega, k_\eps,\alpha_\eps)$. By Lemma~\ref{lemmaelemenmtarycontinuouscirc}\ref{lemmaelemenmtarycontinuouscirccurl}, $\Curl L\xi=0$ in $\Omega'\setminus \overline B_{m\eps}(\gamma)$. For each $\discbeta$ we shall define a measure $\hat\mu_\eps := \curl L\xi \in \calM(\Omega';\R^{3\times 3})$ concentrated on $B_{m\eps}(\gamma)\cap\Omega'$. 
We denote by $\calM(\Omega;\R^k)$ the set of $\R^k$-valued  Radon measures on $\Omega$, by $\calM^1(\Omega)$ the set of divergence-free measures of the form
$\theta\otimes t\calH^1\LL \gamma$,
{where $\gamma$ is a 1-rectifiable curve in $\Omega$, $t\in L^\infty(\gamma;S^2)$ is a tangent vector,
$\theta\in L^1(\gamma;\R^k)$}; and by $\calM^1_{\calL'}(\Omega)$, $\calL'$ a lattice, the subset of $\calM^1(\Omega)$ for which $\theta\in\calL'$.
It is then natural to introduce the set of dislocation measures concentrated on dilute polyhedrals, following the approach of \cite{ContiGarroniOrtiz2015}, and compare $\hat\mu_\eps$ with the latter.

\begin{definition}\label{defM1Lcont}
Given  $\alpha$, $k>0$, a lattice $\calL'$ and an open set $\omega\subseteq\R^3$, a dislocation measure $\mu\in \calM_{\calL'}^1(\omega)$ is said to be $(k,\alpha)$-{\em dilute} if there are finitely many closed segments $\gamma_j\subseteq\Omega$ and vectors $\theta_j\in \mathcal{L'}$, $t_j\in S^2$ (with $t_j$ tangent to $\gamma_j$) such that
\begin{equation}\label{eqmucontsegments}
	\mu = \sum_{j} \theta_j \otimes t_j \calH^1\LL \gamma_j
\end{equation}
and $\gamma=\cup_i\gamma_i$ is in $P(\omega,k,\alpha)$, in the sense that the segments $(\gamma_i)$ satisfy the properties stated in Definition~\ref{defdiluteness-curve}. We let $\calM_{\calL'}^1(\omega,k,\alpha)$ be the space of all measures in $\calM_{\calL'}^1(\omega)$ which are $(k,\alpha)$-dilute. 
\end{definition}

Given a measure in  $\mu\in\calM_{\calL'}^1(\Omega)$, with $\partial\Omega$ connected, we denote by $\contbeta_\mu\in L^{3/2}(\R^3;\R^{3\times 3})$ the solution of
\begin{equation}\label{eqdefbetamu}
    \Curl\eta=E\mu\qquad \Div \C\eta=0
\end{equation}
where $E:\calM_{\calL'}^1(\Omega)\to\calM_{\calL'}^1(\R^3)$ denotes an extension operator (see \cite{BourgainBrezis2004, ContiGarroniMassaccesi2015,ContiGarroniOrtiz2015}).
For notational simplicity we drop the dependence of $\beta_\mu$ on the extension operator.

We recall the definition of $\rho$-compatible pairs of dislocation measures and strains introduced in \cite{ContiGarroniMarziani}.
\begin{definition}\label{defbetamucompatible}
    Let $\Omega\subseteq\R^3$ be a bounded Lipschitz set with connected boundary, $\rho>0$. A pair $(\mu,\beta)\in  \calM^1(\Omega)\times L^1(\Omega\setminus B_\rho(\supp\mu);\R^{3\times 3})$ is $\rho$-compatible in $\Omega$ if
    there are an extension $E\mu\in\mathcal M^1(\R^3)$ of $\mu$ and a field $\beta_0\in L^1(\Omega;\R^{3\times 3})$ such that $\beta=\beta_{\mu}+\beta_0$ in $\Omega\setminus B_\rho(\supp E\mu)$ and $\Curl \beta_0=0$ in $\Omega$.
\end{definition}

The extension result given in Proposition \ref{propextendcylincircles} for fields with $\Curl$ concentrated on a curve $\gamma$ that is $(k,\alpha)$-dilute, gives a canonical way to associate a dilute dislocation measure $\mu$ to every $(\gamma,\discbeta)\in \Ade(\Omega,k,\alpha,m)$. Precisely, the following holds.

\begin{proposition}\label{propmufrombeta}
Fix $\Omega'\subset\subset\Omega$, and let $k_\eps$, $\alpha_\eps$ satisfy \eqref{eqdefheps}. Then, for $\eps$ small enough for every $(\gamma,\discbeta)\in \Ade(\Omega,k_\eps,\alpha_\eps, m)$ there exists a dislocation measure $\mu\in \calM_{\eps\calB}^1(\Omega')$ such that $(\mu,L\xi)$ is $m\eps$-compatible in $\Omega'$. Further,  $\supp\mu \subseteq\gamma$.
\end{proposition}

\begin{proof}
We apply Proposition \ref{propextendcylincircles} below to $\beta= L\xi$, $q=3/2$, and $\rho=m\eps$. We obtain an extension $\hat\beta$ of $\beta$ such that $\curl\hat\beta=0$ on $\Omega'\setminus\gamma$. The measure $\mu$ is given by the restriction of $\curl\hat\beta$ to $\Omega'$ and is supported on a subset of $\gamma$. With $\beta_\mu$ defined as above, we see that $\curl(\hat\beta - \beta_\mu) = 0$ on $\Omega'$. 
\end{proof}
The condition that 	$(\mu,L\xi)$ are $m\eps$-compatible in $\Omega'$ provides a natural way to compute the circulation of $L\xi$ away from the support of $\mu$. In particular, Proposition \ref{propmufrombeta} implies that, for all closed loops $\ell\subseteq \Omega'\setminus B_{m\eps}(\gamma)$, 
\begin{equation}\label{circ-prop}
	\int_\ell L\xi \tau ds= \int_\ell \contbeta_\mu  \tau ds,
\end{equation}
i.~e., the circulation around a closed loop of the field $L\xi$ is determined by the measure $\mu$. 

\subsection{Extension and compactness}
\newcommand{\omegaout}{{\omega^{\mathrm{out}}}}
\newcommand{\omegain}{{\omega^{\mathrm{in}}}}
\newcommand{\omegaouti}[1]{{\omega_{#1}^{\mathrm{out}}}}
\newcommand{\omegaini}[1]{{\omega_{#1}^{\mathrm{in}}}}

This subsection is devoted to the construction of an extension of admissible strains inside the core, which is stated below and proved using some intermediate results in Lemma~\ref{lemmacurll1h1}, Lemma~\ref{lemmaextendbetacyl}, and Lemma~\ref{lemmaextendbetaball}. We assume that the core is concentrated along a dilute curve, in the sense of Definition \ref{defdiluteness-curve}.

\begin{proposition}\label{propextendcylincircles}
Let $\Omega\subseteq\R^3$ be a bounded Lipschiz set, $\rho>0$, $k\in(0,1]$, $\alpha\in(0,\frac14]$, and let $\gamma\in P(\Omega,k,\alpha)$. Let $q\in(1,\frac32]$, $\contbeta\in L^{q}(\Omega_{\rho}(\gamma);\R^{3\times 3})$ be such that $\curl\contbeta=0$ on $\Omega_{\rho}(\gamma)$, with $\rho\le \alpha^{1+q} k/48$. Then letting $r:=2k$ there is $\hat\contbeta \in L^{q}(\Omega_{r};\R^{3\times 3})$ such that $\curl\hat\contbeta=0$ on $\Omega_{r}\setminus\gamma$, 
\begin{equation}
    \|\hat\contbeta+\hat\contbeta^T\|_{L^{q}(\Omega_{r})}
    \le
    c \|\contbeta+\contbeta^T\|_{L^{q}(\Omega_{\rho}(\gamma))}
\end{equation}
and $\hat\contbeta=\contbeta$ on $\Omega_{r}\setminus B_{\rho_*}(\gamma)$, with $\rho_*:=12\rho/\alpha$. Further, $\hat\mu:=\curl\hat\beta$ belongs to $\calM^1(\Omega_{r})$, obeys $\supp\hat\mu\subseteq\gamma$, and
\begin{equation}\label{eq-ext-gamma0}
	\|\hat\contbeta+\hat\contbeta^T\|_{L^{2}(\Omega_{r}
    \setminus 
    B_{\rho_*}(\supp\hat\mu))}
	\le
	c \|\contbeta+\contbeta^T\|_{L^{2}(\Omega_{\rho}(\gamma))}.
\end{equation}

Moreover, if there is a measure $\tilde\mu\in \calM^1(\Omega)$ with $\supp\tilde\mu\subseteq\gamma$ and such that $(\tilde\mu,\beta)$ are $\rho$-compatible, then necessarily $\tilde\mu=\hat\mu$ in $\Omega_{r}$. 
\end{proposition}

A similar statement, with a similar proof, holds in nonlinear kinematics using geometric rigidity instead of Korn's inequality.

Note that we extend $\beta$ outside $\gamma$ obtaining a field with curl concentrated on $\supp\hat\mu\subseteq\gamma$. In general, it might happen that $\gamma_0:=\Omega_r\cap \gamma\setminus \supp\hat\mu\neq\emptyset$, in which case the result provides an $L^2$ extension of $\beta$ near $\gamma_0$ (see Lemma \ref{lemmaextendbetacyl}, case $\theta=0$, and \eqref{eq-ext-gamma0}).

We first show that a divergence-free measure concentrated on a segment has necessarily a constant multiplicity. This will be needed to identify the Burgers vector on each part of the dislocation line.

\begin{lemma}\label{lemmacurll1h1}
Let $B\subseteq\R^3$ be a ball, $\gamma$ a diameter of $B$, and $\beta\in L^1(B;\R^{d\times 3})$ such that $\curl\beta=0$ on $B\setminus\gamma$. Then, there is $b\in \R^d$ such that $\curl\beta=b\otimes t\calH^1\LL\gamma$, where $t$ is a unit vector tangent to $\gamma$.
\end{lemma}
\begin{remark} 
It is however not true that any distribution $\Lambda\in (C_c^\infty(B_1;\R^3))'$ that is divergence free and supported on $\gamma$ has that form. For example, $\Lambda(\varphi):=\int_{-1}^1 \partial_2\varphi_1(x_1,0,0)dx_1$ is divergence free and supported on $(-1,1)e_1$.
\end{remark}
\begin{proof}[Proof of Lemma~\ref{lemmacurll1h1}]
After an affine change of variables, we can assume that $B=B_1(0)$ and that $\gamma=(-1,1)e_1$, and, working componentwise, that $d=1$. Let $B^\pm:=B\cap\{\pm x_3>0\}$. Since each of them is simply connected, there is $u^+\in W^{1,1}(B^+)$ and $u^-\in W^{1,1}(B^-)$ such that $\beta=Du^+$ on $B^+$, $\beta=Du^-$ on $B^-$. By \cite[Corollary~3.89]{AmbrosioFP}, the function
\begin{equation*}
	v:=u^+\chi_{B^+}+u^-\chi_{B^-}
\end{equation*}
belongs to $SBV(B)$ and
\begin{equation*}
	Dv=\beta\calL^3 \LL B + [v]\otimes e_3 \calH^2\LL \Sigma
\end{equation*}
where $\Sigma:=B\cap\{x_3=0\}=\partial B^+\cap\partial B^-$. Furthermore, as $\curl Dv=0$, for every $\varphi\in C^\infty_c(B;\R^3)$ we have
\begin{equation}\label{eqcurlbv}
	\int_{B} \beta\cdot\curl\varphi \, dx =
	-\int_\Sigma [v] (\curl \varphi) \cdot e_3 \,d\calH^2.
\end{equation}
Let $\Sigma^\pm:=\Sigma\cap\{\pm x_2>0\}$. Next we show that $[v]$ is constant on each of them. Given $\theta\in C_c^\infty(\Sigma^+)$, we can find $\eta\in C^\infty_c((-1,1))$ such that $\eta(0)=1$ and $x\mapsto \theta(x_1,x_2)\eta(x_3)$ belongs to $C_c^\infty(B\setminus\gamma)$. Then, using that $\curl\beta=0$ on $\supp (\theta\eta)\subseteq B\setminus\gamma$ implies $\int_B \beta \cdot \curl(\theta\eta e_i) dx=0$ for $i=1,2$,  from \eqref{eqcurlbv} we obtain
\begin{equation*}
	\int_{\Sigma^+} [v] \partial_i\theta \,d\calH^2=0 \text{ for }i=1,2
\end{equation*}
for every $\theta\in C_c^\infty(\Sigma^+)$, and the same on $\Sigma^-$. Therefore there are $b^\pm\in\R$ with $[v]=b^+$  on $\Sigma^+$ and $[v]=b^-$ on $\Sigma^-$, up to $\calH^2$-null sets.
	
Let now $\varphi\in C^\infty_c(B;\R^3)$. Then,
\begin{equation*}
\begin{split}
    \int_B 
        \beta \cdot\curl\varphi\, dx&=-\int_\Sigma [v] (\curl\varphi)\cdot e_3
    \,d\calH^2 
    \\ & = 
    b^+
    \int_{\Sigma^+} (\partial_2\varphi_1-\partial_1\varphi_2)d\calH^2
    +
    b^-
    \int_{\Sigma^-} (\partial_2\varphi_1-\partial_1\varphi_2)d\calH^2
    \\ & =
    \int_\gamma (b^--b^+) \varphi_1 \,d\calH^1,
\end{split}
\end{equation*}
hence $\curl\beta=(b^+-b^-)\otimes e_1 \calH^1\LL \gamma$, as claimed.
\end{proof}

The extension results needed for the proof of Proposition~\ref{propextendcylincircles} are stated here and proven in the Appendix. The first result concerns the extension inside cylinders, the second extension inside balls. We denote by $B'_r$ the two-dimensional ball of radius $r$ centered in the origin.

\begin{lemma}\label{lemmaextendbetacyl}
Let $\rho>0$, $\ell\ge\rho$, and $\omega:=\omegaout\cup\omegain$, with $\omegaout:=(B_{2\rho}'\setminus B_\rho')\times (0,\ell)$, $\omegain:= B_\rho'\times (0,\ell)$. Let $p\in [1,2)$, and assume that $\contbeta\in L^p(\omegaout;\R^{3\times 3})$, with $\curl\contbeta=0$ in $\omegaout$, is given. Then, there is $\tilde \contbeta\in L^p(\omega;\R^{3\times 3})$ such that $\tilde\contbeta=\contbeta$ on $\omegaout$, $\curl\tilde\contbeta=0$ on $\omega\setminus(\{0\}\times (0,\ell))$, and
\begin{equation}\label{eqlemmaextendbetacyllinelast}
    \|\tilde\contbeta+\tilde\contbeta^T\|_{L^p(\omegain)}
    \le 
    C \|\contbeta+\contbeta^T\|_{L^p(\omegaout)}.
\end{equation}
The constant {depends only on $p$.} Furthermore, there exists $\theta\in \R^3$ such that $\Curl\tilde\beta=\theta\otimes e_3\calH^1\LL (0,\ell)e_3$. If $\theta=0$, then \eqref{eqlemmaextendbetacyllinelast} holds for every $p\in[1,\infty)$.
\end{lemma}
In the case of the ball, the rigidity statement for incompatible fields from \cite{ContiGarroniRigidity} presented in Proposition~\ref{proprigiditycras} below is needed, and therefore we need to assume $p\in[1,\frac32]$.
\begin{lemma}\label{lemmaextendbetaball}
Let $\rho>0$, $p\in [1,\frac32]$, and fix some vectors $v_1, \dots , v_K\in S^2$. Define $\omegaout:=(B_{2\rho}\setminus B_\rho) \setminus \cup_i [\rho,2\rho) v_i$, and $\omegain:=B_{\rho}\setminus \cup_i [0,\rho)v_i$. Assume that $\contbeta\in L^p(\omegaout;\R^{3\times 3})$, with $\curl\contbeta=0$ in $\omegaout$, is given. 	Then, with $\omega:=\omegain\cup\omegaout$, there is $\hat \contbeta\in L^p(\omega;\R^{3\times 3})$ such that $\hat\contbeta=\contbeta$ on $\omegaout$, $\curl\hat\contbeta=0$ on $\omega$, and
\begin{equation}\label{eqlemmaextendbetaballsym}
    \|\hat\contbeta+\hat\contbeta^T\|_{L^{p}(\omegain)}
    \le 
    c \|\contbeta+\contbeta^T\|_{L^{p}(\omegaout)} 
    + 
    c \rho^{\frac3p-2} |\curl \contbeta|(B_{2\rho}\setminus B_\rho).
\end{equation}
The constant depends only on $p$. If $\Curl\beta=0$ on $B_{2\rho}\setminus B_\rho$, then
\begin{equation}\label{eqlemmaextendbetaballsym0}
	\|\hat\contbeta+\hat\contbeta^T\|_{L^{2}(\omegain)}
    \le c 
    \|\contbeta+\contbeta^T\|_{L^{2}(\omegaout)} .
\end{equation}
Furthermore, $\curl\hat\beta = \sum_i\theta_i\otimes v_i \calH^1\LL\cup_i[0,2\rho)v_i$ for some $\theta_i\in \R^3$.
\end{lemma}

\begin{figure}
\begin{center}
    \includegraphics[width=10cm]{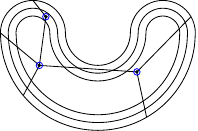}
\end{center}
\caption{Sketch of the construction of the segments in Proposition \ref{propextendcylincircles}. The three sets are $\Omega$, $\Omega_k$, $\Omega_{2k}$. The blue circles are the regions $B_d(z_i)$ which are eliminated from $\gamma$ to form $\hat\gamma$.}
\label{fig-cyls3}
\end{figure}
  
\begin{figure}
\begin{center}
    \includegraphics[width=12cm]{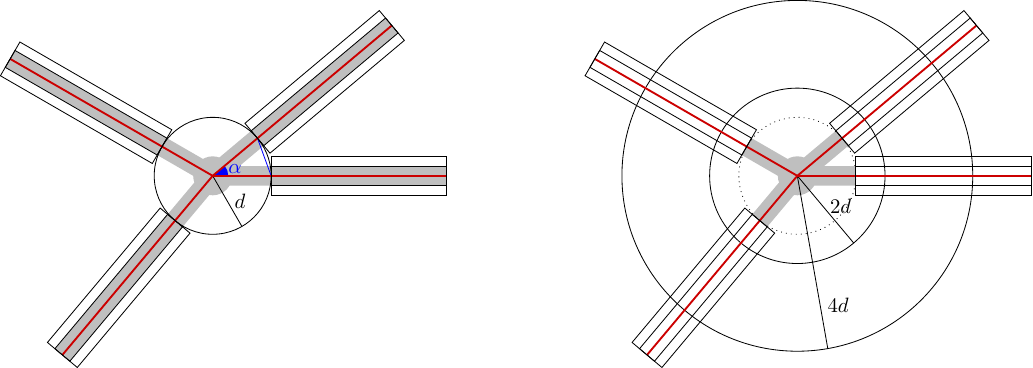}
\end{center}
\caption{Sketch of the construction in Proposition \ref{propextendcylincircles} around one of the points $z_i$. The red segments are the polygonal $\gamma$, the grey area denotes the one where $\contbeta$ needs to be defined. Left: Since the angle between two segments that intersect is at least $\alpha$, the distance between their parts outside $B_d(z_i)$ is at least $2d\sin\frac\alpha2$. The cylinders of radius $\rho$ with the segments $s_i$ as axis are therefore disjoint. Right: in the first step the curl is concentrated in the cylinders away from the $z_i$, in the second step it is concentrated inside balls around the points $z_i$.}
\label{fig-cyls2}
\end{figure}

\begin{figure}
\begin{center}
    \includegraphics[width=8cm]{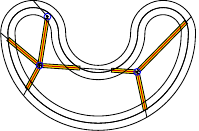}
\end{center}
\caption{Sketch of the construction in Proposition \ref{propextendcylincircles}. The outer set is $\Omega$, on which $\gamma$ and $\contbeta$ are defined. The intermediate one is $\Omega_k$, which is used to select the segments $s_i$ (marked red). The inner one is $\Omega_{2k}$, in which the end result is defined.}
\label{fig-cyls}
\end{figure}

We finally present the proof of Proposition~\ref{propextendcylincircles}. 
\newcommand{\raggioin}{d}
\newcommand{\raggioout}{D}
\begin{proof}[Proof of Proposition~\ref{propextendcylincircles}]
The set $\gamma$ consists of finitely many closed segments of length at least $k$ and of distance (if they are disjoint) at least $\alpha k$. We denote by $z_1, \dots, z_H$ their endpoints inside $\Omega$. If one of the segments does not have $z_i$ as an endpoint, then its distance from $z_i$ is at least $\alpha k$. Since each segment has length at least $k$, the number of points is bounded by $H\le 2\calH^1(\gamma)/k$.

Fix $\raggioin:=\frac{2\rho}{\sin(\alpha/2)}$, the reason will become clear below. We first observe that since $\alpha\in (0,\frac14]$  one has
\begin{equation}\label{eq-d-rho-k}
    16\rho
    \le 
    \raggioin
    \le 
    \frac{6}{\alpha}\rho
    \le
    \frac18 \alpha k
    \le 
    \frac1{32} k.
\end{equation} 
The third inequality is a consequence of the fact that $\rho\le \alpha^2k/48$. The second inequality follows from $\sin \frac\alpha2\ge\frac\alpha 3$ for $\alpha\in[0, \frac\pi2]$.

The set $\hat\gamma:=\gamma\setminus \bigcup_i B_\raggioin(z_i)$ consists of finitely many segments. We show that the distance between any two of them is at least $4\rho$. Consider two of them. If the corresponding segments in $\gamma$ were disjoint, then the distance is at least $\alpha k$, which by \eqref{eq-d-rho-k} is larger than $4\rho$. If they shared a point, say $z_i$, the distance is given by the distance between the two points from which they exit $B_d(z_i)$. Since the angle between the segments is at least $\alpha$, the distance is at least $2\raggioin\sin \frac\alpha 2$, which by the definition of $\raggioin$ equals $4\rho$ (see Figure~\ref{fig-cyls2}).

Since every segment comprised in $\gamma$ has length at least $k$, every segment comprised in $\hat\gamma$ has length at least $k-2d$. By \eqref{eq-d-rho-k}, we have that
\begin{equation}\label{eqk2drho}
    k-2d\ge \frac{k}2.
\end{equation}
The set $\hat\gamma\cap \Omega_k$ is a union of segments. We denote by $s_1, \dots, s_N$ the segments comprised in $\hat\gamma\cap \Omega_k$ which have length at least $k/2$, see Figure~\ref{fig-cyls3} and Figure~\ref{fig-cyls}. By \eqref{eqk2drho}, all the remaining segments in $\hat\gamma\cap \Omega_k$ have at least an endpoint in $\partial\Omega_k$, so that
\begin{equation}\label{eqk2drhos2om}
    \hat\gamma\cap \Omega_{\frac{3}2k}
    \subseteq 
    \bigcup_{i=1}^N s_i  
    \subseteq
    \hat\gamma\cap\Omega_k.
\end{equation}
Then, $\dist(s_i, \partial\Omega)\ge k\ge 2\rho$. 
By the construction of $\hat\gamma$, if $s_i$ and $s_j$ are not part of the same straight line, then $\dist(s_i, s_j)\ge 4\rho$. For every $s_i$ we consider the cylinder $C_i$ with radius $\rho$ and axis $s_i$, and analogously the cylinder $\tilde C_i$ with radius $2\rho$. Each of the two families consists of pairwise disjoint cylinders. As $\tilde C_i\setminus C_i\subseteq\Omega_\rho(\gamma)$, $\curl\contbeta=0$ on $\tilde C_i\setminus C_i$ for all $i$. By Lemma~\ref{lemmaextendbetacyl} there is an extension $\tilde\contbeta : (\Omega_k\setminus B_\rho(\gamma))\bigcup\cup_i C_i\to\R^{3\times 3}$ such that $\tilde\contbeta=\contbeta$ on $\Omega_k\setminus B_\rho(\gamma)$, $\curl \tilde\contbeta = 0$ in $\cup_i \tilde C_i\setminus s_i$, and 
\begin{equation}\label{eqtildebetabvetaprob}
    \|\tilde\contbeta+\tilde\contbeta^T\|_{L^q(\cup_i C_i)} 
    \le 
    c \|\contbeta+\contbeta^T\|_{L^q(\cup_i (\tilde C_i\setminus C_i))} .
\end{equation}
Since the constant in Lemma~\ref{lemmaextendbetacyl} depends only on $q$, and the cylinders $\tilde C_i$ are disjoint, the constant in the above estimate also depends only on $q$. Additionally, we obtain that for every $i$ there is  $\hat\theta_i\in \R^3$ such that  
\begin{equation}\label{eq-thetacappuccio}
	(\curl{\tilde\contbeta})\LL C_i 
    =
    \hat\theta_i\otimes \tau_i \calH^1\LL s_i,
\end{equation}
where $\tau_i$ is a unit vector parallel to $s_i$. If $\hat\theta_i=0$, then Lemma~\ref{lemmaextendbetacyl} gives also
\begin{equation}\label{eqtildebetabvetaprobp2}
    \|\tilde\contbeta+\tilde\contbeta^T\|_{L^2(C_i)} 
    \le 
    c \|\contbeta+\contbeta^T\|_{L^2(\tilde C_i\setminus C_i)} .
\end{equation}

In the next step, we use the extension inside the balls, as stated in Lemma \ref{lemmaextendbetaball}. We first estimate the total variation of $\curl\contbeta$ in the cylinders. This is done using the rigidity estimate in a subset of the cylinders with fixed aspect ratio and then using the standard circulation estimate.

Let $C_i^*$ be a part of $\tilde C_i\setminus C_i$ with axis of length $\rho$, which is a rotated and translated copy of $(B_{2\rho}'\setminus B_\rho')\times (0,\rho)$, and let $s_i^*$ be its axis. This exists since $\calH^1(s_i)\ge k/2\ge\rho$. By Korn's inequality, there is $S_i\in \R^{3\times 3}_\skw$ such that
\begin{equation}
    \|\contbeta-S_i\|_{L^{q}(C_i^*)}
    \le  
    c \|\contbeta+\contbeta^T\|_{L^{q}(C_i^*)}.
\end{equation}
We estimate, denoting by $\gamma_t$ a circle of radius $t\in(\rho,2\rho)$ centered in the origin and contained in a plane orthogonal to $s_i$, for a.~e.~$y\in s_i^*$ and a.~e.~$t\in (\rho,2\rho)$,
\begin{equation*}
    |\hat \theta_i| 
    \le
    \int_{y+\gamma_t} |\contbeta-S_i| d\calH^1.
\end{equation*}
Integrating over all $t\in (\rho,2\rho)$ and $y\in s_i^*$,
\begin{equation*}
\begin{split}
    \rho^2\, |\hat\theta_i| \le&
    \int_{s_i^*}
    \int_{\rho}^{2\rho}
    \int_{y+\gamma_t} |\contbeta-S_i| d\calH^1dt d\calH^1(y) 
    \\ = &
    \int_{C_i^*} |\contbeta-S_i| dx  
    \le 
    c \rho^{3(1-\frac1q)} \|\contbeta-S_i\|_{L^{q}(C_i^*)}
    \\ \le & 
    c \rho^{3(1-\frac1q)} \|\contbeta+\contbeta^T\|_{L^{q}(C_i^*)}.
\end{split}
\end{equation*}
Therefore,
\begin{equation*}
    |\hat\theta_i|^q
    \le 
    c \rho^{q-3} \|\contbeta+\contbeta^T \|_{L^{q}(C_i^*)}^q.
\end{equation*}
We remark that every $\tilde C_i\setminus C_i$ contains $\lfloor \frac{\calH^1(s_i)}{\rho}\rfloor\geq \lfloor\frac{k}{2\rho}\rfloor\geq \frac{k}{4\rho}$ disjoint copies of $C_i^*$. Summing over them, we conclude that
\begin{equation}\label{eqestrthetai}
    |\hat\theta_i|
    \le 
    c k^{-\frac1q} \rho^{1-\frac 2q} \|\tilde \contbeta
    +
    \tilde \contbeta^T\|_{L^{q}(\tilde C_i\setminus C_i)} ,
\end{equation}
with a constant that depends only on $q$.

The top and bottom faces of the cylinders $C_i$ are disks of radius $\rho$ centered in an endpoint of $s_i$
(recall~\eqref{eqk2drho} and~\eqref{eqk2drhos2om}). If the endpoint belongs to $\partial\Omega_k$, then the entire face does not intersect $\Omega_{2k}$ and we can ignore it. Otherwise, the endpoint belongs to
$\partial B_\raggioin(z_j)$ for one of the points $z_j$ introduced at the beginning of the proof, and the entire face is contained in $B_{2\raggioin}(z_j)$. We apply Lemma \ref{lemmaextendbetaball} to the balls $B_j:=B_{{2\raggioin}}(z_j)\subset\subset \tilde B_j := B_{4\raggioin}(z_j)$ for all $z_j\in\Omega_{\frac32k}$. The balls $\tilde B_j$  are disjoint since $|z_i-z_j|\ge \alpha k$ for all $i\ne j$. We obtain $\hat\contbeta\in L^1(\Omega_{2k};\R^{3\times 3})$ such that $\hat\beta=\tilde \beta$ in $\Omega_{2k}\setminus \cup_j B_{2d}(z_j)$, with $\curl\hat\contbeta=0$ on $\Omega_{2k}\setminus \gamma$,
\begin{equation}\label{eqextendbbprime}
\begin{split}
    \|\hat\contbeta+\hat\contbeta^T\|_{L^{q}(B_j)} 
    \le &
    c \|\tilde\contbeta+\hat\contbeta^T\|_{L^{q}(\tilde B_j\setminus B_j)}
    +
    c \raggioin^{\frac3q-2}
    |\curl\tilde\contbeta|(\tilde B_j\setminus B_j) 
    \\ \le & 
    c \|\tilde\contbeta+\tilde\contbeta^T\|_{L^{q}(\tilde B_j\setminus B_j)}
    +
    c \raggioin^{\frac3q-2}
    \sum_{i: s_i\cap B_j\ne \emptyset} \raggioin|\hat\theta_i|
    \\ \le & 
    c \|\tilde\contbeta+\tilde\contbeta^T\|_{L^{q}(\tilde B_j\setminus B_j)}
    \\ & +
    c k^{-\frac1q}{\rho^{1-\frac2q} \raggioin^{\frac3q-1}} 
    \sum_{i: s_i\cap B_j\ne \emptyset} \|\tilde\contbeta
    +
    \tilde\contbeta^T\|_{L^{q}(\tilde C_i\setminus C_i)},
\end{split}
\end{equation}
where in the last step we have used \eqref{eqestrthetai}. If all $\hat\theta_i$ are zero, then Lemma \ref{lemmaextendbetaball} also gives
\begin{equation}\label{eqextendbbprimep2}
    \|\hat\contbeta+\hat\contbeta^T\|_{L^{2}(B_j)} 
    \le 
    c \|\tilde\contbeta+\hat\contbeta^T\|_{L^{2}(\tilde B_j\setminus B_j)}.
\end{equation}
From \eqref{eq-d-rho-k}, we recall $\raggioin\le 6\rho/\alpha$, we take the $q$-th power of \eqref{eqextendbbprime} and sum to obtain
\begin{equation}\label{eqbetahjatbeta}
\begin{split}
    \sum_j \|\hat\contbeta+\hat\contbeta^T\|_{L^{q}(B_j)}^{q} 
    \le & 
    c  \sum_j 
    \|\tilde\contbeta+\tilde\contbeta^T\|_{L^{q}(\tilde B_j \setminus B_j)}^{q}
    \\ & +
    c\frac{\rho}{k \alpha^{3-q}} 
    \sum_j \Bigl(\sum_{i: s_i\cap B_j\ne \emptyset}
    \|\tilde\contbeta+\tilde\contbeta^T\|_{L^{q}(\tilde C_i\setminus C_i)}\Bigr)^{q}.
\end{split}
\end{equation}
By Hölder's inequality and the fact that diluteness implies $\#\{i: s_i\cap B_j\ne \emptyset\}\le c\alpha^{-2}$ we get
\begin{equation}
    \Bigl(\sum_{i: s_i\cap B_j\ne \emptyset}
    \|\tilde\contbeta+\tilde\contbeta^T\|_{L^{q}(\tilde C_i\setminus C_i)}\Bigr)^q
    \le 
    c \alpha^{-2(q-1)}\sum_{i: s_i\cap B_j\ne \emptyset} \|\tilde\contbeta+\tilde\contbeta^T\|_{L^{q}(\tilde C_i\setminus C_i)}^q.
\end{equation}
With \eqref{eqtildebetabvetaprob} and
\eqref{eqbetahjatbeta}, using that 
the balls are disjoint, that each cylinder touches only two balls, and that the cylinders are disjoint, we conclude
\begin{equation}
\begin{split}
    \|\hat\contbeta+\hat\contbeta^T\|_{L^{q}(\Omega_{2k})}^{q} 
    \le & 
    c \|\contbeta+\contbeta^T\|_{L^{q}(\Omega_\rho(\gamma))}^{q}
    +
    c\frac{\rho}{k \alpha^{q+1}}
    \|\contbeta+\contbeta^T\|_{L^{q}(\Omega_\rho(\gamma))}^{q}
\end{split}
\end{equation}
with
$\hat\beta=\beta$ in $\Omega_{2k}\setminus B_{2d}(\gamma)$. 
Using Lemma~\ref{lemmacurll1h1} on small balls centered in points of $\gamma$, we obtain that $\hat\mu:=\curl\hat\beta$ is a measure in $\calM^1(\Omega_{2k})$, supported on a subset of $\gamma$ and it is given by  $\hat\mu=\sum_i \hat\theta_i\otimes \tau_i\calH^1\LL \gamma_i$, with the vectors $\hat\theta_i\in\R^3$ given in \eqref{eq-thetacappuccio}.

Similarly, using \eqref{eqtildebetabvetaprobp2} and \eqref{eqextendbbprimep2} we obtain \eqref{eq-ext-gamma0}.

Assume now that there is a measure $\mu =\sum_i \theta_i\otimes \tau_i\calH^1\LL \gamma_i \in \calM^1(\Omega)$ with $\supp\mu\subseteq\gamma$ and such that $(\mu,\beta)$ are $\rho$-compatible. We only need to show that, if $\gamma_i\cap \Omega_{2k}$ is nonempty, then $\hat\theta_i=\theta_i$.

Pick $\rho'\in (\rho_*, 2\rho_*)$ and a point $y\in \gamma_i\cap \Omega_k$ which is at distance at least $k/4$ from the endpoints of $\gamma_i$. Let $\ell_i$ be a circle of radius $\rho'$ centered in $y$ and orthogonal to $\gamma_i$. By diluteness, $\dist(\ell_i,\gamma)>\rho$, and $\ell_i\subseteq \Omega_\rho(\gamma)$. Since $\beta=\hat\beta$ on $\Omega_r\setminus B_{\rho_*}(\gamma)$, for almost every $\rho'$ and $y$ we have
\begin{equation}
    \theta_i=\int_{\ell_i} \beta t\, ds 
    =
    \int_{\ell_i} \hat\beta t \, ds = \hat\theta_i.
\end{equation}
Therefore, $\curl\hat\beta=\mu$.
\end{proof}

Finally we obtain the desired compactness result.

\begin{theorem}\label{theo-compact-cont}
Let $\Omega\subseteq\R^3$ be a bounded Lipschitz set with connected boundary, $\mu_\eps\in \calM^1_\calB(\Omega)$ with $\mu_\eps=\eps\theta_\eps\otimes t_\eps \calH^1\LL\gamma_\eps$ which is $(k_\eps,\alpha_\eps)$ dilute, $\beta_\eps\in L^2(\Omega_{\rho_\eps}(\gamma_\eps);\R^{3\times 3})$, such that the two are $\rho_\eps$ compatible,
\begin{equation}
    \liminf_{\eps\to0}
    \frac{1}{\eps^2\ln\frac1\eps} 
    \int_{\Omega_{\rho_\eps}(\gamma_\eps)} 
        |\contbeta_\eps+\contbeta_\eps^T|^2 
    dx 
    <
    \infty,
\end{equation}
with $\lim_{\eps\to0}\frac{\ln \rho_\eps}{\ln\eps}=1$ and assume \eqref{eqdefheps}.
Then, there is $(\mu,\eta) \in\mathcal M^1_\calB(\Omega)\times L^1_\loc(\Omega;\R^{3\times 3}) $ and a subsequence such that 
\begin{equation}\label{eqweakconvmuepsmu0}
    \frac1{\eps_j}\mu_{\eps_j}\weakstarto\mu \text{ in $\Omega'$}, 
\end{equation}
for every $\Omega'\subset\subset \Omega$, there is a sequence $\rho^*_{\eps_j} \to 0$ with $\lim_{\eps\to 0}  \frac{\ln \rho^*_{\eps_j}}{\ln {\eps_j}}=1$, and there exists $S_j\in\R^{3\times 3}_\skw$ such that
\begin{equation}\label{eqetaepstheoremlin2}
    \frac{\beta_{\eps_j}-S_j}{\eps_j(\ln\frac1{\eps_j})^{1/2}}  
    \chi_{\Omega\setminus B_{\rho^*_{\eps_j}}(\gamma_{\eps_j})}
    \weakto 
    \conteta \text{ weakly in } L^{\frac32}_\loc(\Omega;\R^{3\times 3}).
\end{equation}
\end{theorem}
\begin{proof}
Using Proposition~\ref{propextendcylincircles} with $q=3/2$ there are $\rho_\eps^*\to 0$ and $\tilde\beta_\eps\in L^{\frac32}(\Omega_{2k_\eps};\R^{3\times 3})$  with $\tilde\beta_\eps=\beta_\eps$ on $\Omega_{2k_\eps}\setminus B_{\rho_\eps^*}(\gamma_\eps)$, $\curl\tilde\beta_\eps=0$ on $\Omega_{2k_\eps}\setminus \gamma_\eps$, and
\begin{equation}
    \int_{\Omega_{2k_\eps}} 
        |\tilde\contbeta_\eps+\tilde \contbeta_\eps^T|^{\frac32} 
    dx 
    <
    c 
    \int_{\Omega_{\rho_\eps}(\gamma_\eps)}
        |\contbeta_\eps+\contbeta_\eps^T|^{\frac32} 
    dx .
\end{equation}
Moreover, $\curl\tilde \beta_\eps=\mu_\eps$. Using Proposition~5.2 of \cite{ContiGarroniMarziani} on Lipschitz sets $\Omega'\subset\subset \Omega$, and then taking a diagonal subsequence, we obtain \eqref{eqweakconvmuepsmu0}. Similarly, by Proposition~5.4 of \cite{ContiGarroniMarziani} we obtain
\begin{equation}\label{eqetaepstheoremlin2bb}
    \frac{\tilde\beta_{\eps_j}-S_j}{\eps_j(\ln\frac1{\eps_j})^{1/2}}
    \weakto 
    \conteta \text{ weakly in } L^{\frac32}_\loc(\Omega;\R^{3\times 3}).
\end{equation}
Since $|B_{\rho^*_{\eps_j}}(\gamma_{\eps_j})|\to0$, in particular this implies \eqref{eqetaepstheoremlin2}.
\end{proof}

\subsection{The line-tension energy density $\psiC^\rel$}
\label{seccellproblem}
In a continuum setting, to a given distribution of dislocations $\mu$ we associate a line-tension energy which is defined via a cell-problem formula.

Given a Burgers vector $b$ and an orientation $t$, the line-tension energy density of a straight dislocation $\psiC(b,t)$ is defined by solving a one-dimensional problem, see \cite[Section~5]{ContiGarroniOrtiz2015}. The result only depends on $b$, $t$, and the matrix of elastic coefficients $\C$ and corresponds to the classical textbook line-tension energy, see for example \cite[Eq.~(3.87)]{HirthLothe1968}. In particular, the problem 
\begin{equation}
\begin{cases}
    \curl\beta = b\otimes t \calH^1\LL (\R t) \\
    \Div \C\beta=0
\end{cases}
\end{equation}
has a solution in $L^1_\loc(\R^3;\R^{3\times 3})$ which decays as $1/\dist(\cdot, \R t)$ at infinity \cite[Lemma~2.3, Lemma~2.4]{ContiGarroniMarziani}. The solution $\beta_{b,t}$ satisfies
\begin{equation}
    \psiC(b,t)
    =
    \frac{1}{h\ln \frac Rr} 
    \int_{Q_t[(B'_R\setminus B'_r)\times (0,h)]}
    \frac12 \C \beta_{b,t}\cdot\beta_{b,t}\, dx
\end{equation}
for any $0<r<R$ and any $h>0$, where $Q_t\in\SO(3)$ is a rotation such that $Q_te_3=t$.

By \cite[Lemma~5.8]{ContiGarroniOrtiz2015}, for every $M\ge 1$ there is $\omega_M:(0,\infty)\to(0,\infty)$ with $\lim_{r\to0} \omega_M(r)=0$ such that, for every $\beta\in L^1(\R^3;\R^{3\times3})$ with $\curl\beta=b\otimes t \calH^1\LL (\R t)$ and every $Q_t\in \SO(3)$ with $Q_te_3=t$, we have
\begin{equation}\label{eqlemma58}
    \left(1-\frac cM-\omega_M(\frac rR)\right) \psiC(b,t)
    \le
    \frac{1}{h\ln \frac Rr} 
    \int_{Q_t[(B'_R\setminus B'_r)\times (0,h)]} \frac12 \C \beta\cdot\beta dx.
\end{equation}

Explicit examples show that an energy concentrated on lines of the form $\int_\gamma \psiC(b,t)d\calH^1$ may not be lower semicontinuous with respect to weak-* convergence of the dislocation density measure. Formation of microstructure results in the relaxed energy
\begin{equation}
    E^\mathrm{rel}[b\otimes t\calH^1\LL\gamma]
    = 
    \int_\gamma \psiC^\rel(b,t) d\calH^1,
\end{equation}
where the energy density $\psiC^\rel$ is the $\calH^1$-elliptic envelope of $\psiC$, given by \cite{ContiGarroniMassaccesi2015}
\begin{equation}
    \psiC^\rel(b,t)
    :=
    \inf 
    \left\{
        \int_{\gamma\cap B_{1/2}}
            \psiC(\theta,\tau)
        d\calH^1 : 
        \supp(\mu-b\otimes t \calH^1\LL(t\R))
        \subseteq B_{1/2}
    \right\},
\end{equation}
where $\mu=\theta\otimes \tau\calH^1\LL\gamma$ is assumed to be divergence-free.

\section{The $\Gamma$-convergence result} \label{secgammaconvresult}

This section is the core of the paper, in which we derive the continuum limit of our discrete model. The first  result (Theorem~\ref{theomainresultelastic})
is simpler and concerns the elastic scaling in generic dimension $n\ge 2$; the second (Theorem~\ref{theomainresult}) deals with the more complex line-tension energy scaling in three dimensions, its proof covers most of this section and is divided into three propositions.
The compactness statement (Section~\ref{seccomp0}) builds upon the one obtained in  Section~\ref{sectionextensionrigidity} above; the lower bound (Section~\ref{seccomp}) follows from the a classical strategy and uses the continuum comparison of Section~\ref{seclocalcontinuum}. The upper bound (Section~\ref{secdiscreteupperbound}) requires a ad-hoc construction and a careful discretization of the limiting strain fields.

\begin{theorem}\label{theomainresultelastic}
Let $\Omega\subseteq\R^n$ be a bounded connected Lipschitz open set, $n\ge 2$. Let $\calC$, $\calN$ be as in Definition~\ref{deflattice} such that a cover in the sense of Definition~\ref{def-tetrahedra} exists. Let the cluster energy $E_\calC$ be as in Definition~\ref{defcluster}, which in particular obeys \eqref{eqasscoerc}, define $E_\eps$ as in Definition~\ref{defdepsu} and \eqref{eqdeftotalenergy}. Then:
\begin{enumerate}
\item(Compactness) If a sequence of maps $u_\eps:\Omega\cap \eps\calL\to\R^n$, $\eps\to0$, obeys $E_\eps[d_\eps u_\eps, \Omega]\le C\eps^2$, then there is  a subsequence and a function $v\in W^{1,2}(\Omega;\R^{n})$ such that for some matrices $S_k\in \R^{n\times n}_\skw$
\begin{equation}\label{eqconvbetaelastic}
    \frac{Ld_{\eps_k}u_{\eps_k} -S_k}{\eps_k} 
    \weakto 
    Dv \text{ weakly in }L^{2}_\loc(\Omega;\R^{n\times n}).
\end{equation} 
\item (Lower bound) For every $v\in W^{1,2}(\Omega;\R^{n})$ and every sequence $\eps_k\to0$, $u_k:\Omega\cap \eps_k\calL\to\R^n$, $S_k\in\R^{n\times n}_\skw$ converging to $v$ in the sense of \eqref{eqconvbetaelastic}, 
\begin{equation}
    \int_\Omega\frac12 \C  Dv\cdot Dv \,dx
    \le 
    \liminf_{k\to\infty} 
    \frac{1}{\eps_k^2} E_{\eps_k}[d_{\eps_k}u_k,\Omega].
\end{equation}
\item(Upper bound) For every $v\in W^{1,2}(\Omega;\R^n)$ and every sequence $\eps_k\to0$, there is a sequence $u_k:\Omega\cap \eps_k\calL\to\R^n$ converging to $v$ in the sense of \eqref{eqconvbetaelastic} with $S_k=0$ such that
\begin{equation}
   \limsup_{k\to\infty} \frac{1}{\eps_k^2} E_{\eps_k}[d_{\eps_k}u_k,\Omega]
   \le 
   \int_\Omega\frac12 \C  Dv\cdot Dv\, dx.
\end{equation}
\end{enumerate}
\end{theorem}

\begin{proof}
Compactness: Let $\Omega'\subset\subset\Omega$, connected and Lipschitz. Let $v_\eps:= I_\eps u_\eps:\Omega'\to\R^n$ be the piecewise affine interpolation as in Definition~\ref{definterpolation}. By Remark~\ref{remark-Lcsiexact}, we have $Ld_{\eps}u_{\eps} =Dv_{\eps}$ in $\Omega'$. By Proposition~\ref{prop-Lcsi}, we have, for sufficiently small $\eps$,
\begin{equation}
    \int_{\Omega'}|Dv_\eps+Dv_\eps^T|^2 dx
    \le 
    c  E_\eps[d_\eps u_\eps, \Omega]\le c\, C\eps^2.
\end{equation}
Using Korn's inequality, we obtain that there are matrices $S_k$ and a subsequence such that $(Dv_{\eps_k}-S_k)/\eps_k$ converges to $Dv$ weakly in $L^2(\Omega'; \R^{n\times n})$. With a diagonal argument, it is easily shown that the matrices $S_k$ can be chosen so that they do not depend on $\Omega'$, and similarly for the function $v$.

Lower bound: This follows from Theorem~\ref{theomollifydisccont}. As above, fix $\Omega'\subset\subset\Omega$ and let $v_\eps:=I_\eps u_\eps$. For every $\delta\in (0,\frac15\dist(\Omega',\partial\Omega)]$, by Theorem~\ref{theomollifydisccont} applied to $\xi=d_\eps u_\eps-\xi^{S_\eps}$, for $\eps$ sufficiently small, we have
\begin{equation}
    \frac{1}{\eps^2} 
    \int_{\Omega'} 
        \frac12 \C 
        (\psi_\delta\ast Dv_\eps-S_\eps)
        \cdot 
        (\psi_\delta\ast Dv_\eps-S_\eps) 
    dx
    \le 
    \frac{1}{\eps^2}(1+c\frac\eps\delta)
    E_\eps[d_\eps u_\eps,\Omega].
\end{equation}
Taking the limit of $\eps\to0$,
\begin{equation}
    \int_{\Omega'} 
        \frac12 \C 
        (\psi_\delta\ast Dv)
        \cdot 
        (\psi_\delta\ast Dv) 
    dx 
    \le 
    \liminf_{\eps\to0} 
    \frac{1}{\eps^2} E_\eps[d_\eps u_\eps,\Omega].
\end{equation}
Taking the limit $\delta\to0$,
\begin{equation}
    \int_{\Omega'} 
        \frac12 \C Dv\cdot Dv 
    \,dx
    \le 
    \liminf_{\eps\to0}  
    \frac{1}{\eps^2} E_\eps[d_\eps u_\eps,\Omega].
\end{equation}
Since the bound holds for all $\Omega'\subset\subset\Omega$, the proof is concluded.

Upper bound: Assume that $v\in C^2(\overline\Omega;\R^n)$. Define $u_\eps(x):=\eps v(x)$ for all $x\in \Omega\cap\eps\calL$. For every $x\in \clust_\eps^\Omega$ and every $(y,h)\in \calC_\calN$, we have that
\begin{equation}\begin{split}
    (\tau^x_\eps  d_\eps u_\eps)(y,h)
    =
    d_\eps u_\eps(x+\eps y,h)
    = &
    \frac{u_\eps(x+\eps y+\eps h)-u_\eps(x+\eps y)}{\eps}
    \\ = &
    {v(x+\eps y+\eps h)-v(x+\eps y)} ,
\end{split}\end{equation}
so that
\begin{equation}
    |(\tau^x_\eps  d_\eps u_\eps)(y,h)-\eps Dv(x)h|
    \le 
    c \eps^2 \|D^2v\|_{L^\infty(\Omega)} ,
\end{equation}
for all $(y,h)$. As $\EC$ is a quadratic form, for every $\sigma\in (0,1]$,
\begin{equation}\begin{split}
    \EC[\tau^x_\eps d_\eps u_\eps] 
    \le &
    (1+\sigma) \EC[\xi^{\eps Dv(x)}] 
    + 
    \frac{c}{\sigma}\eps^4 \|D^2v\|_{L^\infty(\Omega)}^2,
\end{split}\end{equation}
and, with \eqref{eqdefC}, we obtain
\begin{equation}\begin{split}
    \EC[\tau^x_\eps d_\eps u_\eps] 
    \le &
    (1+\sigma) \eps^2\frac12 \C Dv(x)\cdot Dv(x) \calL^n(T_*)
    +
    C_\sigma \eps^4 \|D^2v\|_{L^\infty(\Omega)}^2.
\end{split}\end{equation}
Summing over all $x\in \clust^\Omega_\eps$, we obtain
\begin{equation*}
\begin{split}
    E_\eps[d_\eps u_\eps,\Omega]
    \le & 
    (1+\sigma) \eps^2 
    \sum_{x\in \clust_\eps^\Omega} 
        \eps^n  \frac12 \C Dv(x)\cdot Dv(x) \calL^n(T_*)
    +
    C_\sigma \eps^4 
    \sum_{x\in \clust_\eps^\Omega} 
        \eps^n \|D^2v\|_{L^\infty(\Omega)}^2.
\end{split}
\end{equation*}
Using that $|Dv(x)-Dv(z)|\le c\eps\|D^2v\|_{L^\infty(\Omega)}$ for all $z\in x+\eps T_*$ and
\begin{equation}
    \C Dv(x)\cdot Dv(x)
    \le
    (1+\sigma)
    \C Dv(z)\cdot Dv(z)
    +
    \frac c\sigma |Dv(x)-Dv(z)|^2,
\end{equation}
this becomes
\begin{equation}\begin{split}
    E_\eps[d_\eps u_\eps,\Omega]
    \le & 
    (1+\sigma)^2 \eps^2  
    \int_\Omega \frac12  \C Dv\cdot Dv \, dx
    +
    C_\sigma' |\Omega| \eps^4 \|D^2v\|_{L^\infty}^2 .
\end{split}\end{equation}
We divide by $\eps^2$, take the limit $\eps\to0$, and recall that $\sigma$ is arbitrary, to conclude the proof of the upper bound for $v\in C^2(\overline\Omega; \R^n)$. By density the general case follows.
\end{proof}

After these prelimiaries, we study the situation with dislocations, and obtain convergence to
\begin{equation}
    E_0[\conteta,b\otimes t \calH^1\LL\gamma,\Omega]
    :=
    \int_\Omega\frac 12\C \conteta\cdot \conteta \,dx 
    + 
    \int_{\gamma\cap\Omega} \psiC^\rel(b,t)d\calH^1.
\end{equation}
\begin{theorem}\label{theomainresult}
Let $\Omega\subseteq\R^3$ be a bounded Lipschitz open set with $\partial\Omega$ connected. Let $\calC$, $\calN$ be as in Definition~\ref{deflattice} such that a cover in the sense of Definition~\ref{def-tetrahedra} exists. Let the cluster energy $E_\calC$ be as in Definition~\ref{defcluster}, which in particular obeys~\eqref{eqasscoerc}, and define $E_\eps$ as in Definition~\ref{defdepsu} and ~\eqref{eqdeftotalenergy}. Given a complete set of slip systems, in the sense of Definition~\ref{defslipsys} and Definition~\ref{defbetapl2}, and $k_\eps$, $\alpha_\eps$ which obey \eqref{eqdefheps}, $m\ge k_*$, let $\Ade$ be as in \eqref{eq-Ade-short}. Then, the following holds:
\begin{enumerate}
\item(Compactness) If $\discbeta_\eps\in\Ade(\Omega)$ and $E_\eps[\discbeta_\eps, \Omega]\le C\eps^2\log\frac1\eps$, then there is a subsequence and a pair $(\conteta,\mu)\in L^{2}(\Omega;\R^{3\times 3})\times \calM_\calB^1(\Omega)$ such that $\curl\conteta=0$ and, for some matrices $S_k\in \R^{3\times 3}_\skw$,
\begin{equation}\label{eqconvbeta}
    \frac{L\discbeta_{\eps_k} -S_k}{\eps_k\log^{1/2}\frac1{\eps_k}} 
    \weakto 
    \conteta \text{ weakly in }L^{p}_\loc(\Omega;\R^{3\times 3})
\end{equation} 
for every $p\in[1,3/2)$, and
\begin{equation}\label{eqconvmu}
    \frac1{\eps_k} \curl L\discbeta_{\eps_k} 
    \weakto 
    \mu \text{ as distributions in $(C_c^\infty(\Omega;\R^{3\times 3}))^*$}.
\end{equation}
\item (Lower bound) For every $(\conteta,\mu)\in L^{2}(\Omega;\R^{3\times 3})\times \calM_\calB^1(\Omega)$ with $\curl\conteta=0$ and every  $\eps_k\to0$, $\discbeta_k\in \Adek(\Omega)$, $S_k\in\R^{3\times 3}_\skw$ converging to $(\conteta,\mu)$ in the sense of \eqref{eqconvbeta} and \eqref{eqconvmu}, 
\begin{equation}
    E_0[\conteta,\mu,\Omega]
    \le 
    \liminf_{k\to\infty} 
    \frac{1}{\eps_k^2\log\frac1{\eps_k}} 
    E_{\eps_k}[\discbeta_k,\Omega].
\end{equation}

\item(Upper bound) For every $(\conteta,\mu)\in L^{2}(\Omega;\R^{3\times 3})\times  \calM_\calB^1(\Omega)$ with $\curl\eta=0$ and every sequence $\eps_k\to0$, there is a sequence $\discbeta_k\in \Adek(\Omega)$ converging to $(\conteta,\mu)$ in the sense of \eqref{eqconvbeta} and \eqref{eqconvmu} with $S_k=0$, such that
\begin{equation}
    \limsup_{k\to\infty} 
    \frac{1}{\eps_k^2\log\frac1{\eps_k}} 
    E_{\eps_k}[\discbeta_k,\Omega]
    \le 
    E_0[\conteta,\mu,\Omega].
\end{equation}
\end{enumerate}
\end{theorem}

\begin{proof}
The theorem follows from Proposition~\ref{prop:compact}, Proposition~\ref{prop:lower} and Proposition~\ref{prop:upperbound} below.
\end{proof}

\subsection{Compactness}\label{seccomp0}
We next prove compactness of sequences with bounded energy. The symmetric part of $L\discbeta_\eps$ is controlled by the energy via~\eqref{eqLbetasymEcint}, whereas for $\curl L\discbeta_\eps$ we only have a weaker control via~\eqref{eqcurlLbetatetr}. In order to use the rigidity result for incompatible fields in Proposition~\ref{proprigiditycras}, it is necessary to extend $L\discbeta_\eps$ to the core region, starting from the curl-free values of $L\discbeta_\eps$ outside of the core, obtaining a field $\beta_\eps$, with $\curl\beta_\eps$ concentrated on lines. The extension (Proposition~\ref{propextendcylincircles}) and the rigidity estimate (Proposition~\ref{proprigiditycras}) hold in $L^{3/2}$, whence an $L^{3/2}$ estimate for $L\discbeta_\eps$ follows outside of the core. To obtain convergence in distributions, it remains to exclude the core region. This is the step from \eqref{eqpropcompt3} to \eqref{eqpropcompt4}, done in~\eqref{eq-Lxi-compact2}, and leads to weak convergence of the strains in $L^p_\loc$ for $p$ strictly less than $3/2$.

\begin{proposition}\label{prop:compact}
\sloppypar Under the assumptions of Theorem~\ref{theomainresult},  let $(\gamma_\eps,\discbeta_\eps)\in \Ade(\Omega,k_\eps,\alpha_\eps,m)$ be such that $E_\eps[\discbeta_\eps,\Omega]\le E \eps^2\log\frac1\eps$  for some $E\in\R$. Then there are $\rho^*_\eps\to 0$ with
\begin{equation}\label{eqrhostarepseps}
    \lim_{\eps\to0}\frac{\ln \rho^*_\eps}{\ln\eps}=1
\end{equation}
and a sequence $\eps_k\to0$ such that:
\begin{enumerate}
\item
There is $\mu\in \calM^1_\calB(\Omega)$ satisfying
\begin{equation}\label{eqpropcompt2}
    \frac1{\eps_k}  \curl L\discbeta_{\eps_k} 
    \weakto 
    \mu \text{ as distributions in $(C_c^\infty(\Omega;\R^{3\times 3}))^*$}.
\end{equation}
\item There are measures $\mu_{\eps_k}\in \calM^1_{\eps_k\calB}(\Omega)$ satisfying
\begin{equation}\label{eqpropcompt1b}
	\frac1{\eps_k}  \mu_{\eps_k}
    \weakto 
    \mu \text{ weakly as measures in $(C_c^0(\Omega;\R^{3\times 3}))^*$},
\end{equation}
for every $\Omega'\subset\subset\Omega$ and $k$ large enough the pairs $(\mu_{\eps_k},L\xi_{\eps_k})$ are $\rho_{\eps_k}^*$\!-compatible in $\Omega'$ and $\Omega'\cap \supp\mu_{\eps_k}\subseteq\gamma_{\eps_k}$.

\item There are $S_k\in \R^{3\times 3}_\skw$ and $\conteta\in L^{\frac32}_{\rm loc}(\Omega;\R^{3\times 3})$ with $\curl\conteta=0$ satisfying
\begin{equation}\label{eqpropcompt3}
    \frac{L\discbeta_{\eps_k}-S_k}
    {\eps_k(\log {\frac1{\eps_k}})^{1/2}}
    \chi_{\Omega\setminus B_{\rho^*_{\eps_k}}(\gamma_{\eps_k})}
    \weakto 
    \conteta \text{ weakly in }
    L^{\frac32}_\loc(\Omega;\R^{3\times 3})
\end{equation}
and, for any $p\in[1,\frac32)$,
\begin{equation}\label{eqpropcompt4}
    \frac{L\discbeta_{\eps_k}-S_k}
    {\eps_k(\log {\frac1{\eps_k}})^{1/2}}
    \weakto 
    \conteta \text{ weakly in } L^{p}_\loc(\Omega;\R^{3\times 3}).
\end{equation}
\end{enumerate}
\end{proposition}
\begin{proof}
From Proposition \ref{prop-Lcsi}, we have that
\begin{equation}\label{eqbdcpen}
    \limsup_{\eps\to0}\frac{1}{\eps^2\ln\frac1\eps} 
    \int_{\Omega_{d_{\calC}\eps}} |L\xi_\eps+(L\xi_\eps)^T|^2 dx 
    \leq 
    c\,E<\infty,
\end{equation} 
and from Lemma~\ref{lemmaelemenmtarycontinuouscirc}\ref{lemmaelemenmtarycontinuouscirccurl} and $\Ce(\xi_\eps)\subseteq B_{m\eps}(\gamma_\eps)$ we know that $\Curl L\xi_\eps=0$ in $\Omega_{2k_*\eps}\setminus \overline B_{m\eps}(\gamma_\eps)$. From Proposition~\ref{propextendcylincircles} with $q=\frac32$ and $\rho_\eps:=2m\eps\ge 2k_*\eps$, we obtain that for $\eps$ sufficiently small, setting $r_\eps:=2k_\eps\to 0$ and $\rho_\eps^*:=12\rho_\eps/\alpha_\eps\to 0$, there is an extension $\hat\beta_\eps\in L^{\frac32}(\Omega_{r_\eps};\R^{3\times 3})$ with $\hat\beta_\eps=L\xi_\eps$ on $\Omega_{r_\eps}\setminus B_{\rho_\eps^*}(\gamma_\eps)$, $\curl\hat\beta_\eps=0$ on $\Omega_{r_\eps} \setminus \gamma_\eps$,
\begin{equation}\label{eq-estimate-compact-ext}
    \int_{\Omega_{r_\eps}} 
        |\hat\contbeta_\eps+\hat \contbeta_\eps^T|^{\frac32} 
    dx 
    <
    c 
    \int_{\Omega_{\rho_\eps}(\gamma_\eps)} 
        |L\xi_\eps+(L\xi_\eps)^T|^{\frac32}
    dx
\end{equation}
and
\begin{equation}\label{eqenergycompt}
    \int_{\Omega_{r_\eps}\setminus B_{\rho_\eps^*}(\supp\mu_\eps)}
        |\hat\contbeta_\eps+\hat \contbeta_\eps^T|^2 
    dx 
    <
    c 
    \int_{\Omega_{\rho_\eps}(\gamma_\eps)} 
        |L\xi_\eps+(L\xi_\eps)^T|^2
    dx ,
\end{equation}
where $\mu_\eps:=\curl\hat \beta_\eps\in\calM^1(\Omega_{r_\eps})$. By the assumption on $\alpha_\eps$, we see that \eqref{eqrhostarepseps} holds. We  observe that $\supp\mu_\eps\subseteq\gamma_\eps$ and that $\mu_\eps$ and $L\xi_\eps$ are $\rho_\eps^*$-compatible in $\Omega_{r_\eps}$. By Lemma \ref{lemmaelemenmtarycontinuouscirc}\ref{lemmaelementarycirc}, we have that $\mu_\eps\in \calM^1_{\eps\calB}(\Omega_{r_\eps})$. By \eqref{eqenergycompt} and \eqref{eqbdcpen} we have
\begin{equation}
    \frac{1}{\eps^2\ln \frac{1}{\eps}}
    \int_{\Omega_{r_\eps}\setminus B_{\rho_\eps^*}(\supp\mu_\eps)} 
        | \hat\contbeta_\eps+\hat \contbeta_\eps^T|^2 
    dx 
    \le 
        c\,E,
\end{equation}
and using Proposition~5.2 of \cite{ContiGarroniMarziani} (which gives the compactness for the continuum problem with core regularization, and is a variant of Proposition~6.6 of \cite{ContiGarroniOrtiz2015})
on any connected Lipschitz set $\Omega'\subset\subset\Omega$,
and then taking a further diagonal subsequence, we obtain that there is a subsequence such that
\begin{equation}\label{eqpropcompt1}
    \frac1{\eps_k}  \mu_{\eps_k}
    \weakto 
    \mu \text{ weakly as measures in $(C_c^0(\Omega;\R^{3\times 3}))^*$}
\end{equation}
for some $\mu\in \calM^1_\calB(\Omega)$. This proves~\eqref{eqpropcompt1b}.

We fix a connected Lipschitz set $\Omega'\subset\subset\Omega$ and apply the rigidity estimate in Proposition \ref{proprigiditycras} to the field $\hat\beta_{\eps_k}$, for $k$ sufficiently large. Using~\eqref{eqbdcpen} and~\eqref{eq-estimate-compact-ext}, we have that there exists a sequence of skew symmetric matrices $S_k$ such that
\begin{equation}\label{eq-compact-beta-hat}
\begin{split}
    \|L\discbeta_{\eps_k}-S_{k}\|_{L^{3/2}(\Omega'\setminus B_{\rho^*_{\eps_k}}(\gamma_{\eps_k})) }
    \leq &
    \| \hat\beta_{\eps_k}- S_{k}\|_{L^{3/2}(\Omega') }
    \\ \leq & 
    c E \eps_k\sqrt{\log\frac{1}{\eps_k}}
    +
    c|\mu_{\eps_k}|(\Omega'),
\end{split}
\end{equation}
where $c$ may depend on $\Omega'$. Recalling \eqref{eqpropcompt1}, this implies that the sequences $\frac{L\discbeta_{\eps_k}-S_k} {\eps_k(\log {\frac1{\eps_k}})^{1/2}}\chi_{\Omega'\setminus B_{\rho^*_{\eps_k}}(\gamma_{\eps_k})}$ and $\frac{\hat\beta_{\eps_k}-S_k} {\eps_k(\log {\frac1{\eps_k}})^{1/2}}$ are bounded in $L^\frac32(\Omega';\R^{3\times 3})$, and they agree outside $B_{\rho^*_{\eps_k}}(\gamma_{\eps_k})$. Note that, for every $s\in(0,1)$,
\begin{equation}\label{eq-rho-star}
 	\rho_\eps^*=\frac{24m\eps}{\alpha_\eps} \leq c \eps^s
\end{equation}
for $\eps$ small enough. From~\eqref{eq-rho-star} and Remark~\ref{rem-dilute-gamma}, we have that
\begin{equation}\label{eq-limit-gamma-eps}
    |B_{\rho^*_{\eps_k}}(\gamma_{\eps_k})|
    \leq 
    (\rho_{\eps_k}^*)^2\calH^1(\gamma_{\eps_k})
    \le 
    c\eps_k^{2s}\calH^1(\gamma_{\eps_k})\to 0
    \hskip3mm\text{ for any $s\in (0,1)$}.
\end{equation}
Therefore, we can extract a subsequence so that they both converge locally weakly to the same $\eta\in L^{\frac{3}{2}}_\loc(\Omega;\R^{3\times 3})$. This proves~\eqref{eqpropcompt3} and, by linearity,
\begin{equation}
    \curl \frac1{\eps_k(\ln \frac1{\eps_k})^{1/2}} \hat\contbeta_{\eps_k} 
    \weakto
    \curl \conteta \text{ distributionally}.
\end{equation}
Since $\mu_{\eps_k}=\frac1{\eps_k}\Curl\hat\contbeta_{\eps_k} \weakto\mu$ we obtain $\curl\conteta=0$.

In order to obtain \eqref{eqpropcompt4}, we directly apply Proposition \ref{proprigiditycras} to the field $L\xi_{\eps_k}$ in $\Omega'$. Using again Proposition~\ref{prop-Lcsi} and Lemma~\ref{lemmacoremollif} for $\curl L\xi_\eps$, we find that, for any open set $\Omega''$ with $\Omega'\subset\subset\Omega''\subset\subset\Omega$,
\begin{equation}\label{eq-Lxi-compact1}
\begin{split}
    \|L\discbeta_{\eps_k}-\hat S_{k}\|_{L^{3/2}( \Omega') }
    \le & 
    c E_{\eps_k}^{1/2}[\discbeta_{\eps_k},\Omega] 
    + 
    c|\curl L\discbeta_{\eps_k}|(\Omega')
    \\ \le & 
    c E_{\eps_k}^{1/2}[\discbeta_{\eps_k},\Omega] 
    +
    c m (\calH^1(\gamma_{\eps_k}\cap \Omega''))^{1/2}  E_{\eps_k}^{1/2}[\discbeta_{\eps_k},\Omega] ,
\end{split} 
\end{equation}
for some skew symmetric matrices $\hat S_k$ and $k$ sufficiently large. Therefore, using again Remark~\ref{rem-dilute-gamma} to estimate the length of $\gamma_{\eps_k}$,
\begin{equation}\label{eq-Lxi-compact1b}
\begin{split}
    \limsup_k 
    \frac{1}{\eps_k^q}
    \|L\discbeta_{\eps_k}-\hat S_{k}\|_{L^{3/2}( \Omega') }
    <
    \infty \text{ for every $q\in (0,1)$.}
\end{split}
\end{equation}
Combining this with \eqref{eq-compact-beta-hat}, we infer that for every such $q$ the sequence $|S_k-\hat S_k|/\eps_k^q$ is bounded and, therefore, \eqref{eq-Lxi-compact1b} also holds with $\hat S_k=S_k$.

By H\"older's inequality, for every $p\in [1,\frac32)$ we have
\begin{equation}\label{eq-Lxi-compact2}
\begin{split}
    \|L\discbeta_{\eps_k}-S_{k}\|_{L^{p}( \Omega'\cap B_{\rho^*_{\eps_k}}(\gamma_{\eps_k})) }
    \le  
    |B_{\rho^*_{\eps_k}}(\gamma_{\eps_k})|^{\frac1p-\frac{2}3}	\|L\discbeta_{\eps_k}- S_{k}\|_{L^{3/2}( \Omega') }.
\end{split} 
\end{equation}
Similarly to~\eqref{eq-limit-gamma-eps} we obtain $|B_{\rho^*_{\eps_k}} (\gamma_{\eps_k})|/\eps^{1-q}\to 0$ for every $q\in(0,1)$, and, with~\eqref{eq-Lxi-compact1b}, we deduce that
\begin{equation}\label{eqLxiS}
    \lim_{k\to \infty}
    \frac{1}{\eps_k(\log {\frac1{\eps_k}})^{1/2}}	\|L\discbeta_{\eps_k}-S_{k}\|_{L^{p}( \Omega'\cap B_{\rho^*_{\eps_k}}(\gamma_{\eps_k})) } 
    =
    0,
\end{equation}
which, together with~\eqref{eqpropcompt3}, concludes the proof of \eqref{eqpropcompt4}.

In order to show \eqref{eqpropcompt2} it is therefore enough to show that $\frac1{\eps_k}(\hat\contbeta_{\eps_k}-L\xi_{\eps_k})$ converges strongly to zero in $L^1_\loc(\Omega)$. We have,
\begin{equation}\label{eq-conv-L1}
\begin{split}
    \frac{1}{\eps_k}\|\hat\contbeta_{\eps_k}-L\xi_{\eps_k}\|_{L^1(\Omega')}
    & =	
    \frac{1}{\eps_k}
    \|\hat\contbeta_{\eps_k}-L\xi_{\eps_k}\|_{L^1(\Omega'\cap B_{\rho^*_{\eps_k}}(\gamma_{\eps_k}))}
    \\ & \leq 
    \frac{1}{\eps_k} 
    |B_{\rho^*_{\eps_k}}(\gamma_{\eps_k}))|^\frac13 \|\hat\contbeta_{\eps_k}-L\xi_{\eps_k}\|_{L^\frac32(\Omega')}.
\end{split}
\end{equation} Recalling \eqref{eq-compact-beta-hat}, ~\eqref{eqLxiS} and
$|B_{\rho^*_{\eps_k}} (\gamma_{\eps_k})|/\eps^{q}\to 0$ for every $q\in(0,1)$, concludes the proof.
\end{proof}

\subsection{Lower bound}\label{seccomp}

We now show the lower bound in the $\Gamma$-convergence result.

\begin{proposition}\label{prop:lower}
Under the assumptions of Theorem~\ref{theomainresult}, let $\discbeta_\eps\in \Ade(\Omega)$ be such that there is a measure $\mu\in \calM^1_\calB(\Omega)$ which satisfies
\begin{equation}\label{eqpropcompt2b}
    \frac1{\eps}  \curl L\discbeta_{\eps} 
    \weakto 
    \mu \text{ as distributions in $(C_c^\infty(\Omega;\R^{3\times 3}))^*$}
\end{equation}
and  
$S_\eps\in \R^{3\times 3}_\skw$, $\conteta\in L^{1}_\loc(\Omega;\R^{3\times 3})$ with $\curl\conteta=0$, such that
\begin{equation}\label{eqpropcompt3b}
    \frac{L\discbeta_{\eps}-S_\eps}
    {\eps(\log {\frac1{\eps}})^{1/2}}
    \weakto 
    \conteta \text{ weakly in } L^{1}_\loc(\Omega;\R^{3\times 3}).
\end{equation}
Then, $\conteta\in L^2(\Omega;\R^{3\times 3})$ and 
\begin{equation}\label{eqlowerboundprop}
    E_0[\mu,\eta,\Omega] 
    \le 
    \liminf_{\eps\to 0} 
    \frac{1}{\eps^2\ln \frac1{\eps}} 
    E_{\eps}[\discbeta_\eps,\Omega].
\end{equation}
The same holds if all objects are only defined along a sequence $\eps_k\to0$.
\end{proposition}

\begin{proof}
Without loss of generality we can assume that, up to a subsequence (not relabeled), 
\begin{equation}
    \lim_{\eps\to 0} 
    \frac{1}{\eps^2\ln \frac1{\eps}}
    E_\eps[\discbeta_\eps, \Omega]
    =
    \liminf_{\eps\to 0} 
    \frac{1}{\eps^2\ln \frac1{\eps}}
    E_\eps[\discbeta_\eps, \Omega]
    <
    \infty.
\end{equation}
Let $\gamma_\eps$ be curves such that $(\gamma_\eps,\discbeta_\eps)\in \Ade(\Omega,k_\eps,\alpha_\eps,m)$. From Proposition~\ref{prop:compact}, we obtain measures $\mu_\eps$ which converge weakly$^*$ to $\mu$ in the sense that
\begin{equation}\label{eqweakconvmuepsmu}
    \frac1\eps\mu_\eps
    \weakstarto
    \mu \hskip5mm
    \text{ weakly as measures in $(C_c^0(\Omega;\R^{3\times 3}))^*$}.
\end{equation}
Fix connected Lipschitz sets $\Omega'\subset\subset\Omega''\subset\subset\Omega$. Since $\supp\mu_\eps\cap\Omega''\subseteq\gamma_\eps$, which is $(k_\eps, \alpha_\eps)$-dilute, we may assume that
\begin{equation}\label{eq-misura-diluita-lower}
    \mu_{\eps}\LL\Omega''
    = 
    \sum_i b^i_\eps\otimes t_\eps^i 
    {\cal H}^1\LL(\gamma^i_\eps\cap\Omega''),
\end{equation}
where $b_\eps^i\in \eps\calB$ and $\gamma^i_\eps\subseteq\Omega$ are finitely many segments with $\calH^1(\gamma_\eps^i)\in[k_\eps,2k_\eps)$.

Let $\rho_\eps^*$ be as in \eqref{eq-rho-star} and define, for $\eps$ sufficiently small,
\begin{equation}\label{eqtildebetapsiastLbeta}
    \contbeta_\eps := 
    \psi_{\rho_\eps^*}\ast L\discbeta_\eps 
    \hskip1cm \text{ in } \Omega'',
\end{equation}
where $\psi_\delta\in C^\infty_c(B_{\delta/2})$ is the mollifier introduced in~\eqref{eqpsideltamoll}. We use Theorem \ref{theomollifydisccont} with $\delta= \rho^*_\eps/10$, $\omega=\Omega\setminus B_{m\eps}(\gamma_\eps)$ and $\omega'=\Omega''\setminus B_{\rho_\eps^*}(\gamma_\eps)$. Since $\Ce(\xi_\eps)\subseteq B_{m\eps}(\gamma_\eps)$ and $m\eps\le\rho_\eps^*/2$, for $\eps$ sufficiently small we obtain
\begin{equation}\label{eqestbetaepsM2}
    \int_{\Omega''\setminus B_{\rho_\eps^*}(\gamma_\eps)} 
        \frac12 \C \contbeta_\eps \cdot \contbeta_\eps\, 
    dx 
    \le 
    (1+c\frac{\eps}{\rho^*_\eps}) 
    E_\eps[\discbeta_\eps, \Omega].
\end{equation}

We remark that we cannot conclude using the $\Gamma$-convergence results for continuum fields from \cite{ContiGarroniOrtiz2015,ContiGarroniMarziani}. Indeed, the above formula only estimates the energy of $\beta_\eps$ away from a neighbourhood of $\gamma_\eps$. In contrast, the continuum results do not contain $\gamma_\eps$ explicitly, and assume compatibility of $\beta_\eps$ away from $\supp\mu_\eps$. We have shown above that $\mu_\eps$ is supported in a neighbourhood of $\gamma_\eps$ and it can be redefined to be supported on a subset of $\gamma_\eps$, but this could be a strict subset (cf. Proposition~\ref{propextendcylincircles} and comments below). This difficulty can be overcome by using the extension results. However, it is more transparent to use a direct approach and do a local estimate using diluteness.

From \eqref{eqpropcompt3b}, by the properties of mollifiers we obtain 
\begin{equation}\label{eqweakbetaconvmuepsmu}
    \frac{\beta_{\eps}-S_\eps}
    {\eps(\log {\frac1{\eps}})^{1/2}}
    \weakto 
    \conteta \text{ weakly in } L^{1}(\Omega';\R^{3\times 3})
\end{equation}
and by equiintegrability we obtain, setting $R_\eps:=(\alpha_\eps k_\eps)^2$, 
\begin{equation*}
\begin{split}
    \frac{\contbeta_{\eps}-S_{\eps}}{\eps(\log\frac1{{\eps}})^{1/2}} (1-\chi_{B_{R_\eps}(\gamma_\eps)})
    \weakto 
    \conteta \text{ weakly in } L^{1}(\Omega';\R^{3\times 3}).
\end{split}
\end{equation*}
Since $\C S_\eps=0$,
\begin{equation}\label{eqlbelast}
\begin{split}
    \int_{\Omega'} \frac12  \C\conteta \cdot \conteta\,  dx
    \le &
    \liminf_{\eps\to0} 
    \int_{\Omega'\setminus B_{R_\eps}(\gamma_\eps)}
        \frac{\C 
    	(\contbeta_{\eps}-S_{\eps})\cdot (\contbeta_{\eps}-S_{\eps})}
        {2{\eps^2}\log\frac1{{\eps}}} 
    dx
    \\ = & 
    \liminf_{\eps\to0} 
    \frac{1}{\eps^2\log\frac1{{\eps}}}
    \int_{\Omega''\setminus B_{R_\eps}(\gamma_\eps)} 
        \frac12 \C \contbeta_\eps \cdot \contbeta_\eps 
    dx,
\end{split}
\end{equation}
which will then be estimated with \eqref{eqestbetaepsM2}.

The core part is treated as in \cite[Prop.~6.6]{ContiGarroniOrtiz2015}. We denote by $I_\eps(\Omega')$ the set of indices $i$ in \eqref{eq-misura-diluita-lower} such that the segment $\gamma_\eps^i$ intersects $\Omega'$. If $\eps$ is small enough, we have that $B_{R_\eps}(\gamma_\eps^i)\subseteq\Omega''$ for all $i\in I_\eps(\Omega')$. We choose $\lambda_\eps:=\alpha_{\eps} k_{\eps}$ and let $\Gamma_\eps^i\subseteq\R$ be a segment of length $\calH^1(\gamma_\eps^i) - 2\lambda_\eps$, $A_\eps^i$ an affine isometry that maps $\Gamma_\eps^ie_3$ into $\gamma_\eps^i$ and the midpoint of $\Gamma_\eps^ie_3$ to the midpoint of $\gamma_\eps^i$. We define the cylinders
\begin{equation*}
    \cyl^i_\eps
    :=
    A_\eps^i(B'_{R_\eps}\times \Gamma_\eps^i)
    \text{ and }
    \hat \cyl^i_\eps:=A_\eps^i(B'_{\rho^*_\eps}\times \Gamma_\eps^i) \,.
\end{equation*}
By diluteness, the sets $\{\cyl_{\eps}^i\}_i$ are pairwise disjoint and $\cyl^i_\eps\cap \gamma^j_\eps=\emptyset$ for all $j\ne i$. Fix $M>1$, chosen below. For sufficiently small $\eps$, we have $\calH^1(\Gamma^i_\eps)\ge MR_\eps$ for all $i$. By Lemma~\ref{lemmaextendbetacyl}, for every $i$ there is an extension $\beta^i\in L^1(T^i_\eps;\R^{3\times 3})$ of $\beta_\eps|_{T^i_\eps\setminus \hat T^i_\eps}$ with $\curl\beta^i=b^i_\eps\otimes t^i_\eps \calH^1\LL A^i_\eps (\Gamma^i_\eps e_3)$. Using \eqref{eqlemma58} in each of them,
\begin{equation*}
\begin{split}
    &	
    \int_{B_{R_\eps}(\gamma_\eps)\cap\Omega''} 
        \frac12 \C\contbeta_{\eps}\cdot\contbeta_{\eps} 
    dx
    \\ & \geq 
    \sum_{i\in I_\eps(\Omega')}
    \calH^1(\Gamma_\eps^i) 
    \ln\frac{R_\eps}{\rho^*_\eps}
    \left(
        1-\frac{c}{M} 
        -  
        \omega_{M}
        \left(\frac{\rho^*_\eps}{R_\eps}\right)
    \right) \psi_\C(b_\eps^i, t_\eps^i)\,.
\end{split}
\end{equation*}
Recalling that $\calH^1(\Gamma_\eps^i)\ge (1-2\alpha_{\eps}) \calH^1(\gamma_\eps^i)$, we have\begin{equation*}
\begin{split}
    & 	
    \int_{B_{R_\eps}(\gamma_\eps)\cap\Omega''} 
    \frac12 \C\contbeta_{\eps}\cdot\contbeta_{\eps} dx
    \\ & \ge  
    \ln\frac{R_\eps}{\rho^*_\eps} 
    \sum_i 
    \calH^1(\gamma_\eps^i) (1-2\alpha_{\eps}) 
    \left( 
        1-\frac{c}{M}
        -  
        \omega_{M}\left(\frac{\rho^*_\eps}{R_\eps}\right)
    \right) 
    \psi_\C(b_\eps^i,  t_\eps^i)  
    \\ & \ge 
    \ln\frac{R_\eps}{\rho^*_\eps}
    \left(
        1-\frac{c}{M} 
        -  
        \omega_{M}\left(\frac{\rho^*_\eps}{R_\eps}\right)
    \right) 
    (1-2\alpha_{\eps})
    F_0[\mu_{\eps},\Omega'] ,
\end{split}
\end{equation*}
where
\begin{equation}
    F_0[b\otimes t\calH^1\LL\gamma,\Omega]
    :=
    \int_{\gamma\cap\Omega} \psi_\C(b,t)d\calH^1.
\end{equation}
Combining with \eqref{eqlbelast}, taking the limit as $\eps\to0$ and using that $\psiC(\cdot, t)$ is positively 2-homogeneous, \eqref{eqdefheps} and~\eqref{eqrhostarepseps}, gives
\begin{alignat*}1
 	\liminf_{\eps\to0}
 	\frac{1}{\eps^2\ln\frac1\eps}
 	\int_{\Omega''} 
 	\frac12\C\contbeta_{\eps}\cdot \contbeta_{\eps} dx    
 	\ge
 	(1-\frac{c}{M})\liminf_{\eps\to0} F_0[\frac1\eps\mu_{\eps},\Omega']
 	+ 
 \int_{\Omega'} \frac12 \C\conteta \cdot\conteta dx.
\end{alignat*}
Using that $\int_{\gamma\cap\Omega} \psi_\C^\rel(b,t)d\calH^1$ is the lower semicontinuous envelope of $F_0$, the convergence in \eqref{eqweakconvmuepsmu} and that $M$ is arbitrary, the lower bound is proven.

Since $\curl\eta=0$, from $\eta+\eta^T\in L^2(\Omega;\R^{3\times 3})$ and Korn's inequality we obtain $\eta\in L^2(\Omega;\R^{3\times 3})$.
\end{proof}
 
\subsection{Discrete upper bound}
\label{secdiscreteupperbound}

The upper bound is based on discretization of the strain field associated to a polyhedral distribution of dislocations $\mu$. As shown in \cite[Theorem~4.1]{ContiGarroniOrtiz2015}, given a bounded measure  $\mu\in\calM(\R^3;\R^{3\times 3})$ with $\Div\mu=0$, the solution $\contbeta\in L^{3/2}(\R^3;\R^{3\times 3})$ of the problem
\begin{equation}\label{eqdivbmur3}
    \Div \C\contbeta=0 \text{ and } \Curl \contbeta=\mu
\end{equation}
(denoted by $\beta_\mu$ elsewhere) is unique, satisfies
\begin{equation*}
    \|\contbeta\|_{L^{3/2}(\R^3)} \le c |\mu|(\R^3)\,,
\end{equation*}
and has the integral representation 
\begin{equation}\label{eqbetaNmu}
    \contbeta_{ij}(x)
    =
    \sum_{k,l=1}^3\int_{\R^3} N_{ijkl}(x-y) d\mu_{kl}(y)
\end{equation}
(for $x\not\in\supp\mu$).
Here, $N\in C^\infty(\R^3\setminus\{0\};\R^{3\times 3\times 3\times 3})$ depends only on $\C$ and is positively $(-2)$-homogeneous. If $\mu=\sum_i b_i \otimes t_i\calH^1\LL\gamma_i$ for finitely many compact segments $\gamma_i$, then from this representation shows that $\beta\in C^\infty(\R^3\setminus \cup_i \gamma_i;\R^{3\times 3})$ and
\begin{equation}\label{eqDcontbetadist1}
    |\contbeta(x)|\le c \sum_i \frac{|b_i|}{\dist(x, \gamma_i)}.
\end{equation}
Moreover (see \cite[Prop.~6.1]{ContiGarroniMarziani}), if $\Omega\subseteq\R^3$ is an open set with $\mu(\partial\Omega)=0$, then
\begin{equation}\label{eqpwcontlim}
    \lim_{\eps\to0} \frac{1}{\ln \frac1\eps}
    \int_{\Omega\setminus B_\eps(\gamma)} 
        \frac12 \C \contbeta\cdot  \contbeta 
    \, dx  
    = 
    \int_{\Omega\cap\gamma}  \psiC(b,t) d\calH^1.
\end{equation} 
For the discretization procedure we need a decay estimate for the gradient of $\beta$, which is derived from this representation in the following Lemma.
\begin{lemma}
If $\mu=\sum_i b_i \otimes t_i\calH^1\LL\gamma_i$ for finitely many segments $\gamma_i$, then the function $\beta$ in \eqref{eqbetaNmu} satisfies
\begin{equation}\label{eqDcontbetadist2}
    |D\contbeta(x)|\le c \sum_i \frac{|b_i|}{\dist^2(x, \gamma_i)}\,.
\end{equation}
\end{lemma}

\begin{proof}
We observe that
\begin{equation}
    D\beta(x)
    =
    \sum_i \int_{\gamma_i} DN(x-y)b_i\otimes t_i d\calH^1(y)
\end{equation}
and that $DN({z})=|z|^{-3}DN(z/|z|)$, so that
\begin{equation}
    |D\beta|(x)
    \le 
    \|DN\|_{L^\infty(\partial B_1)}
    \sum_i \int_{\gamma_i} \frac{|b_i|}{|x-y|^3} d\calH^1(y).
\end{equation}
We estimate each term in the sum by
\begin{equation*}
\begin{split}
    \int_{\gamma_i} \frac{|b_i|}{|x-y|^3} d\calH^1(y)
    & \le 
    \int_{\R} \frac{|b_i|} {(\dist^2(x,\gamma_i)+s^2)^{3/2}} ds
    \\ & \le 
    \frac{|b_i|}{\dist^2(x,\gamma_i)} \int_\R \frac{1}{(1+s^2)^{3/2}} ds ,
\end{split}
\end{equation*}
and obtain \eqref{eqDcontbetadist2}.
\end{proof}

\begin{proposition}\label{prop:upperbound}
Under the assumptions of Theorem~\ref{theomainresult}, let $\mu=b\otimes t \calH^1\LL\gamma\in \calM_\calB^1(\Omega)$, $\conteta\in L^{2}(\Omega;\R^{3\times 3})$ with $\curl\conteta=0$. Then, for every sequence $\eps_k\to0$ there is a sequence $\discbeta_{\eps_k}\in \Ade(\Omega)$ converging to $(\conteta,\mu)$ in the sense of \eqref{eqconvbeta} and \eqref{eqconvmu} with $S_k=0$, such that
\begin{equation}\label{equbpropeq1}
    \limsup_{k\to\infty} 
    \frac{1}{\eps_k^2\log\frac1{\eps_k}} 
    E_{\eps_k}[\discbeta_{\eps_k},\Omega]
    \le 
    \int_\Omega\frac 12\C \conteta\cdot \conteta dx 
    + 
    \int_{\gamma\cap\Omega} \psiC^\rel(b,t)d\calH^1.
\end{equation}
\end{proposition}
\begin{proof}
By a density argument based on the results from \cite{ContiGarroniMassaccesi2015}, as discussed, e.~g., in \cite{ContiGarroniOrtiz2015} or \cite{ContiGarroniMarziani}, it suffices to construct a sequence such that \eqref{equbpropeq1} holds with $\psi_\C^\rel$ replaced by $\psi_\C$, and $\mu\in\calM^1_\calB(\R^3)$ supported on a polygonal with $|\mu|(\partial\Omega)=0$. We write $\mu=\sum_i \theta_i\otimes t_i\calH^1\LL\gamma_i$. To simplify notation we write $\eps$ for $\eps_k$.

\emph{Step 1: Continuum construction.}
Let $\contbeta\in L^{3/2}(\R^3;\R^{3\times 3})$ be the solution to \eqref{eqdivbmur3}. By \eqref{eqpwcontlim} it satisfies 
\begin{equation}\label{eq-upper-bound1}
    \lim_{\eps\to0} \frac{1}{\ln 1/\eps}
    \int_{\Omega'\setminus B_{\eps}(\gamma)} 
        \frac12\C\contbeta\cdot\contbeta 
    \, dx 
    = 
    \int_{\Omega'\cap\gamma}  \psiC(b,t) d\calH^1
\end{equation} 
for any open set $\Omega'$ with $\Omega\subset\subset\Omega'$ and $|\mu|(\partial\Omega')=0$.

For notational simplicity we work in the case of $\calB$ three-dimensional; the cases $\dim_\Z\calB\in\{1,2\}$ can be treated in the same way.

Since the set of slip systems is complete (Definition~\ref{defbetapl2}), we can select Burgers vectors $\hat b_1$, $\hat b_2$ and $\hat b_{3}$ such that $\calB=\Span_\Z\{\hat b_1, \hat b_2, \hat b_{3}\}$ and such that, for every $j$, there are $m_j'$, $m_j''$ linearly independent, with $m_j'\cdot \hat b_j=m_j''\cdot \hat b_j=0$ and all pairs $(\hat b_j, m_j')$, $(\hat b_j, m_j'')$ are slip systems. For every segment $i$, there are uniquely determined coefficients $c_i^j\in\Z$ such that $\theta_i=\sum_j c_i^j \hat b_j$. Without loss of generality, we can assume that no segment $\gamma_i$ is parallel to one of the three vectors $\hat b_j$ (otherwise it suffices to slightly rotate $\mu$).

We define the three vector measures $\mu_j:=\sum_i c_i^j t_i\calH^1 \LL\gamma_i$ and observe that $\mu=\sum_{j=1}^3 \hat b_j\otimes \mu_j$ and $\Div \mu_j=0$ for all $j$. For every segment $\gamma_i$, we consider the surface $\Sigma_i^j:=\gamma_i + [0,\infty) \hat b_j$ and define the measures
\begin{equation}\label{eqdefhatbetaj}
    \hat\contbeta_j
    :=
    -\sum_i c_i^j \nu_i^j \calH^2 \LL \Sigma_i^j ,
\end{equation}
where $\nu_i^j:=\frac{t_i\wedge \hat b_j}{|t_i\wedge \hat b_j|}$, see Figure \ref{figureslipsurface}. For later reference, we remark that there is $C_\mu>0$ such that
\begin{equation}\label{eqcijbounded}
    |c_i^j| \le C_\mu \text{ for all $i$, $j$.}
\end{equation}
\begin{figure}\label{fig:slip-surface}
\centering
{\footnotesize \def\svgwidth{300pt}
    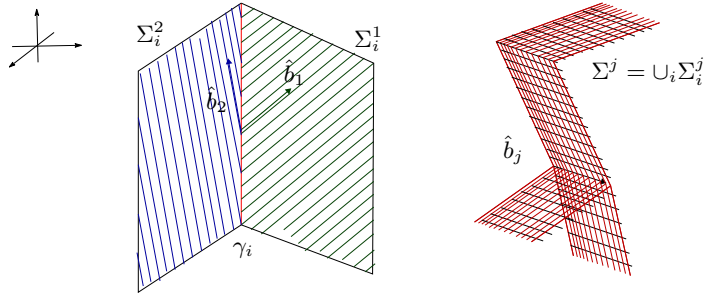}
\caption{Sketch of the construction of the slip surfaces $\Sigma_i^j$ from the segments $\gamma_i$.}
\label{figureslipsurface}
\end{figure}%
For every $i$, we denote by $a_i$ and $d_i$ the end points of the segment $\gamma_i$. The boundary of $\Sigma_i^j$ is given by $\gamma_i\cup (a_i+[0,\infty)\hat b_j)\cup (d_i+[0,\infty)\hat b_j)$ and, therefore,
\begin{equation*}\label{eq-curlbeta}
    \Curl(\nu_i^j\calH^2\LL \Sigma_i^j)
    = 
    t_i\calH^1\LL\gamma_i 
    + 
    \hat b_j \calH^1\LL (d_i+[0,\infty)\hat b_j)
    -
    \hat b_j \calH^1\LL (a_i+[0,\infty)\hat b_j).
\end{equation*}
Using that $\Div\mu_j=0$ in $\R^3$, it is readily verified that
\begin{equation}
    \curl \hat\contbeta_j
    = 
    -\mu_j \hskip1cm \text{ distributionally in $\R^3$}.
\end{equation}
We define
\begin{equation}\label{eqdefcontbetapl}
    \contbetapl:=\sum_j \hat b_j \otimes \hat\contbeta_j ,
\end{equation}
and observe that $\curl\contbetapl=-\mu$. As $\beta+\contbetapl$ is a curl-free Radon measure, there exists $u\in SBV_\loc(\R^3;\R^3)$ such that $Du=\contbeta+ \contbetapl$.  We let $\Sigma:=\cup_i\cup_j\Sigma_i^j$, so that $J_u\subseteq\Sigma$.
 
\emph{Step 2: Discretization of the displacement.}
To simplify notation, we assume that $\eps\calL \cap \Sigma=\emptyset$. Otherwise, we select $y\in B_\eps$ such that $(y+\eps\calL )\cap \Sigma=\emptyset$, and replace $u$ by $u(\cdot - y)$, and correspondingly for $\contbeta$ and $\contbetapl$. Since we are dealing with countably many values of $\eps$, we can choose a common value of $y$ which does not depend on $\eps$.

We define
\begin{equation}
    u_\eps(x)
    :=
    \begin{cases}
        \eps u(x), & 
        \text{ for } x\in \eps \calL \setminus B_{2\eps}(\gamma),\\
        \eps (u\ast \psi_\eps)(x), & 
        \text{ for } x\in \eps \calL \cap B_{2\eps}(\gamma),
    \end{cases}
\end{equation}
where, as usual, $\psi_\eps\in C^\infty_c(B_\eps)$ is a mollifier. We claim that there is a constant $C_\mu$, depending on $\mu$ but not on $\eps$, such that
\begin{equation}\label{eqdepsuepsn}
    |d_\eps u_\eps|(x,h)
    \le 
    C_\mu \text{ for all } x\in \eps\calL, h\in\calN.
\end{equation}
In the proof of \eqref{eqdepsuepsn}, we denote by $c>0$ a generic constant that depends on $\mu$ and may vary from line to line. We first observe that for any $z\in\R^3$, 
\begin{equation}\label{eqDubepseps2}
    |Du|(B_\eps(z))\le c\eps^2.
\end{equation}
To see this, we treat separately the elastic and the plastic part of $Du=\contbeta+\contbetapl$. As for the elastic part, from \eqref{eqDcontbetadist1} we have $|\contbeta(x)|\le c/\dist(x,\gamma)$, with $\gamma$ contained in a finite union of straight lines. For each straight line, an explicit integration (which can be performed assuming that the line contains $z$) gives the required estimate. For the plastic part we use that $|\contbetapl|\le c \calH^2\LL \Sigma$, where $\Sigma$ is a subset of a finite union of planes, so that $\calH^2(B_\eps\cap\Sigma) \le \pi N \eps^2$, where $N$ is the number of planes. This proves \eqref{eqDubepseps2}

In addition, \eqref{eqDubepseps2} implies the pointwise bound
\begin{equation}
    |D(u\ast \psi_\eps)|(x)
    \le 
    \int_{B_\eps(x)} \psi_\eps(x-y) d|Du|(y)  
    \le 
    c \eps^{-3} |Du|(B_\eps(x)) \le \frac c\eps,
\end{equation}
so that, for every $x\in\R^3$ and $h\in \calN$, we have
\begin{equation}\label{equasbvarphieps}
    |(u\ast\psi_\eps)(x)-(u\ast \psi_\eps)(x+\eps h)| 
    \le 
    \eps |h| \|D(u\ast \psi_\eps)\|_{L^\infty} \le c.
\end{equation}

In order to relate $u\ast \psi_\eps$ with $u$, we observe that, if $x\not\in B_{2\eps}(\gamma)$, then $\curl\contbeta=\curl\contbetapl=0$ in $B_\eps(x)$, so that two functions $u^{\rm el}$, $u^{\rm pl}\in SBV(B_\eps(x);\R^3)$ exist with $Du^{\rm el}=\contbeta$, $Du^{\rm pl}=\contbetapl$. Furthermore, again from \eqref{eqDcontbetadist1} we have $|\contbeta(z)|\le c/\eps$ for every $z\in B_\eps(x)$, so that $|u^{\rm el}(z)-u^{\rm el}(x)|\le c$ for all $z\in B_\eps(x)$.
By~\eqref{eqdefhatbetaj} and~\eqref{eqcijbounded}, $|u^{\rm pl}(z)-u^{\rm pl}(x)|\le c$ is also obtained for all $z\in B_\eps(x)$. Therefore, $|u(z)-u(x)| \leq c$ for all $z\in B_\eps(x)$. By convexity of the norm, we conclude
\begin{equation}
    |u-(u\ast\psi_\eps)|(x)
    \le 
    c \text{ for all } x\not\in B_{2\eps}(\gamma).
\end{equation}
Recalling the definition of $u_\eps$, this leads to
\begin{equation}\label{eqdistuuepsmoll}
    |u_\eps-\eps (u\ast\psi_\eps)|(x)
    \le 
    c\eps  \text{ for all } x\in \eps\calL.
\end{equation}
Fix now $x\in\eps\calL$ and $y=x+\eps h$, with $h\in\calN$. Using \eqref{equasbvarphieps} and \eqref{eqdistuuepsmoll},
\begin{equation}
\begin{split}
    |d_\eps u_\eps|(x,h)
    = &
    \frac{|u_\eps(x+\eps h)-u_\eps(x)|}{\eps}
    \\ \le &
    |(u\ast \psi_\eps)(x+\eps h)-(u\ast\psi_\eps)(x)| 
    +
    \frac{2\sup_{z\in\eps\calL} 
    |u_\eps-\eps u\ast \psi_\eps|(z)}{\eps}
    \\ \le & 
    c.
\end{split}
\end{equation}
This concludes the proof of \eqref{eqdepsuepsn}.
 
\emph{Step 3. Discretization of the plastic strain.}

Next, we construct $\discbetapl_\eps$. For any edge $(x,h)\in\bonds_\eps^\Omega$, we let 
\begin{equation}\label{eq-def-xipl-upper}
    \discbetapl_\eps(x,h)
    :=
    \sum_{j=1}^{n'} \hat b_j q^\eps_j(x,h),
\end{equation} 
with
\begin{equation}\label{eq-xipl-upper}
    q^\eps_j (x,h)
    := 
    \sum_{i: [x,x+\eps h] \cap\Sigma_i^j\ne\emptyset} 
    c_i^j \text{sign}(h\cdot \nu_i^j)\in \Z,
\end{equation}
where $\text{sign}(t)=0$ if $t=0$, so that bonds which are parallel to $\Sigma_i^j$ are not included in the sum. In order to show that this an admissible plastic slip in the sense of Definition \ref{defbetapl}, we need to verify that, for every $x$ and $h$, there are $\zeta^\eps_\ell(x,h)\in\R$ such that
\begin{equation}
    \discbetapl_\eps(x,h)
    =
    \sum_{\ell=1}^{N_s} 
    \discxi^\eps_\ell(x,h)\burgersbi{\ell} m_\ell\cdot h.
\end{equation}
and $\zeta^\eps_\ell(x,h) m_\ell\cdot h\in \Z$ for all $\ell$, $x$ and $h$. Fix $j\in \{1,\dots, n'\}$. We remark that the bonds $h$ which are included in the sum \eqref{eq-xipl-upper} are not parallel to $\hat b_j$ and, therefore, they have a nonzero projection on the space $\hat b_j^\perp$. In turn, the two vectors $m_j'$ and $m_j''$ are linearly independent and orthogonal to $\hat b_j$. Hence, they span the two-dimensional space $\hat b_j^\perp$ and there is one of them which is not orthogonal to $h$. This implies that there is $\ell\in\{1,\dots, N_s\}$ such that $b_\ell=\hat b_j$ and $m_\ell\cdot h\ne 0$. We define 
\begin{equation}
    \discxi^\eps_{\ell,j}(x,h)
    := 
    \frac{  q^\eps_j(x,h)}{m_\ell\cdot h},
\end{equation}
the other components being zero, and, finally, $\discxi^\eps_\ell:=\sum_j \discxi^\eps_{\ell,j}$. Therefore, $\discbetapl_\eps$ is a discrete plastic strain. 

We define $\discbetael_\eps:=d_\eps u_\eps-\discbetapl_\eps$ and observe that
$(\gamma,\discbetael_\eps)\in \Ade(\Omega, k_\eps, \alpha_\eps. m)$ and, for $(x,h)\in \bonds^\Omega_\eps$ with $\dist(x,\gamma)\ge 2\eps$ and $\dist(x+\eps h,\gamma)\ge 2\eps$, recalling \eqref{eqdefhatbetaj} and \eqref{eqdefcontbetapl},
\begin{equation}\label{eqxielepsbetaint}
\begin{split}
    \discbetael_\eps(x,h)
    = &
    d_\eps u_\eps(x,h)-\discbetapl_\eps(x,h)
    \\ = &
    u(x+\eps h)-u(x) - \sum_j\hat b_j q_j^\eps(x,h)
    \\ = &
    \int_0^\eps  \nabla u(x+th)h dt
    =
    \int_0^\eps  \contbeta(x+th)h dt.
\end{split}
\end{equation}

\emph{Step 4: Estimate of the elastic energy.} 
Choose $\sigma\in(0,1)$ and consider a cluster $x+\eps \calC$, $x\in \clust_\eps^\Omega$. We define a representative strain in this cluster by $F_x := \eps(\psi_\eps\ast \contbeta)(x)$. Since $E_\calC$ is a quadratic form, we have
\begin{equation}\label{eqxielclust}
    E_\calC[\tau^x_\eps\discbetael_\eps] 
    \le 
    (1+\sigma) E_\calC[\discbeta^{F_x}] 
    + 
    c_\sigma E_\calC[\tau^x_\eps\discbetael_\eps-\discbeta^{F_x}],
\end{equation}
where we recall that the map $\discbeta^{F_x}:\calC_\calN\to\R^3$ is defined from the matrix $F_x$ as $\discbeta^{F_x}(x,h)=F_xh$.
 
The first term is controlled recalling the definition of the elastic coefficients $\C$ and Poincar\'e's inequality, with the result,
\begin{equation}\label{eq-estimate-elastic-upper}
\begin{split}
    \eps^3 E_\calC[\discbeta^{F_x}] 
    & = 
    \frac12 \calL^3(\eps T_*) F_x\cdot \C F_x
    \\ & \le  
    (1+\sigma)\eps^2\int_{x+\eps T_*} 
    \frac12\contbeta\cdot \C \contbeta dy 
    + 
    c_\sigma \eps^4\int_{B_{d_{T_*}\eps}(x)} |D\contbeta|^2 dy .
\end{split}
\end{equation}
To estimate the second term of \eqref{eqxielclust}, we use \eqref{eqclusterfromabove}, which leads to
\begin{equation}
    E_\calC[\tau^x_\eps\discbetael_\eps-\discbeta^{F_x}]
    \le 
    c \sum_{(y,h)\in \calC_\calN}
    |\discbetael_\eps(x+\eps y,h)-F_xh|^2.
\end{equation}
If $\dist(x,\gamma)\ge 2d_{T_*}\eps$, recalling \eqref{eqxielepsbetaint}, \begin{equation}\begin{split}
    |\discbetael_\eps(x+\eps y,h)-\eps\beta(x)h|
    \le &
    \int_0^\eps |\beta(x+\eps y+t h)-\beta(x)|\,|h|dt
    \\ \le &
    c \eps^2 \|D\contbeta\|_{L^\infty(B_{d_{\calC}\eps}(x))} .
\end{split}
\end{equation}
By the definition of $F_x$,
\begin{equation}
    |F_x-\eps\beta(x)|
    =
    \eps |(\psi_\eps\ast\beta)(x)-\beta(x)|
    \le 
    c \eps^2 \|D\contbeta\|_{L^\infty(B_{\eps}(x))} 
\end{equation}
so that with~\eqref{eqDcontbetadist2}, if $\dist(x,\gamma)\ge 2 d_\calC\eps$,
\begin{equation}
     |\discbetael_\eps(x+\eps y,h)-F_xh|
     \le 
     c \eps^2 \|D\contbeta\|_{L^\infty(B_{d_{\calC}\eps}(x))} 
     \le 
     \frac{c\eps^2}{\dist^2(x,\gamma)} ,
\end{equation}
and \eqref{eqxielclust} becomes
\begin{equation}
    \eps^3  E_\calC[\tau^x_\eps\discbetael_\eps]
    \le 
    (1+\sigma)^2 \eps^2
    \int_{x+\eps T_*} \frac12\contbeta\cdot \C \contbeta dy 
    + 
    c_\sigma \eps^4
    \int_{B_{d_\calC\eps}(x)}\frac{c}{\dist^4(x,\gamma)} dy,
\end{equation}
for all $x$ such that $\dist(x,\gamma)\ge 2 d_\calC\eps$. Summing over all $x$ which are at distance at least $2d_\calC\eps$ from the singularity and recalling that the balls have finite overlap,
\begin{equation}\label{eq-finally-upper-elastic}
\begin{split}
    &
    \sum_{x\in  \clust_\eps^\Omega , \dist(x,\gamma)
    \ge 
    2d_\calC\eps} \eps^3 E_\calC[\tau^x_\eps\discbetael_\eps]
    \le 
    (1+\sigma)^2 \eps^2
    \int_{\Omega'\setminus B_{d_\calC\eps}(\gamma)} 
        \frac12\contbeta\cdot \C \contbeta 
    dy 
    \\ & \qquad\qquad + 
    c_\sigma \eps^4
    \int_{\Omega'\setminus B_{d_\calC\eps}(\gamma)}
        \frac{1}{\dist^4(x,\gamma)}
    dy . 
\end{split}
\end{equation}
The first term is estimated by \eqref{eq-upper-bound1}, the second by a direct integration,
\begin{equation}\label{equbdist4}
    \eps^2
    \int_{\Omega'\setminus B_{d_\calC\eps}(\gamma)} 
        \frac{1}{\dist^4(x,\gamma)}
    dy
    \le 
    C_\mu.
\end{equation}
It remains to estimate the contribution of the core region. By \eqref{eqdepsuepsn} and \eqref{eqcijbounded}, we have that $|\discbetael_\eps|(x,h)\le C_\mu$ everywhere. Therefore,
\begin{equation}
\begin{split}
    \sum_{x\in  \clust_\eps^\Omega , \dist(x,\gamma)< 2d_\calC\eps}
    E_\calC[\tau^x_\eps\discbetael_\eps]
    \le 
    & C_\mu  \#\{x\in\clust_\eps^\Omega , \dist(x,\gamma)< 2d_\calC\eps\}.
\end{split}
\end{equation}
For every segment $\gamma_i$, we have
\begin{equation}
\begin{split}
    \#\{x\in\clust_\eps^\Omega , \dist(x,\gamma_i)< 2d_\calC\eps\}
    \le 
    c +c\frac{\calH^1(\gamma_i)}{\eps} ,
\end{split}
\end{equation}
where the constant only depends on the lattice. Therefore,
\begin{equation}\label{equbthree}
\begin{split}
    \limsup_{\eps\to0}\frac{1}{\eps^2\ln\frac1\eps}
    \sum_{x\in  \clust_\eps^\Omega , \dist(x,\gamma)< 2d_\calC\eps} 
    \eps^3 E_\calC[\tau^x_\eps\discbetael_\eps]
    \le & 
    C_\mu \limsup_{\eps\to0}\frac{1}{\ln\frac1\eps}
    =
    0 .
\end{split}
\end{equation}
Collecting terms, \eqref{eq-upper-bound1}, \eqref{eq-finally-upper-elastic}, \eqref{equbdist4} and \eqref{equbthree}, concludes the proof of the upper bound.
\end{proof}

\appendix
\section{Extension}
\label{secappext}
\subsection{Rigidity}

We first recall a rigidity result from \cite{ContiGarroniRigidity}, which is based on Korn's inequality and on the Bourgain-Brezis critical integrability bound \cite{BourgainBrezis2007}, and extend it to smaller exponents.
\begin{proposition}[From \cite{ContiGarroniRigidity}]\label{proprigiditycras}
Let $\Omega\subseteq\R^3$ be a bounded connected Lipschitz set, $q\in[1,\frac32]$. Then, there is a constant $c=c(\Omega,q)$ such that for any $\contbeta\in L^q(\Omega;\R^{3\times 3})$ there is ${S}\in \R^{3\times 3}_\skw$ such that \begin{equation}
    \|\contbeta-{S}\|_{L^{q}(\Omega)} 
    \le 
    c \|\contbeta+\contbeta^T\|_{L^{q}(\Omega)} 
    +
    c|\curl\contbeta|(\Omega).
\end{equation}
\end{proposition}
{The corresponding statement holds in any dimension and in the geometrically nonlinear case, with the same proof.}
\begin{proof}
The important case is the critical exponent $q=\frac32$, which was proven in \cite{ContiGarroniRigidity}. The case $q\in [1,\frac32)$ follows from the same argument, and we briefly mention the changes using the notation of \cite{ContiGarroniRigidity} and referring to formulas in that proof.
{For clarity we work in generic dimension $n$, proving the assertion for $q\le 1^{*}=(n-1)/n$.
One starts with the unit cube; \cite[eq.~(8)]{ContiGarroniRigidity}
holds unchanged and gives by Hölder a bound for $Y$ in $L^q(Q)$, $Q:=(0,1)^n$. Equation
\cite[eq.~(9)]{ContiGarroniRigidity}
becomes
\begin{equation}
 \|Du+Du^T\|_{L^q(Q)}\le \|\beta+\beta^T\|_{L^q(Q)}
 + c \|\curl\beta\|_{L^1(Q)},
\end{equation}
which by Korn's inequality implies the assertion if $\Omega=Q$.

After scaling to $Q_r=(0,r)^n$, this leads to
(cf. \cite[eq.~(11), note that the right-hand side of that equation should contain $\beta$ instead of $Du$)]{ContiGarroniRigidity})
\begin{equation}
 \|\beta-R\|_{L^q(Q_r)}\le \|\beta+\beta^T\|_{L^q(Q_r)}
 + c |Q_r|^{\frac 1q-\frac1{1^{*}}}
 \|\curl\beta\|_{L^1(Q_r)}.
\end{equation}

The next part of the proof, up to (16), is unchanged, replacing $s$ by $q$. Equation (17) becomes
\begin{equation}\begin{split}
\int_\Omega \dist^q(x,\partial\Omega)|DR|^qdx\le& c \sum_j \left[
\|\beta+\beta^T\|^q_{L^q(Q^j)} +
c|Q_j|^{1-\frac{q}{1^*}}
|\curl\beta|(Q_j)^q\right].
                \end{split}
\end{equation}
Since $q\le 1^*$, the factor $|Q_j|^{1-\frac q{1^*}}$ is uniformly bounded and can be dropped. Therefore, as in
\cite[eq.~(17)]{ContiGarroniRigidity},
\begin{equation}\begin{split}
\int_\Omega \dist^q(x,\partial\Omega)|DR|^qdx\le& c
\|\beta+\beta^T\|^q_{L^q(\Omega)} +
c\sum_j
|\curl\beta|(\Omega)^{q-1}
|\curl\beta|(Q_j)\\
\le &c
\|\beta+\beta^T\|^q_{L^q(\Omega)} +c
|\curl\beta|(\Omega)^{q}
.                \end{split}
\end{equation}
This concludes the proof.}
\end{proof}

\subsection{Extension in cylinders}

\begin{proof}[Proof of Lemma \ref{lemmaextendbetacyl}]
\emph{Step 1. We show that, for every $d\in\N$, $\contbeta\in L^p(\omegaout;\R^{d\times 3})$ with $\curl\contbeta=0$ in $\omegaout$, there is $\hat \contbeta\in L^p(\omega;\R^{d\times 3})$ such that $\hat\contbeta=\contbeta$ on $\omegaout$, $\curl\hat\contbeta=0$ on $\omega\setminus(\{0\}\times (0,\ell))$, and
\begin{equation}\label{eqlemmaextendbetacyl}
    \|\hat\contbeta\|_{L^p(\omegain)}
    \le 
    C \|\contbeta\|_{L^p(\omegaout)}.
\end{equation}
}
By scaling, we can assume $\rho=1$. Since the curl is computed row-wise, it suffices to consider the case that $\contbeta$ is a vector, i.~e., $d=1$, and we can identify column vectors with row vectors. For $\lambda\in (\frac32,2)$ chosen below, we set $\theta(t):=\lambda+(1-\lambda)t$. We observe that $\theta$ is a bijective map from $[0,1]$ onto $[1,\lambda]$ and that $\theta(1)=1$. We define $\varphi:\omegain\setminus(\{0\}\times (0,\ell))\to\omegaout$ as
\begin{equation}
    \varphi(x):=\frac{x'}{|x'|}\theta\big(|x'|\big)+ x_3e_3,
\end{equation}
where $x':=(x_1,x_2,0)$ denotes the projection of $x$ onto the plane spanned by $e_1$ and $e_2$. We compute
\begin{equation}
    D\varphi(x)
    =
    \Big(\Id_2-\frac{x'\otimes x'}{|x'|^2}\Big) 
    \frac{1}{|x'|} \theta\big(|x'|\big) 
    + 
    \frac{x'\otimes x'}{|x'|^2} 
    \theta'\big(|x'|\big)+ e_3\otimes e_3,
\end{equation}
where $\Id_2:=e_1\otimes e_1+e_2\otimes e_2{\in\R^{3\times 3}}$ is the identity matrix in the first $2\times 2$ block. In particular, $D\varphi(x)$ is symmetric and has eigenvalues 
\begin{equation}\label{eqeigenvaldvarphi}
    \frac{\lambda}{|x'|}+1-\lambda,\hskip2mm  {1-\lambda},\hskip2mm 1.
\end{equation}
We define
\begin{equation}
    \hat\contbeta(x)
    :=
    \begin{cases}
        \contbeta(x), & \text{ if } x\in\omegaout,\\
        D\varphi^T(x)\contbeta(\varphi(x)), & \text{ if } x\in\omegain.
    \end{cases}
\end{equation}
In $\omegain$, we compute
\begin{equation}
    D\hat\contbeta 
    = 
    D\varphi^T D\contbeta\circ\varphi\, D\varphi 
    +
    \contbeta\circ\varphi\cdot D^2\varphi,
\end{equation}
which componentwise reads
\begin{equation}
    \partial_j \hat\contbeta_i 
    = 
    \partial_j((\partial_i\varphi_h) (\contbeta_h\circ\varphi))
    =
    (\partial_i\varphi_h) 
    (\partial_k\contbeta_h) \circ\varphi \partial_j\varphi_k
    +
    (\partial_i\partial_j \varphi_h)(\contbeta_h\circ\varphi).
\end{equation}
Since $\curl\contbeta=0$ on $\omegaout$, we have that $D\contbeta$ is a symmetric matrix. Hence, $D\hat\contbeta$ is symmetric and $\curl\hat\contbeta=0$ in $\omegain\setminus(\{0\}\times\R)$. Let $x \in \partial\omegain \cap \partial\omegaout$ and let $\tau$ be a vector tangential to $\partial\omegain \cap \partial\omegaout$ in $x$. Since $\varphi(x)=x$ on this set, $D\varphi(x)\tau = \tau$. Therefore, the traces satisify
\begin{equation}
     (D\varphi(x)^T\contbeta(\varphi(x)))\cdot\tau 
     =
    \contbeta(x)\cdot D\varphi(x)\tau=\contbeta(x)\cdot\tau.
\end{equation}
This proves that $\curl\hat\contbeta=0$ in $\omega\setminus(\{0\}\times\R)$.
	
It remains to estimate the norm. By \eqref{eqeigenvaldvarphi} {and $\lambda\in(\frac32,2)$}, we obtain
\begin{equation}
	|D\varphi|^p(x)\le \frac{C}{|x'|^{p-1}} |\det D\varphi|{(x)}.
\end{equation}
{From the definition of $\theta$ we obtain} $|x'| = (\lambda-\theta(|x'|))/(\lambda-1)$. We compute
\begin{equation}\label{eqhatbetaomegain}
\begin{split}
    \int_\omegain |\hat\contbeta|^p dx 
    & \le  
    \int_\omegain |D\varphi|^p |\contbeta\circ\varphi|^p dx
    \\ & \le 
    \int_\omegain 
        \frac{C}{|x'|^{p-1}}
        |\det D\varphi(x)|\,  |\contbeta(\varphi(x))|^p 
    dx 
    \\ & = 
    \int_\omegain 
        \frac{C (\lambda-1 )^{p-1}}{|\lambda-\theta(|x'|)|^{p-1}}
        |\det D\varphi(x)|\,  |\contbeta(\varphi(x))|^p 
    dx
    \\ & = 
    \int_\omegaout 
        \frac{C(\lambda-1 )^{p-1}}{(\lambda-|y'|)^{p-1}}  
        |\contbeta(y)|^p \chi_{\{|y'|<\lambda\}}
    dy \,,
\end{split}
\end{equation}
where $\chi_{\{|y'|<\lambda\}}=1$ if $|y'|:=|(y_1,y_2)|<\lambda$ and $0$ otherwise. We set,
\begin{equation}
    f(\lambda)
    :=
    \int_\omegaout 
        \frac{|\contbeta(y)|^p}{(\lambda-|y'|)^{p-1}}  
        \chi_{\{|y'|<\lambda\}} 
    dy
\end{equation}
and average over all possible choices of $\lambda$. Recalling that $p\in[1,2)$,
\begin{equation}
\begin{split}
    \int_{1}^2 f(\lambda)d\lambda
    = &
    \int_\omegaout \int_{|y'|}^2 
        \frac{ |\contbeta(y)|^p}{(\lambda-|y'|)^{p-1}}
    d\lambda \, dy 
    \\ = &
    \frac{1}{2-p}
    \int_\omegaout  
        |\contbeta(y)|^p  (2-|y'|)^{2-p}
    dy
    \\ \le &
    c_* \int_\omegaout |\contbeta|^p dy.
\end{split}
\end{equation}
Therefore, there is $\lambda\in(\frac32,2)$ such that $f(\lambda)\le 2 c_* \|\contbeta\|_{L^p(\omegaout)}^p$. Inserting in \eqref{eqhatbetaomegain} concludes the proof of \eqref{eqlemmaextendbetacyl}.

\emph{Step 2. Proof of \eqref{eqlemmaextendbetacyllinelast}.}
By scaling, we can work with $\rho=1$ and $\ell\ge1$. Let $N:=2\lfloor 2\ell\rfloor$, $z_j:=j(\ell-\frac12)/N$, so that $0=z_0<z_1<\dots <z_N=\ell-\frac12$ and $z_{i+1}-z_i=(\ell-\frac12)/(2\lfloor 2\ell\rfloor)\in [\frac18,\frac14)$. Consider the sets $\omega_i := B_2'\times (z_i, z_i+\frac12)$, $\omegaouti{i} := (B_2'\setminus B_1')\times (z_i, z_i+\frac12)$, $\omegaini{i} := B_1'\times (z_i, z_i+\frac12)$.
 
By Korn's inequality, for every $i$ there is $R_i\in\R^{3\times 3}_\skw$ such that
\begin{equation}\label{eqkornonecyl}
    \|\contbeta-R_i\|_{L^p(\omegaouti{i})}
    \le 
    C \|\contbeta+\contbeta^T\|_{L^p(\omegaouti{i})}.
\end{equation}
Since $z_{i+1}\le z_i+\frac14$, we obtain $\calL^3(\omegaouti{i} \cap \omegaouti{i+1})\ge \frac34\pi$ and, therefore,
\begin{equation}\label{eqkornnextonecyl}
    |R_i-R_{i+1}|
    \le 
    C \|\contbeta+\contbeta^T\|_{L^p(\omegaouti{i}\cup\omegaouti{i+1})}.
\end{equation}
We define the affine isometries $u_i:\R^3\to\R^3$ as
\begin{equation}
    u_i(x):=R_ix+d_i ,
\end{equation}
where $d_0=0$ and $d_{i+1}:=d_i+R_ie_3z_i-R_{i+1}e_3z_i$, so that
\begin{equation}\label{equiukw1p}
    \|u_i-u_{i+1}\|_{W^{1,p}(\omega_i\cup\omega_{i+1})}
    \le 
    c |R_i-R_{i+1}|.
\end{equation}
We fix a partition of unity $\theta_0, \dots, \theta_N\in C^\infty(\R;[0,1])$, with $\sum_i \theta_i(t)=1$ for all $t\in\R$ and $\supp\theta_0 \subseteq(-\infty,1/2)$, $\supp\theta_N\subseteq(z_N,\infty)$, $\supp\theta_i\subseteq(z_i, z_i+\frac12)$ for $1\le i\le N-1$, and $|\theta_i'|(t)\le C$ for all $i\in\{0, \dots, N\}$ and all $t\in\R$. We define $u:\R^3\to\R^{3}$ as
\begin{equation}
    u(x):= \sum_{i=0}^N \theta_i(x_3)u_i(x)
\end{equation}
and observe that
\begin{equation}
    Du(x)-Du_j(x)
    = 
    \sum_{i=0}^N  
    \theta_i(x_3)(Du_i-Du_j)(x)+\theta_i'(x_3)(u_i-u_j)(x)\otimes e_3,
\end{equation}
where, for $x\in\omega_j$, only the terms with $j-4\le i\le j+4$ are relevant. With \eqref{equiukw1p}, \eqref{eqkornnextonecyl} and $Du_j+Du_j^T=0$, we obtain
\begin{equation}\label{eqdudutlpbeta}
    \|Du+Du^T\|_{L^p(\omega)}
    \le 
    c \|\contbeta+\contbeta^T\|_{L^p(\omegaout)},
\end{equation}
and similarly with \eqref{eqkornonecyl}
\begin{equation}\label{eqdudutlpbetaout}
    \|Du-\contbeta\|_{L^p(\omegaout)}
    \le 
    c \|\contbeta+\contbeta^T\|_{L^p(\omegaout)}.
\end{equation}

We apply Step~1 to the function $\contbeta-Du$ and obtain a function $\hat \contbeta\in L^p(\omega;\R^{3\times 3})$ with $\curl\hat\contbeta$ concentrated on $\{0\}\times (0,\ell)$ and 
\begin{equation}
    \|\hat\contbeta\|_{L^p(\omegain)}
    \le 
    c \|\contbeta-Du\|_{L^p(\omegaout)}.
\end{equation}
We set $\tilde\contbeta:=Du+\hat\contbeta$, so that $\curl \tilde \contbeta = \curl\hat\contbeta$ in $\omega$ and $\tilde\contbeta=\contbeta$ in $\omegaout$, and use \eqref{eqdudutlpbeta} and \eqref{eqdudutlpbetaout} to obtain
\begin{equation}
\begin{split}
    \|\tilde\contbeta+\tilde\contbeta^T\|_{L^p(\omegain)}
    \le &
    \|Du+Du^T\|_{L^p(\omegain)} 
    +
    2 \|\hat\contbeta\|_{L^p(\omegain)}
    \\ \le & 
    c\|\contbeta+\contbeta^T\|_{L^p(\omegaout)}
    + 
    c \|\contbeta-Du\|_{L^p(\omegaout)}
    \le 
    c\|\contbeta+\contbeta^T\|_{L^p(\omegaout)}.
\end{split}
\end{equation}

\emph{Step 3. Structure of $\curl\beta$, case $\theta=0$.} The extension provided above, together with  Lemma \ref{lemmacurll1h1}, gives that there exists $\theta\in\R^d$ such that $\Curl\tilde\beta=\theta\otimes e_3\calH^1\LL (\{0\}\times(0,\ell))$. In the case $\theta=0$, this implies that $\tilde\beta$ is a gradient in $\omega$. Therefore, there a function $u\in W^{1,p}(\omegaout;\R^d) $ such that $Du=\beta$. By standard extension argument in Sobolev spaces, an extension $\tilde u$ of $u$ in $W^{1,p}(\omega;\R^d)$ is obtained, and a redefinition $\tilde\beta:=D\tilde u$. Hence, Step 2 is unchanged for every $p\in [1,\infty)$.
\end{proof}

\subsection{Extension in balls}

\begin{proof}[Proof of Lemma \ref{lemmaextendbetaball}]
\emph{Step 1. We show that for every $p\in[1,2]$, $d\in\N$, if $\contbeta\in L^p(\omegaout;\R^{d\times 3})$	with $\curl\contbeta=0$ in $\omegaout$ is given, then, there is $\tilde \contbeta\in L^p(\omega;\R^{d\times 3})$ such that $\tilde\contbeta=\contbeta$ on $\omegaout$, $\curl\tilde\contbeta=0$ on $\omega:=\omegain\cup\omegaout$, and
\begin{equation}\label{eqlemmaextendbetaball1}
    \|\tilde\contbeta\|_{L^p(\omegain)}
    \le 
    2^{1/p} \|\contbeta\|_{L^p(\omegaout)}.
\end{equation}}
The construction is similar that in Step 1 of the proof of Lemma \ref{lemmaextendbetacyl}. Also in this case, we work in the case $\rho=1$ and $d=1$. 	We set $\theta(t):=2-t$ and consider $\varphi:B_1\setminus\{0\}\to B_{2}\setminus B_1$ defined as
\begin{equation}
	\varphi(x):=\frac{x}{|x|}\theta\big(|x|\big).
\end{equation}
We observe that $\varphi((0,1]v_i)=[1, 2)v_i$, so that $\varphi(\omegain\setminus\{0\}) \subseteq\omegaout$. We compute
\begin{equation}
    D\varphi(x)
    =
    \Big(\Id-\frac{x\otimes x}{|x|^2}\Big) \frac{1}{|x|} \theta\big(|x|\big) 
    + 
    \frac{x\otimes x}{|x|^2} \theta'\big(|x|\big).
\end{equation}
In particular, $D\varphi(x)$ is symmetric and has eigenvalues 
\begin{equation}\label{eqeigenvaldvarphiba}
    \frac{2}{|x|}-1,\hskip2mm  \frac{2}{|x|}-1,\hskip2mm -1,
\end{equation}
so that $|\det D\phi|(x)=(2/|x|-1)^2\ge 1$ and 
\begin{equation}\label{eqdphi2detdphi}
    |D\varphi|^2(x)
    =
    2|\det D\varphi|(x)+1\le 3|\det D\varphi|(x) \hskip5mm 
    \text{ for all } x\in \omegain.
\end{equation}
We define
\begin{equation}
    \tilde\contbeta(x)
    :=
    \begin{cases}
        \contbeta(x), & \text{ if } x\in \omegaout,\\
        D\varphi^T(x)\contbeta(\varphi(x)), & \text{ if } x \in \omegain.
    \end{cases}
\end{equation}
By the same computation as in Lemma \ref{lemmaextendbetacyl}, $\curl\tilde\contbeta=0$ in $B_{2}\setminus \cup_i [0,2)v_i$.

We compute as in \eqref{eqhatbetaomegain}, and using \eqref{eqdphi2detdphi}, $p\le 2$ and $|D\varphi|(x)\ge 1$ pointwise,  we have
\begin{equation}
\begin{split}
    \int_\omegain |\tilde\contbeta|^p dx 
    & \le  
    \int_\omegain |D\varphi|^p|\contbeta\circ\varphi|^p dx
    \\ & \le 
    \int_\omegain 3|\det D\varphi(x)|\,  |\contbeta(\varphi(x))|^p dx
    = 
    \int_\omegaout 3 |\contbeta(y)|^p dy ,
\end{split}
\end{equation}
which concludes the proof of \eqref{eqlemmaextendbetaball1}.

\emph{Step 2. Proof of \eqref{eqlemmaextendbetaballsym}.}
By scaling, it suffices to consider the case $\rho=1$. As $p\le 3/2$, by the rigidity estimate of Proposition \ref{proprigiditycras}  there is $S\in\R^{3\times 3}_\skw$ such that
\begin{equation*}
    \|\contbeta-S\|_{L^{p}(\omegaout)}
    \le c 
    \|\contbeta+\contbeta^T\|_{L^{p}(\omegaout)} 
    + 
    c |\curl \contbeta|(B_2\setminus B_1).
\end{equation*}
By Step 1, we can obtain an extension $\tilde\contbeta$ of $\contbeta-S$ such that
\begin{equation*}
    \|\tilde\contbeta\|_{L^p(\omegain)}
    \le
    c\|\contbeta-S\|_{L^{p}(\omegaout)}.
\end{equation*}
The map $\hat\contbeta:=S+\tilde \contbeta$ satisfies $|\hat\contbeta + \hat\contbeta^T| = |\tilde\contbeta+\tilde\contbeta^T|$ pointwise and, therefore,   \eqref{eqlemmaextendbetaballsym}.

The proof of \eqref{eqlemmaextendbetaballsym0} is analogous, using Korn's inequality instead of Proposition~\ref{proprigiditycras}.
  
\emph{Step 3. Structure of $\curl\hat\beta$.} Since $\hat\beta\in L^1(B_2;\R^{3\times 3})$ and $\curl\hat\beta=0$ on $B_2\setminus\cup_i [0,2)v_i$, by Lemma~\ref{lemmacurll1h1} in $(\mathcal D(B_2\setminus \{0\};\R^{3\times 3}))'$ we obtain $\curl \hat \beta|_{B_2\setminus\{0\}}=\sum_i \theta_i\otimes v_i\calH^1\LL(0,2)v_i$. It remains to deal with the origin. Fix a test function $\eta\in C^\infty_c(B_2;\R^{3\times 3})$ and, for any $r\in (0,1/2)$, select $h_r\in C^\infty_c(B_{2r})$ such that $h_r=1$ on $B_r$ and $|Dh_r|\le 2/r$.  We write
\begin{equation}\label{eqhatbetaeta}
\begin{split}
    \int_{B_2} \hat\beta\curl\eta \,dx 
    = &
    \int_{B_2\setminus B_r} \hat\beta\curl((1-h_r)\eta) dx  
    +
    \int_{B_{2r}} \hat\beta\curl(h_r\eta) dx  
    \\ = &
    -\sum_i \theta_i\otimes v_i  
    \int_{(0,2)v_i} (1-h_r)\eta \, d\calH^1  
    +
    \int_{B_{2r}} \hat\beta\curl(h_r\eta) dx 
\end{split}
\end{equation}
and estimate the last term using
\begin{equation}
    \int_{B_{2r}} |\hat\beta| \, |\curl(h_r\eta)| dx  
    \le 
    \left(\frac{c}{r}\|\eta\|_{L^\infty}
    +
    {c}\|D\eta\|_{L^\infty}\right)
    \int_{B_{2r}} |\hat\beta| dx  
    \le 
    C_\eta\|\hat\beta\|_{L^{3/2}(B_{2r})},
\end{equation}
where $C_\eta$ depends on $\eta$ but not on $r$. Taking the limit $r\to0$ in \eqref{eqhatbetaeta} leads to
\begin{equation}
\begin{split}
    \int_{B_2} \hat\beta\curl\eta \,dx 
    = &
    -\sum_i \theta_i\otimes v_i  \int_{(0,2)v_i} \eta \, d\calH^1  
\end{split}
\end{equation}
for every $\eta\in C^\infty_c(B_2;\R^{3\times 3})$, as desired.
\end{proof}

\section{Notation}
\parindent0pt

\emph{Discrete structure}

$\calL$ is a lattice, see Definition~\ref{deflattice};

$\calN\subseteq\calL$ is a set of bonds, see Definition~\ref{deflattice};

$\calC\subseteq\calL$ is a cluster, see Definition~\ref{deflattice};

$a_1,\dots, a_n$ are a basis for $\calL$, see Definition~\ref{def-tetrahedra} and Definition~\ref{deflattice};

$T\subset\R^n$ is a lattice simplex, see Definition \ref{def-tetrahedra};

$T_*:=\cup_i T_i$ is a unit cell of $\calL$, see Definition \ref{def-tetrahedra};

$T_i$ are the elementary tetrahedra defined in Definition \ref{def-tetrahedra};

$\vertici(T)$ is the set of vertices of the simplex $T$, see Definition~\ref{defdepsu};

$\edges(T)$ is the set of edges of the simplex $T$, see Definition~\ref{defdepsu};

$\calC_\calN$ set of bonds inside a cluster; see Definition~\ref{defcluster};

$D_\calC$ set of deformations on $\calC_\calN$; see Definition~\ref{defcluster};

$\bonds_\eps^\Omega$ is the set of bonds in $\Omega$, see Definition~\ref{defdepsu};

$\tetr_\eps^\Omega$ is the set of lattice simplices in $\Omega$, see Definition~\ref{defdepsu};

$ \clust_\eps^\Omega $ are the centers of the clusters contained in $\Omega$, see~\eqref{eqdefclustepsomega};

$(\burgersbi{l},m_l)$, $l=1,\dots, N_s$, are slip systems, see Definition~\ref{defslipsys};

complete set of slip systems, see Definition~\ref{defbetapl2};

$\calL^*$ denotes the dual lattice, see Definition~\ref{defslipsys};

$\calB\subseteq\calL$ is the lattice of possible  Burgers vectors, see Definition~\ref{defslipsys}.

\medskip\emph{Elastic kinematics}

$u:\eps\calL\cap\Omega\to\R^n$ is the displacement field, see Definition~\ref{eqdefduxi} and Definition~\ref{defdepsu};

$du:\bonds_1^\Omega\to\R^n$ is a discrete deformation gradient, see Definition~\ref{eqdefduxi};

$d_\eps u:\bonds_\eps^\Omega\to\R^n$ is a scaled discrete displacement gradient, see Definition~\ref{defdepsu} and Remark~\ref{remdepsucont}; 

$\xi$ admissible deformation, see Definition~\ref{eqdefduxi}, \eqref{eqbetacompatible1} and \eqref{eqbetacompatible};

$\xi^F$ admissible deformation corresponding to $F\in\R^{n\times n}$, see Definition~\ref{defcluster};

$L\discbeta:\Omega\to\R^{n\times n}$ interpolation of discrete deformation,
see Definition~\ref{defLbeta};

$\conteta$ is the  continuum curl-free strain.

\medskip\emph{Plastic kinematics}

$\discxi$ are the discrete slip parameters;

$\discbetapl:\bonds_\eps^\Omega\to\eps\calB$ is a discrete plastic strain, 
see Definition~\ref{defbetapl};

$P=(x_0, x_1, \dots, x_K)$ $\eps$-discrete path, see Definition~\ref{definitioncurlfree};

$P$ closed and/or elementary path, see Definition~\ref{definitioncurlfree};

$\oplus$ concatenation of paths, see Definition~\ref{definitioncurlfree};

$\circu(\xi,P)$  over $P$, see Definition~\ref{definitioncurlfree};

$\discbeta:\bonds_\eps^\omega\to\R^n$ exact, see Definition~\ref{definitioncurlfree};

$(\mu,\beta)$ $\rho$-compatible, see Definition~\ref{defbetamucompatible};

$\Ce(\discbeta,\Omega)$ or $\Ce(\discbeta)$  is the core region, see Definition~\ref{definitioncore}.

\medskip
\emph{Dilute dislocations}

$\Ade(\Omega)$ and $\Ade(\Omega,k_\eps,\alpha_\eps,m)$ see Definition \ref{def-Ade} and \eqref{eq-Ade-short};

$P(\Omega, k,\alpha)$ set of $(k,\alpha)$ dilute curves in $\Omega$, see Definition \ref{defdiluteness-curve};

$\calM^1(\Omega)$ and $\calM_{\calL'}^1(\Omega)$ are divergence-free matrix-valued measures and their $\calL'$-valued, counterpart, see text before Definition~\ref{defM1Lcont}; 

$\calM_{\calL'}^1(\Omega,k,\alpha)$ dilute dislocation measures, see Definition~\ref{defM1Lcont}.

\medskip
\emph{Energies}

$E_\calC$ energy on a cluster, see Definition~\ref{defcluster};

$E^0_\calC$ canonical lower bound for $E_\calC$,  see Definition~\ref{defcluster};

$\C$ matrix of elastic constants, see Definition~\ref{defcluster} and~\eqref{eqpropC};

$E_\eps[\discbeta_\eps, \Omega]$ total energy in $\Omega$, see~\eqref{eqdeftaux};
 
$\tau^x_\eps$ translation operator, see~\eqref{eqdeftaux}.

\medskip\emph{Interpolation and extension}
 
$I_\eps u$ is  the piecewise affine interpolation, see Definition~\ref{definterpolation};

$J_\eps v$ is  the piecewise constant interpolation, see Definition~\ref{definterpolation};

$\psi_\delta$ and $\psi_\delta^\eps$, continuous and discrete mollifiers, see Equation~\ref{eqpsideltamoll} and Equation~\ref{eqdefpsideltaeps};

$E\mu$ and $\beta_\mu$, extension of a dislocation measure and corresponding strain field, see Equation~\eqref{eqdefbetamu}.

\medskip\emph{Sets}

$\Omega_r:=\{x\in\Omega: \dist(x,\partial\Omega)>r\}$;

$\Omega_r(\gamma):=\{x\in\Omega: \dist(x,\gamma\cup \partial\Omega)>r\}$;

$B_r(E):=(E)_r:=\{x: \dist(x,E)<r\}$.

\medskip\emph{Constants}

$\ccorerad$ factor for the reduction to elementary paths, Lemma~\ref{lemmaccorerad}, $k_*\ge d_\calC$;

$c_{ext}$ is the constant in the extension Proposition \ref{propextendcylincircles};

$d_{T_*}$ is twice the diameter of $T_*$, see \eqref{eq-def-dT}, $d_{T_*}\le d_\calC$;

$d_\calC$ is twice the diameter of $\calC$, see \eqref{eq-def-dC}, $d_{T_*}\le d_\calC\le k_*$;

$m$ is the radius of the core in Definition~\ref{def-Ade}, $m\ge k_*$.

\medskip\emph{Other}

$Df$ denotes the weak gradient of a function $f\in W^{1,1}_\loc$ and the distributional derivative of a function $f\in BV_\loc$.

\section*{Acknowledgements}

This work was partially funded by the Deutsche Forschungsgemeinschaft (DFG, German Research Foundation) {\sl via} project 211504053 - SFB 1060, and project 390685813, GZ 2047/1.

\addcontentsline{toc}{section}{References}

\bibliographystyle{alpha-noname}
\bibliography{cogaor}

\end{document}